\documentclass[11pt,reqno]{amsart}

 \usepackage{amsmath,amsthm,amsfonts}
\usepackage{latexsym,mathabx,txfonts}
\usepackage{amssymb}
\usepackage{slashed}
\usepackage{tikz}
\usepackage{enumitem}
\usepackage{hyperref}

\allowdisplaybreaks

\newcommand{\lam}{\lambda}

\newcommand{\wt}[1]{\widetilde {#1}}

\newcommand{\R}{\mathbb{R}}

\newcommand{\C}{\mathbb{C}}

\renewcommand{\bar}{\overline}

\newcommand{\bch}{\boldsymbol{\mathcal{H}}}

\newcommand{\im}{\operatorname{Im}}

\numberwithin{equation}{section}

\newtheorem{thm}{Theorem}[section]

\newtheorem{lem}[thm]{Lemma}
\newtheorem{prop}[thm]{Proposition}

\theoremstyle{remark}
\newtheorem{rem}{Remark}[section]

\usepackage{enumitem}

\newcommand{\bfq}{{\bf Q}}
\newcommand{\bfcq}{\boldsymbol{\mathcal{Q}}}
\newcommand{\inn}{\text{inn}}
\newcommand{\out}{\text{out}}
\newcommand{\co}{\text{cor}}

\newcommand{\Del}[1]{}

\definecolor{bluegreen}{rgb}{0.0, 0.3, 0.9}

\author[J. Krieger]{Joachim Krieger}
\address{EPFL SB MATH PDE, B\^atiment MA, Station 8, CH-1015 Lausanne}
\email{joachim.krieger@epfl.ch}

\author[J.M. Palacios]{Jos\'e M. Palacios}
\address{EPFL SB MATH PDE, B\^atiment MA, Station 8, CH-1015 Lausanne}
\email{jose.palaciosarmesto@epfl.ch}

\begin{document}

%%%%%%%%%%%%%%%%%%%%%%%%%%%%%%%%%%%%%%%%%%%


%%%%%%%%%%%%%%%%%%%%%%%%%%%%%%%%%%%%%%%%%%%

\begin{abstract}
    We show that the energy critical Wave Maps equation from $\R^{2+1}$ into $\mathbb{S}^2$, restricted to the $k=2$ co-rotational setting, admits arbitrarily large numbers of concentrating concentric $n$ bubble profiles. For any $n\in\mathbb{N}$, we construct an $n$-bubble solution concentrating at scales  $\lambda_1(t)\gg \lambda_2(t)\gg \ldots\gg \lambda_n(t)$, where $\lambda_n(t)=t^{-1}\vert \log t\vert^\beta$, and $\lambda_j(t)\gtrsim \exp( \int_t^{t_0} \lambda_{j+1}(s)ds)$, for any $j<n$. Here $\beta>\tfrac32$ is a parameter that can be chosen arbitrarily. This shows that, as far as finite time blow-up case is concerned, the entirety of cases postulated in the soliton resolution theorem indeed occur, provided the concentric collapsing bubbles have alternating signs.
\end{abstract}

%%%%%%%%%%%%%%%%%%%%%%%%%%%%%%%%%%%%%%%%%%%

\title[Long finite time bubble trees for two co-rotational wave maps]{Long finite time bubble trees for two co-rotational wave maps}

\setcounter{tocdepth}{1}
\maketitle
\tableofcontents

% [To erase later] Notation: \begin{align*}
% \bfq_{n-1}(t,r) &:= \sum_{j=2}^n (-1)^{j+1} Q\big(\lambda_j(t)\,r\big),
% \\ \bfcq_{n-1} (t,r)  & :=  \mathbf{Q}_{n-1}(t,r)  + w_{n-1}(t,r) , 
% \\ \lambda_1 &= \text{inner most}, \quad \lambda_n=\text{outer most},
% \\ h_j&:=h_j^{\out}+h_j^{\inn}
% \\ e_{j,1} &= E_{2,\co} + \sum_{k=1}^j e_{k,\co}
% \\ \mathcal{L} & = \text{operator in }\R^2, \quad \mathcal{H}= r^{1/2}\mathcal{L}r^{-1/2}=\text{operator on the half-line}
% \\ \tau_1 & := \int_t^{t_0} \lambda_1(s)ds, \qquad \tau_2 := \int_t^{t_0} \lambda_2(s)ds
% \end{align*} 

\section{Introduction}

This paper is dedicated to constructing novel finite time blow-up solutions for the $k$ co-rotational wave maps from $\mathbb{R}^{2+1}$ into $\mathbb{S}^2$. We recall that using the polar angle $u$ on $\mathbb{S}^2$, this problem can be cast as the following scalar wave equation
\begin{equation}\label{eq:Mainkcorotational}
-u_{tt} + u_{rr} + \frac{1}{r}u_r = k^2\cdot \frac{\sin(2u)}{2r^2}.
\end{equation}
We shall focus on the case $k = 2$ from now on, although it is likely that aspects of the constructions below can be adapted to the case $k\geq 3$.
Our goal shall be to show that {\it{for any $n\geq 1$}}, there exist finite time bubble trees consisting of $n$ concentric bubbles, concentrating at the space time origin $(t, r) = (0,0)$. This shows that {\it{as far as the finite time blow up case is concerned}}, the entirety of cases postulated in the soliton resolution theorem indeed exist, provided the concentrating bubbles exhibit alternating signs. 
\\

To formulate the main theorem, we recall that up to sign change and scaling transforms, the unique static solution of \eqref{eq:Mainkcorotational} is given by 
\[
Q(r) = 2\arctan(r^2).
\]
The following theorem is our main result and establishes the existence of such $n$ concentric bubble solutions that blow up simultaneously.
\begin{thm}\label{thm:Main} Let $n\geq 2$ and assume $\beta>\frac32$. Then there exist $t_0 = t_0(\beta, n)>0$ and a finite energy solution of \eqref{eq:Mainkcorotational} on $(0, t_0]\times [0,\infty)$ of the form 
\[
u(t, r) = \sum_{j=1}^n (-1)^{j+1}Q\big(\lambda_j(t)r\big) + \epsilon(t, r),
\]
where 
\[
\lim_{t\rightarrow 0}\big\|\nabla_{t,r}\epsilon\big\|_{L^2_{r\,dr}(r\leq t)} = 0,
\]
and the scaling parameters $\lambda_j(t)$ are given by 
\begin{align*}
&\lambda_n(t) = t^{-1}|\log t|^{\beta},\\
&\lambda_{j-1}\sim e^{\mu_j(t)\int_t^{t_0}\lambda_{j}(s)\,ds},\,j\leq n,
\end{align*}
where the functions $\mu_j(t)$ satisfy 
\[
\lim_{t\rightarrow 0}\mu_j(t) = 1.
\]
In particular there is $c = c(\beta, n)>0$ such that
\[
\lambda_j(t)\geq e^{e^{\ldots e^{c|\log t|^{\beta+1}}}}, \qquad j<n,
\]
where the expression at the end is an $n-j$ times iterated exponential tower. 
\end{thm}

\begin{rem} Very recently, a result by Hwang and Kim establishing the existence of {\it{infinite time bubble towers of arbitrary length}} appeared. The solutions in this paper are of threshold type, i.e., they are {\it{pure multi bubble solutions}} without asymptotic radiation, which is fundamentally different than the solutions constructed in our paper (see \cite{HK}). 
\end{rem}

\smallskip
\subsection{Previous results}
There is an extensive literature devoted to the study of small
solutions to the general (non-co-rotational) wave maps equation. Seminal contributions in this direction include \cite{KlM,KS97,Tao01,Tataru01}. We emphasize that the first two of these works show that, for energy-critical wave maps, it is natural to consider initial data whose regularity is only slightly above the energy class, even in the non-equivariant setting. The study of the dynamical behavior of large solutions to \eqref{eq:Mainkcorotational}, such as the solutions treated here, has a fairly long history as well. Foundational contributions were obtained in \cite{CTZduke,CTZcpam,STZ92,STZ94}; see also \cite{ShSt00}. The upshot of these works is the decay of energy at the self-similar scale, which in particular rules out self-similar blow-up. Moreover, as shown by Struwe \cite{Struwe03}, this decay mechanism gives rise to bubbling: if $u$ is a solution of \eqref{eq:Mainkcorotational} that blows up in finite time $T$, then there exist sequences $t_n\to T$ and $0 < \lambda_n \ll T-t_n$ such that \[
\Big(u(t_n,\lambda_n\cdot), \lambda_n \partial_t u(t_n,\lambda_n\cdot)\Big) \to m\pi \pm Q,
\]
where the convergence is understood in the topology induced by the energy, locally on bounded sets. Earlier numerical simulations by Bizo\'n, Chmaj, and Tabor provided evidence for finite-time blow-up via energy concentration in the equivariant setting \cite{Bizon}.

This bubbling phenomenon also implies that any solution with energy strictly below that of $Q$ cannot blow up. C\^ote et al. proved that such solutions in fact exhibit purely radiative behavior (see \cite{CKLS15}). For energies above this threshold, finite-time blow-up may occur (see \cite{KST1,RaRod,RoSt10}). Significant advances toward the soliton resolution conjecture were obtained in \cite{CKLS2}. Sequential soliton resolution—that is, convergence along a sequence of times to a superposition of solitons—was established by C\^ote \cite{Cote15} for $k\in\{1,2\}$, and by Jia and Kenig \cite{JK} for $k\geq 3$. Related results without assuming equivariant symmetry, though with a less detailed description of the radiation, were obtained in \cite{Gr}. In \cite{JL4}, a continuous-in-time soliton resolution was proved at the minimal energy level that still allows for the existence of a two-bubble configuration. As a relatively straightforward consequence, continuous-in-time resolution was also established in \cite{JL1} under the assumption that the solution contains at most two bubbles. More recently, a full soliton resolution result in the equivariant setting (for all equivariance classes $k$) was established in \cite{JL3} using a time decomposition based on ``collision intervals" (to accurately distinguish between ``interior" bubbles that come into collision) and ``no-return" lemmas. In the co-rotational case $
k=1$, an independent resolution theorem was obtained earlier by Duyckaerts, Kenig, Martel, and Merle, whose proof relies on a Liouville-type rigidity theorem (of independent interest) to rule out nontrivial non-radiative dynamics, which in turn is based on the energy channels method (see \cite{DKMM}).

While these soliton resolution results give an a priori description of general solutions of \eqref{eq:Mainkcorotational},
in particular those blowing up in finite time, they do not by themselves guarantee the existence of blow-up solutions with more than one collapsing bubbling profile. For example, we observe that the existence of such solutions for the related harmonic map heat flow \[
\partial_t u -\partial_r^2 u - \dfrac{1}{r}\partial_r u + \dfrac{k^2}{2r^2} \sin(2u)=0,
\]
has been disproved for the case $k=1$ in \cite{VanderHout}, and for $k\geq 2$ in \cite{Samu}.

The first constructions of concentric two-bubble solutions for \eqref{eq:Mainkcorotational} with $k\geq 3$, as well as for the equivariant Yang–Mills problem closely related to \eqref{eq:keq2corotational}, were obtained in \cite{J1,J2}. These are pure two-bubble solutions that exist globally in one time direction, in the sense that, as time goes to infinity, they decouple into two bubbles, one remaining static while the other is dynamically rescaled (its scale parameter tends to 0 at large times). In particular, the total energy of such solutions must be exactly twice the energy of a single static bubble. For $k\geq4$, uniqueness of solutions of this type was proved in \cite{JL2}, and \cite{JL4} described their asymptotic behavior in both time directions, proving, in particular, the inelastic nature of the collision of two bubbles. The case $k=1$ was treated in \cite{Rodrig}, where, among other results, it was shown that pure two-bubble solutions do not exist in this case.

\smallskip
\subsection{Notation} We adopt the following notation, most of which are standard: First, we use $\langle x\rangle$ for $\sqrt{1+\vert x\vert^2}$. An inequality of the form $F \lesssim G$ for nonnegative quantities $F$ and $G$ means the existence of a universal positive constant $c>0$ independent of $F$ and $G$ such that \[
F \leq c\cdot G.
\]
The notation $F\sim G$ means that $F\lesssim G$ and $G\lesssim F$. 

We denote by $a+$ (respectively $a-$) a number $b>a$ (respectively $b<a$) that can be chosen arbitraily close to $a$. In this regard, we use the notation \[
F \lesssim \tau^{-p+}
\]
to imply that for any $\varepsilon>0$, there is a constant $C_\varepsilon$ such that \[
F \leq C_\varepsilon \cdot \tau^{-p+\varepsilon}.
\]
Such inequalities will always occur in a setting where $\tau \geq \tau_*$ for some sufficiently
large $\tau_*\gg 1$.

For positive functions $\mu(t)$ and $\lambda(t)$ defined on $(0,t_0]$ we shall write $\mu(t)\ll \lambda(t)$, provided \[
\lim_{t\to 0+} \dfrac{\mu(t)}{\lambda(t)}=0.
\]

We will often use a dot symbol ‘·’ to distinguish the bounds on various quantities; this is meant to help the reader navigate some of the bounds.

\bigskip
\section{Starting point}

The strategy to construct solutions of the kind asserted in the theorem shall be by means of an inductive procedure, commencing with a single bubble solution corresponding to $n = 1$, as constructed in \cite{JenKri}. We shall then proceed by 'adding in' higher frequency bubbles, which evolve on a low frequency background. The starting point is provided by the following 
\begin{thm}\label{thm:outersolnYManalogue} Let $\lambda(t) = t^{-1}\cdot\big({-}\log t\big)^{\beta}$, $\beta\geq \frac32$. Then there is $t_0 = t_0(\beta)>0$ such that the equation 
\begin{equation}\label{eq:keq2corotational}
\big({-}\partial_{tt} + \partial_{rr} + \frac{1}{r}\partial_r\big)u = 2\frac{\sin(2u)}{r^2}
\end{equation}
admits a finite time blow up solution on $(0, t_0]\times \mathbb{R}_+$ of the form 
\[
\tilde{Q}(t, r) = Q(\lambda(t)r) + v(t, r),\qquad Q(r) = 2\arctan(r^2),
\]
where we have the bounds 
\[
\sum_{k=0}^K\|\nabla_{t,r}S^k v\|_{L^2}\lesssim_K |\log t|^{-1}, \qquad S = t\partial_t + r\partial_r,
\]
as well as the bound $\big|v(t, r)\big|\lesssim |\log t|^{-1}$. 
\end{thm}
This is proved in \cite{JenKri}. 

\bigskip
\section{Outline of the inductive procedure}

We shall construct the solutions described in Theorem~\ref{thm:Main} by means of an inductive procedure, where the case $n = 1$ corresponds to the preceding theorem.

Fixing $n\geq 2$, and assuming a solution for $n$ replaced by $n-1$ to have been constructed on $(0, t_0(\beta, n-1)]\times \mathbb{R}_+$, we shall by re-labelling denote the 'bubble part' of the $(n-1)$-bubble solution, which we shall also denote as {\it{outer solution}}, by 
\begin{equation}\label{eq:outerbubblepart}
\bfq_{n-1}(t,r) := \sum_{j=2}^n (-1)^{j} Q\big(\lambda_j(t)\,r\big),
\end{equation}
and the full outer solution by 
\begin{equation}\label{eq:outerbubblesoln}
\bfcq_{n-1} (t,r)   :=  \bfq_{n-1}(t,r)  + w_{n-1}(t,r).
\end{equation}
Our task shall be to add an innermost (high-frequency) bubble 
\[
Q(\lambda_1(t))
\]
and a further correction term $v(t,r)$, such that the expression 
\[
Q(\lambda_1(t)r) - \bfcq_{n-1} (t,r) + v(t, r)
\]
solves \eqref{eq:Mainkcorotational} on $(0, t_0(\beta, n)]\times \mathbb{R}_+$. Here $t_0(\beta, n)\leq t_0(\beta, n-1)$, and per construction we shall have 
\[
\lambda_1(t)\gg \lambda_2(t)\gg \ldots\gg \lambda_n(t) = \frac{|\log t|^{\beta}}{t},\,t\in (0, t_0(\beta, n)].
\]
In fact, the growth from $\lambda_j(t)$ to $\lambda_{j-1}(t)$ shall be extremely fast (essentially exponential), which will play an important role when dealing with the propagator associated with the linearisation around the bulk part $\bfcq_{n-1}$. 
\\

In order to understand how we shall choose $\lambda_1$ to leading order, having chosen $\lambda_j,\,n\geq j\geq 2$ already, we consider now the equation which $v$ will have to satisfy. Letting 
\[
Q(\lambda_1(t)r) =: Q_1,
\]
we deduce the following equation:
\begin{equation}\label{eq:veqn}\begin{split}
&-v_{tt} + v_{rr} + \frac{1}{r}v_r - \frac{4\cos\big(2Q_1 - 2\bfcq_{n-1}\big)}{r^2}v = \sum_{j=1}^3 E_j,\\
&E_1 = \frac{2\cos\big(2Q_1- 2\bfcq_{n-1}\big)}{r^2}\cdot \big[\sin\big(2v\big) - 2v\big],\\
  &E_2 = \frac{2}{r^2}\big[\sin\big(2Q_1- 2\bfcq_{n-1}\big) - \sin\big(2Q_1\big) + \sin\big(2\bfcq_{n-1}\big)\big] +Q_{1,tt},\\
  &E_3 =  \frac{2\sin\big(2Q_1- 2\bfcq_{n-1}\big)}{r^2}\cdot \big[\cos\big(2v\big) - 1\big]. 
\end{split}\end{equation}
We point out that, among the above errors, the largest contribution comes from $E_2$, which in particular is of order zero in $v$. We will first deal with $E_2$ in a preliminary separate step (see Step 0 in Section \ref{sec:construction_acc}).

Our strategy shall be to first find an approximate solution for the preceding equation, up to an error which decays very rapidly with respect to the 'innermost' time variable essentially given by 
\begin{equation}\label{eq:tauasympto}
\tau \sim \lambda_1(t)\lambda_2^{-1}(t). 
\end{equation}
Note that this time variable blows up as we approach the singular time $t = 0$. 
\\
The main challenge for solving \eqref{eq:veqn} consists in solving the linear inhomogeneous problem 
\begin{equation}\label{eq:veqnlininhom}
-v_{tt} + v_{rr} + \frac{1}{r}v_r - \frac{4\cos\big(2Q_1 - 2\bfcq_{n-1}\big)}{r^2}v = f. 
\end{equation}
As in \cite{JenKri}, the idea shall be to find an approximate solution by alternating the solution of a wave equation corresponding to the ``outer potential term"
\[
 -\frac{4\cos\big(2\bfcq_{n-1}\big)}{r^2}v
\]
with an elliptic equation involving the ``inner potential term" 
\[
 - \frac{4\cos\big(2Q_1\big)}{r^2}v.
 \]
 Precisely, we shall set 
 \[
 v = v^{\inn} + v^{\out},
 \]
where we start by solving for $v^{\out}$ by means of the wave equation 
\begin{equation}\label{eq:vouter}
-v^{\out}_{tt} + v^{\out}_{rr} + \frac{1}{r}v^{\out}_r - \frac{4\cos\big(2\bfcq_{n-1}\big)}{r^2}v^{\out} = f.
\end{equation}
The main difference of this equation for the case $n>2$ as compared to the case $n = 2$ treated in \cite{JenKri} comes from the more complex structure of the potential \[
- \frac{4\cos\big(2\bfcq_{n-1}\big)}{r^2}v^{\out},
\]
due to the presence of at least two bubbles in the expression $\bfcq_{n-1}$. Our way of dealing with this rests on two features of our setup, namely
\begin{itemize}
\item The fact that approximating $\bfcq_{n-1}$ by $-Q(\lambda_2(t)r)$ results an error which are essentially critical (when working with polynomially decaying functions) in the sense that they can be bounded by a multiple of $\frac{v^{\out}}{\tilde{\tau}^2}$, where \[
\tau_2 = \int_t^{t_0}\lambda_2(s)\,ds.
\]
\item The fact that the source term $f$ will be decaying exponentially fast in $\tau_2$, and the preceding error is no longer critical when working with functions decaying exponentially with respect to $\tau_2$. 
\end{itemize}
Now $v^{\out}$ only solves \eqref{eq:veqnlininhom} approximately, in the sense that 
\begin{align*}
&-v^{\out}_{tt} + v^{\out}_{rr} + \frac{1}{r}v^{\out}_r - \frac{4\cos\big(2Q_1 - 2\bfcq_{n-1}\big)}{r^2}v^{\out}\\
&= f + \frac{4\cos\big(2\bfcq_{n-1}\big) - 4\cos\big(2Q_1 - 2\bfcq_{n-1}\big)}{r^2}v^{\out}.
\end{align*}
We hence pick $v^{\inn} $ to solve the elliptic problem 
\begin{align*}
&v^{\inn}_{rr} + \frac{1}{r}v^{\inn}_r - \frac{4\cos\big(2Q_1 \big)}{r^2}v^{\inn}\\
& = -\frac{4\cos\big(2\bfcq_{n-1}\big) - 4\cos\big(2Q_1 - 2\bfcq_{n-1}\big)}{r^2}v^{\out}.
\end{align*}
In fact, the preceding equation needs to be modified by adding a suitable correction term on the right, in order to ensure that good bounds result for $v^{\inn}$. These extra correction terms in turn force a delicate fine tuning of the precise innermost scaling parameter $\lambda_1(t)$.
\\

Setting $v = v^{\out} + v^{\inn}$, we have still only solved \eqref{eq:veqnlininhom} approximately, namely (omitting the extra correction term required to obtain good bounds for $v^{\inn}$)
\begin{align*}
&-v_{tt} + v_{rr} + \frac{1}{r}v^{\out}_r - \frac{4\cos\big(2Q_1 - 2\bfcq_{n-1}\big)}{r^2}v\\
&= f - v^{\inn}_{tt} + \frac{4\cos\big(2Q_1\big) - 4\cos\big(2Q_1 - 2\bfcq_{n-1}\big)}{r^2}v^{\inn}.
\end{align*}
It turns out that the additional error generated here, namely the sum of the last two terms at the end involving $v^{\inn}$, will have gained smallness, and we can then repeat the procedure with this error as source term instead. 
\\

Repeated application of the preceding steps allows us to {\it{approximately}} solve \eqref{eq:veqn}, up to an error of size 
\[
O\big(\tau^{-N}\big),
\]
where $N\gg 1$ is a fixed sufficiently large number, and we recall \eqref{eq:tauasympto}.  
We can then obtain an {\it{exact solution}} of \eqref{eq:veqn} by solving one last perturbation problem, for which the problematic potential term  
\[
\frac{4\cos\big(2Q_1 - 2\bfcq_{n-1}\big)}{r^2}v
\]
may be replaced by the simpler 
\[
\frac{4\cos\big(2Q_1\big)}{r^2}v,
\]
up to generating an error term which is of perturbative nature, thanks to the fast polynomial decay of the final correction. 

\subsection{Inductive hypothesis} From now on and for the remainder of this work we assume that the radiation part (or correction) $w_{n-1}$, which completes $\bfq_{n-1}$ to an exact solution of \eqref{eq:keq2corotational}, can actually be decomposed as the sum of $n-1$ other corrections, each of which associated with the completion of an $m$-bubble solution starting from a given $(m-1)$-bubble solution. In other words, we assume that there is a decomposition \[
w_{n-1}=w_{n-1,2} + w_{n-1,3} +... + w_{n-1,n-1}+ w_{n-1,n},
\]
where, roughly speaking, for each $m\in\{2,...,n\}$, we can think of \[
\sum_{j=m}^{n} w_{n-1,j} \quad \, \hbox{ as the correction one adds to } \, \quad 
\sum_{j=m}^{n} (-1)^j Q\big(\lambda_j(t) r\big)
\]
to complete it to an exact solution of \eqref{eq:keq2corotational}. We also require that the radiation associated with each new bubble satisfy the pointwise decay estimate \begin{align}\label{eq:intro_wn-1_decay_indhyp}
\vert S^k w_{n-1,j}\vert \lesssim_k \tau_j^{-2+}, \qquad j\in\{2,...,n-1\},\,k\geq 0, 
\end{align}
where $\tau_j$ is defined in Section \ref{sec:linout}, and is given by \[
\tau_j= \int_t^{t_0} \lambda_j(s)ds, \quad \hbox{ or, if one prefers, } \quad \tau_j= \exp\bigg( c \int_t^{t_0} \lambda_{j+1}(s)ds\bigg), \quad c>0.
\]
Furthermore, we require there to be an additional decomposition 
\begin{equation}\label{eq:intro-vn-1refined}
w_{n-1,j}(t, r) = c_{n-1,j}(t)r^2 + \tilde{w}_{n-1,j}(t,r), \qquad j\in\{2,...,n\},
\end{equation}
where we have the bounds (throughout $k\in \{0,1,...\}$, $\varepsilon>0$ arbitrarily small, and $t\in (0, t_0]$ with $t_0\ll 1$)
\begin{equation}\label{eq:intro-vn-1boundsrefined}\begin{split}
&\big|(t\partial_t)^kc_{n-1,j}(t)\big|\lesssim_k \lambda_j^{1+k \varepsilon}, 
\\ & \big\|r^{-4}S^k \tilde{w}_{n-1,j}\big\|_{L^\infty_{r\,dr}(r\lesssim t)}\lesssim_k \lambda_j^{ 3+k \varepsilon },
\end{split}\end{equation}
for all $j\in\{2,...,n\}$. Regarding the pointwise decay estimate \eqref{eq:intro_wn-1_decay_indhyp}, in the case $j=n$, the correction $w_{n-1,n}$ satisfies the specific decay rate described in Theorem \ref{thm:outersolnYManalogue}. Note that the decay rate of $w_{n-1,n}$, the correction associated with the outermost (lowest frequency) bubble, is by far the slowest one.

%%%%%%%%%%%%%%%%%%%%%%%%%%%%%%%%%
%%%%%%%%%%%%%%%%%%%%%%%%%%%%%%%%%
%%%%%%%%%%%%%%%%%%%%%%%%%%%%%%%%%
%%%%%%%%%%%%%%%%%%%%%%%%%%%%%%%%%
%%%%%%%%%%%%%%%%%%%%%%%%%%%%%%%%%

%\newpage
\bigskip
\section{Spectral Theory and the distorted Fourier Transform}

In this section we recall the definition and construction of the distorted Fourier Transform associated with \[
\mathcal{H}:= r^{1/2}\mathcal{L}\big( r^{-1/2} \cdot \big) = -\partial_r^2 + \dfrac{15}{4r^2} - \dfrac{32r^2}{(1+r^4)^2}, \]
where \[
\mathcal{L}:= - \partial_r^2 - \dfrac{1}{r}\partial_r + \dfrac{4}{r^2}\cos(2Q).
\]
All results in this section have been taken from \cite{JenKri} and \cite{KST1}. From now on we denote by $\phi$ and $\theta$ the Weyl solutions at zero for the half-line operator $\mathcal{H}$, and by $\Phi$ and $\Theta$ the associated Weyl solutions at zero for the radial $2$D operator, that is, \[
\Phi=r^{-1/2}\phi, \qquad \quad \Theta= r^{-1/2}\theta.
\]
These are explicitly given by \begin{equation}\label{dft1}\begin{aligned}
\Phi_{0} & := r \partial_rQ= \dfrac{4r^2}{1+r^4}, \quad  & & \quad \Theta_{0}  :=  \dfrac{1-8r^4\ln r-r^8}{16r^2(1+r^4)},
\\ 
\phi_{0} & := r^{3/2}\partial_r Q= \dfrac{4r^{5/2}}{1+r^4},
 \quad & & \quad \theta_{0}  := \dfrac{1-8r^4\ln r-r^8}{16r^{3/2}(1+r^4)}. 
\end{aligned}\end{equation}
Note that $\phi_{0}\in L^2(0,\infty)$, and they are normalized so that  $W[\theta_{0}, \phi_{0}] =1$. We emphasize that the eigenfunction $\Phi_{0}$ is nothing but the zero mode associated with the scaling invariance $r\partial_r Q$. 

Let us denote by $\phi(r,\xi)$ and $\theta(r,\xi)$ the fundamental system of solutions of \[
\mathcal{H}\phi=\xi \phi, \qquad \xi>0.
\]
We choose $\phi(r,\xi)$ as the (unique, up to normalization) Weyl solution at zero. We also denote by $\psi(r,\xi)$ the Weyl solution at infinity.
\begin{prop}[\cite{JenKri}]
The following statements hold:
\begin{enumerate}[leftmargin=*]
\item The operator $\mathcal{H}$ is self-adjoint with domain \[
\qquad \ \ \mathrm{Dom}(\mathcal{H})=\big\{ f\in L^2\big( (0,\infty)\big): \ f,f'\in \mathrm{AC}\big( (0,\infty) \big) , \ \mathcal{H}f \in L^2 \big( (0,\infty) \big) \big\}.
\]
The spectrum of $\mathcal{H}$ is \[
\sigma( \mathcal{H} ) = [0,\infty).
\]
Moreover, $(0,\infty)$ is purely absolutely continuous, and $0$ is a threshold eigenvalue.
\item There is a fundamental system $\phi(r,z)$,  $\theta(r,z)$  for $\mathcal{H}-z$ and arbitrary $z\in\C$, which
depends analytically on $(r,z)$, for $r>0$, which has the asymptotic behavior \[
\phi(r,z) \sim 4r^{5/2}, \qquad \theta(r,z)\sim \dfrac{1}{16} r^{-3/2}, \qquad \hbox{as} \qquad r\to0. \]
\end{enumerate}
\end{prop}

The following proposition gives us asymptotic power series expansions for $\phi$ and $\psi_+$ that will be useful for calculations and estimates.
\begin{prop}[ \cite{JenKri}]\label{dft7}
\begin{enumerate}[leftmargin=*]
\item Weyl solution at zero: For any $z\in\C$, the function $\phi(r,z)$ admits an absolutely convergent asymptotic expansion \[
\phi(r,z) = \phi_0(r) + r^{-3/2} \sum_{j=1}^\infty (r^2 z)^j \widetilde{\phi}_j(r^2),
\]
where the functions $\widetilde{\phi}_j$ are holomorphic on the region $\Omega := \big\{ \vert \im u\vert < \tfrac12 \big\} $, and satisfy the pointwise bounds \begin{align*}
\vert \wt \phi_j(u)\vert \leq \dfrac{C^j}{j!}\, \dfrac{\vert u\vert^2}{\langle u \rangle}, \qquad j\geq 1.
\end{align*}
\item Weyl solution at infinity: For any $\xi>0$, the equation $(\mathcal{H}-\xi)f = 0$ admits a solution $\psi_+(r,\xi)=\psi_+(r,\xi+i0)$ arising as limit of the Weyl-Titchmarsh solutions on the upper half plane, and of the form
\[
\qquad  \ \ \psi_+(r,\xi) = \xi^{-1/4} e^{i r \xi^{1/2}} \sigma (r\xi^{1/2}, r), \qquad r^2\xi \gtrsim1,
\]
where $\sigma = \sigma(q,r)$ admits an asymptotic series expansion of the form \[
\qquad  \ \ \sigma(q,r) \approx \sum_{j=0}^\infty q^{-j} \psi_{+,j}(r), \qquad \psi_{+,0}= 1, \qquad \psi_{+,1} = \dfrac{15i}{8}+O\big(\langle r\rangle^{-2}\big),
\]
with zero order symbols $\psi_{+,j}(r)$ that are analytic at infinity, in the sense that for all large integers $j_0$, and all indices $\alpha$, $\beta$, we have \[
\qquad  \ \ \sup_{r>0} \bigg\vert (r\partial_r)^\alpha (q \partial_q )^\beta \Big( \sigma(q,r) - \sum_{j=0}^{j_0} q^{-j} \psi_{+,j}(r) \Big) \bigg\vert \leq c_{\alpha, \beta, j_0} q^{-1-j_0}, 
\]
for all $q>1$.
\end{enumerate}
\end{prop}
Part (1) of the above proposition is useful to obtain various estimates for $\phi$ in the region $r^2\xi \lesssim1$, whereas, in the oscillatory regime  $r^2\xi \gtrsim1$, we  represent $\phi$ in terms of the Weyl solution at infinity, namely, \begin{align}\label{dft9}
\phi(r,\xi) =  a(\xi) \psi_+(r,\xi) + \overline{a (\xi) \psi_+(r,\xi)}, 
\end{align}
and use the asymptotic expansion described in part (2). The coefficient $a(\xi)$ is given by \[
a(\xi) = \dfrac{W\big( \phi(r,\xi), \overline{\psi_+(r,\xi)}\big)}{W\big( \psi_+(r,\xi), \overline{\psi_+(r,\xi)}\big)}.
\]
\begin{prop}
The function $a(\xi)$ is smooth on $(0,\infty)$, always non-zero, and satisfies \begin{align*}
\vert a(\xi)\vert &\sim \langle \xi\rangle^{ -1}.
\end{align*}
\end{prop}
Consequently, the Fourier representation formulas are given by  \begin{equation}\label{dft2}\begin{aligned}
f(r) & =  \lim_{ a \to +\infty } \int_0^a \phi(r,\xi) \widehat{f}(\xi) \rho'(\xi) d\xi +  \dfrac{\langle \phi_{0},f\rangle_{L^2_{dr}} } {\langle \phi_0,\phi_0 \rangle_{L^2_{dr}}} \phi_{0} ,
\\ \widehat{f}(\xi) & =  \lim_{b\to\infty} \int_0^b \phi(r,\xi) f(r) dr + \langle \phi_{0},f\rangle_{L^2_{dr}}  ,
\end{aligned}\end{equation}
where the spectral measure is given by 
\begin{align*}
\rho\big((\lambda_1,\lambda_2]\big)=\dfrac{1}{\pi}\lim_{\delta\to0^+}
  \lim_{\epsilon\to0^+}\int_{\lambda_1+\delta}^{\lambda_2+\delta}\mathrm{Im}\, m\big( \lambda + i\epsilon \big) d\lambda,
\end{align*}
and $m(\cdot)$, the so called Weyl-Titchmarsh $m$ function, is defined as 
\[
m(\xi)=\dfrac{W(\theta(r,\xi),\psi_+(r,\xi))}{W(\psi_+(r,\xi),\phi(r,\xi))}.
\] 
After some Wronskian calculations (see \cite{GZ}),  choosing the normalization of $\psi_+$ so that $W( \overline{ \psi_+},\psi_+) = 2i$, one finds that the density of the spectral measure is given by
\begin{align}\label{dft6}
\frac{d\rho}{dk}(\xi) = \dfrac{1}{\pi}\mathrm{Im}\,m(\xi) = \dfrac{1}{4\pi |a(\xi)|^2}.
\end{align}

\smallskip
\subsection{The Transference Operator}
In this subsection, we review some properties and representation of the scaling vector field $r\partial_r$ in the distorted Fourier side, which has already been derived previously in \cite{DK,JenKri,KST3}.

We would like to express the operator $r\partial_r$ in terms of the distorted Fourier coefficients. Recall that
\[
f(r) =  \dfrac{\langle \phi_{0},f\rangle_{L^2_{dr}} } {\langle \phi_0,\phi_0 \rangle_{L^2_{dr}}} \, \phi_{0}  +  \int_0^\infty \phi (r,\xi)  \mathcal{F}[f](\xi) \rho'(\xi)d\xi.
\]
The main difficulty in doing this is caused by the operator $r\partial_r$, which is not diagonal in the distorted Fourier basis. Moreover, the fact that in our present case we have both continuous and discrete spectra will require the introduction of a matrix-valued operator to express the operator $r\partial_r$ in terms of the distorted Fourier coefficients. To motivate the decomposition, we first recall that in the flat case we have \[
\Big( r\partial_r - 2 \xi \partial_\xi \Big)e^{\pm i\xi^{1/2}r} = 0.
\]
According to Lemma \ref{dft7}, this cancellation also occurs (to leading order at least) in the distorted case. This suggests one could expect to have 
\[
\mathcal{F}\big[ r\partial_r u\big](\xi) = - 2 \xi \partial_\xi \widehat{u} + \mathcal{K} \widehat{u}, \qquad \widehat{u}:= \begin{pmatrix}
\widehat{u}_p \\ \widehat{u}_c
\end{pmatrix},
\]
where $\mathcal{K}$ is a nonlocal operator, and from now on we think (for the sake of simplicity) of $\widehat{u}$ as a vector composed of its discrete and continuous components. The operator $\mathcal{K}$ takes the form (here $p$ stands for point and $c$ for continuous)
\begin{align*}
\mathcal{K}:= \begin{pmatrix}
 \mathcal{K}_{pp} & \mathcal{K}_{pc}
 \\ \mathcal{K}_{cp} & \mathcal{K}_{cc},
\end{pmatrix}
\end{align*}
where the entries are either operators or functions (which act by multiplication), as  \begin{equation}\label{dft8}\begin{aligned}
\mathcal{K}_{pp} & :=   \big\langle r \partial_r \phi_{0}, \,  \phi_{0} \big\rangle_{L^2_{dr}} / \langle \phi_0, \phi_0\rangle_{L^2_{dr}},
\\ \mathcal{K}_{cp} & :=   \big\langle r \partial_r \phi_{0}, \,  \phi(\cdot,\xi)\big\rangle_{L^2_{dr}} / \langle \phi_0,\phi_0\rangle_{L^2_{dr}},
\\ \mathcal{K}_{pc} & :=   \bigg\langle \int_0^\infty r \partial_r \phi(r,\eta) \widehat{u}_c(\eta) \rho'(\eta) d\eta, \,  \phi_{0} \bigg\rangle_{L^2_{dr}} ,
\\ \mathcal{K}_{cc} & :=  \bigg\langle \int_0^\infty r \partial_r \phi(r,\eta) \widehat{u}_c(\eta) \rho'(\eta) d\eta, \,  \phi(\cdot,\xi) \bigg\rangle_{L^2_{dr}}
\\ &  ~ + \bigg\langle \int_0^\infty  \phi(r,\eta) \big( 2\eta \partial_\eta \widehat{u}_c(\eta) \big) \rho'(\eta)   d\eta, \,  \phi(\cdot,\xi) \bigg\rangle_{L^2_{dr}},
\end{aligned}\end{equation}
having written $2\xi \partial_\xi \widehat{u} $ as  \[
2\xi \partial_\xi \widehat{u} = \mathcal{F}\Big[ \mathcal{F}^{-1}[2\xi \partial_\xi \widehat{u}]\Big]= \bigg\langle \int_0^\infty \phi(r,\eta) \big( 2\eta \partial_\eta \widehat{u}(\eta)\big) \rho'(\eta)d\eta  , \phi(r,\xi)\bigg\rangle_{L^2_{dr}},
\]
to obtain the last term in the definition of $\mathcal{K}_{cc}$. Note that, integrating by parts, we find that
\[
\big\langle r \partial_r \phi_{0}, \,  \phi_{0} \big\rangle_{L^2_{dr}}=-\dfrac{1}{2}\Vert \phi_0 \Vert_{L^2_{dr}}^2 = - \pi.
\]
On the other hand, the operator $\mathcal{K}_{pc}$ can be seen as an operator $\mathcal{K}_{pc}:L^2_{\rho d\xi} \to \C$ given by integration against the kernel \[
K_p(\xi) = -\big\langle r\partial_r \phi_0(r), \phi(r,\xi)\big\rangle_{L^2_{dr}},
\]
having integrated by parts in $r$. It is worth noting that, due to the commutation relation
\[
[\mathcal{H},r\partial_r] = 2\mathcal{H} +\big( 2 + r\partial_r\big) \dfrac{32r^2}{(1+r^4)^2},
\]
along with the eigen-equation
\[
\big(\mathcal{H}-\xi\big)(r\partial_r\phi)=[\mathcal{H},r\partial_r]\phi,
\]
one concludes that, for $\xi\gg 1$ large, \begin{align*}
& \big\langle r\partial_r\phi, \phi_0\big\rangle_{L^2_{dr}} 
\\ & = - \xi^{-1} \big\langle [\mathcal{H},r\partial_r]\phi ,\phi_0\big\rangle_{L^2_{dr}}  
\\ & = - 2 \xi^{-1} \Big\langle \dfrac{32r^2}{(1+r^4)^2}\phi(\cdot,\xi),\phi_0(\cdot)\Big\rangle_{L^2_{dr}} - \xi^{-1}\Big\langle r\partial_r\Big( \dfrac{32r^2}{(1+r^4)^2}\Big)\phi(\cdot,\xi),\phi_0(\cdot)\Big\rangle_{L^2_{dr}},
\end{align*}
repeated integration by parts lead to rapid decays as $\xi\to +\infty$ (see \eqref{dft9} and Lemma \ref{dft7}). Moreover, we also have that $\mathcal{K}_{cp}=-K_p$. Finally, the operator $\mathcal{K}_{cc}$ is the most delicate one, which is analyzed in the following proposition.
\begin{prop}[\cite{KST3}]\label{dft5}
We can write
\begin{align}\label{dft4}
\mathcal{K}_{cc} = -\Big( \dfrac{3}{2} + \dfrac{\eta \rho'}{\rho} \Big) \delta(\xi-\eta) + \mathcal{K}_0,
\end{align}
where $\mathcal{K}_0$ is the operator given by integration against the kernel \[
K_0(\eta,\xi)=\dfrac{1}{\eta-\xi}\,\rho(\xi)F(\xi,\eta),
\]
with $F$ of regularity $C^2$ on $(0,\infty)\times (0,\infty)$, continuous on its closure, and satisfying the bounds \begin{align*}
    \vert F(\xi,\eta)\vert & \lesssim \begin{cases}
        \xi+\eta, & \xi+\eta \leq 1,
        \\ (\xi+\eta)^{-3/2}(1+\vert \xi^{1/2}-\eta^{1/2}\vert)^{-N}, & \xi+\eta \geq 1,
    \end{cases}
    \\ \vert \partial_\xi F\vert + \vert \partial_\eta F\vert & \lesssim \begin{cases}
        1, & \xi+\eta \leq 1,
        \\ (\xi+\eta)^{-2}(1+\vert \xi^{1/2}-\eta^{1/2}\vert)^{-N}, & \xi+\eta \geq 1,
    \end{cases}
    \\ \sup_{j+k=2}\vert \partial_xi^j\partial_\eta^k F\vert & \lesssim \begin{cases}
        \vert \log(\xi+\eta)\vert^3, & \xi+\eta \leq 1,
        \\ (\xi+\eta)^{-5/2}(1+\vert \xi^{1/2}-\eta{1/2})^{-N}, & \xi+\eta \geq1,
    \end{cases}
\end{align*}
where $N$ is an arbitrarily large integer.
\end{prop}
The proof of the above proposition can be found in \cite{KST3}, see Theorem $6.1$. Moreover, from \cite{KST3} we also deduce the following mapping bounds.
\begin{prop}\label{prop:dft1}
The operators $\mathcal{K}_0$ and $\mathcal{K}$ satisfy
    \[
    \mathcal{K}_0: L^{2,\alpha}_\rho \to L^{2,\alpha+\frac12}_{\rho}, \ \qquad \mathcal{K}: L^{2,\alpha}_\rho \to L^{2,\alpha}_\rho,
    \]
    for $\alpha\in \R$ arbitrary. Moreover, we have the commutator bound \[
    \big[ \mathcal{K}, \xi\partial_\xi \big]: L^{2,\alpha}_\rho \to L^{2,\alpha}_\rho,
    \]
    where $\xi\partial_\xi$ only acts non-trivially on the contrinuous part of the spectrum.
\end{prop}

%%%%%%%%%%%%%%%%%%%%%%%%%%%%%%%%%
%%%%%%%%%%%%%%%%%%%%%%%%%%%%%%%%%
%%%%%%%%%%%%%%%%%%%%%%%%%%%%%%%%%
%%%%%%%%%%%%%%%%%%%%%%%%%%%%%%%%%

%\newpage
\bigskip
\section{Approximate solution of \eqref{eq:veqn}, part I: the main mechanism driving the evolution of $\lambda_1(t)$}

As before, assuming an $n-1$-bubble solution $\bfcq_{n-1}$ to have been constructed, we now commence the inductive process leading to a highly accurate approximate solution $v_N$ of \eqref{eq:veqn}, including the leading order asymptotics of $\lambda_1(t)$. We shall construct $v_N$ in the form 
\[
v_N = \sum_{j=0}^N h_j,
\]
and here we explain how to obtain $h_0$, modulo technical details referred to the subsequent sections. 
\\

To construct $h_0$, we shall neglect the terms $E_1$ and $E_3$ on the right-hand side in \eqref{eq:veqn}, and we shall replace the left-hand side by a simplified elliptic operator acting on $h_0$: 
\begin{equation}\label{eq:h0eqn}
h_{0,rr} + \frac{1}{r}h_{0,r} - \frac{4\cos\big(2Q_1\big)}{r^2}h_0 = E_2,\quad Q_1 = Q(R),\quad R = \lambda_1(t) r . 
\end{equation}
Setting %need to introduce these functions before%
\[
\Phi(R): = R^{-\frac12}\phi_0(R),\qquad \Theta(R): = R^{-\frac12}\theta_0(R),
\]
and recalling that $ W\big[\Theta(R), \Phi(R)\big] = 1/R$, we have the solution formula 
\begin{equation}\label{eq:variationofconstants}
h_0(r) = \Phi(R)\cdot \int_0^R \lambda_1^{-2}\cdot E_2\big(\tfrac{s}{\lambda_1}\big)\Theta(s)\,sds  -  \Theta(R)\cdot \int_0^R \lambda_1^{-2}\cdot E_2\big(\tfrac{s}{\lambda_1}\big)\Phi(s)\,sds .
\end{equation}
In order to prevent quadratic growth of the second term on the right towards $R = +\infty$, we need to impose the vanishing condition
\begin{equation}\label{eq:keyvanishing}
\int_0^{\infty} \lambda_1^{-2}\cdot E_2\Phi(s)\,sds = 0. 
\end{equation}
In fact, this condition will be modified slightly for technical reasons, as we shall replace $E_2$ by a simpler expression $\tilde{E}_2$, in particular only including the bulk components of $\bfcq_{n-1}$:
\begin{equation}\label{eq:tildeE2}
 \tilde{E}_2: = \frac{\lambda_1''}{\lambda_1^3}\cdot\Phi(R) + \big(\frac{\lambda_1'}{\lambda_1^2}\big)^2\cdot \big(R\Phi'(R) - \Phi(R)\big) - 8\cdot\frac{\sum_{j=2}^n(-1)^j\lambda_j^2}{\lambda_1^2}\cdot \big[\cos\big(2Q(R)\big) - 1\big].
\end{equation}
Note that the first two terms simply correspond to $\lambda_1^{-2}\partial_t^2Q_1$, namely, \[
\lambda_1^{-2} \partial_t^2 Q_1 =  \frac{\lambda_1''}{\lambda_1^3}\cdot\Phi(R) + \big(\frac{\lambda_1'}{\lambda_1^2}\big)^2\cdot \big(R\Phi'(R) - \Phi(R)\big),
\]
which is exact, whereas the last expression in \eqref{eq:tildeE2} is obtained by approximating the difference of the trig functions in \eqref{eq:veqn}, writing \begin{align*}
    & \sin(2Q_1-2\bfcq_{n-1})-\sin(2Q_1)+\sin(2\bfcq_{n-1})
    \\  & = \sin(2Q_1)\Big(\cos(2\bfcq_{n-1})-1\Big) - \sin(2\bfcq_{n-1})\Big(\cos(2Q_1)-1\Big),
\end{align*} 
and then neglecting the first term while substituting $\sin(2\bfcq_{n-1})/r^2$ by $4$ times the alternating sum of the $\lambda_j$'s. Then, replacing $E_2$ by $\tilde{E}_2$ in \eqref{eq:keyvanishing}, we deduce the following crucial relation which implicitly forces our choice of $\lambda_1$ in the inductive scheme: 
\begin{align*}
&\int_0^\infty  \Big[\frac{\lambda_1''}{\lambda_1^3}\cdot\Phi(R) + \big(\frac{\lambda_1'}{\lambda_1^2}\big)^2\cdot \big(R\Phi'(R) - \Phi(R)\big)\Big]\cdot \Phi(R)\,R\,dR\\
& =  -\frac{\sum_{j=2}^n(-1)^j\lambda_j^2}{\lambda_1^2}\cdot 8\int_0^\infty \big[{-}\cos\big(2Q(R)\big) + 1\big]\cdot R\Phi(R)\,dR, 
\end{align*}
Performing integration by parts in the integral on the left, and taking advantage of the identities
\begin{equation}\label{eq:explicitintegrals}\begin{split}
 &\int_0^\infty \Phi^2(R) RdR=-\dfrac{4R^2}{1+R^4}+4\arctan(R^2)\bigg\vert_{R=0}^\infty  = 2\pi ,
\\ &\int_0^\infty \Big(1-\cos\big(2Q_1(R)\big)\Big) \Phi(R) R dR  = -\dfrac{4(1+2R^4)}{(1+R^4)^2}\bigg\vert_{R=0}^\infty = 4,
\end{split}\end{equation}
we find the relation
\begin{equation}\label{eq:lambdaoneleadingorder}
\boxed{\frac{\lambda_1''}{\lambda_1^3} - 2 \big(\frac{\lambda_1'}{\lambda_1^2}\big)^2 = -\frac{16}{\pi}\cdot\frac{\sum_{j=2}^n(-1)^j\lambda_j^2}{\lambda_1^2}}
\end{equation}
It is worth mentioning that, morally speaking, one can get a good expectation of how the solution should behave near $0+$ from the WKB method. More specifically, consider the equation  \[
\dfrac{y''(t)}{y^3(t)} - 2 \dfrac{(y'(t))^2}{y^4(t)} + \dfrac{f(t)}{y^2(t)}=0 \quad \implies \quad u''(t) - f(t)u(t)=0, \quad u(t)=\dfrac{1}{y(t)},
\]
where, for simplicity, we assume that $f(t)$ grows monotonically and fast enough to infinity as $t\to0+$. Then, the standard WKB method tells us that the growing solution behaves as \begin{align}\label{modu5}
y(t) \sim (f(t))^{1/4}\exp\Big(\int_t^1 \big(f(s)\big)^{1/2}ds\Big), \qquad \quad t\sim 0+.
\end{align}
With the initial choice of 
\begin{equation}\label{eq:lambdandef}
\lambda_n(t) = \frac{|\log t|^{\beta}}{t},\qquad \beta>\frac32,
\end{equation}
we shall see that the functions $\lambda_{n-j}$, $n-1\geq j\geq 1$, grow very rapidly (on sufficiently short time intervals of the form $(0, t_0]$). In particular, this will imply that 
\[
\sum_{j=2}^n(-1)^j\lambda_j^2\sim \lambda_2^2, 
\]
and so 
\[
\lambda_1(t)\gtrsim e^{c\int_t^1\lambda_2(s)\,ds}
\]
for suitable $c>0$. In the following, we shall give the precise inductive construction of the $\lambda_j$, also involving a small perturbation parameter $m$ which shall be required for the 'elliptic stages' of the construction of the higher order corrections. 

\begin{lem}\label{lem:simpleinductivelambdaone}
Let $t_0>0$ and $\overline{\lambda}_2(t)$ a positive, decreasing, smooth function on $(0,t_0]$ satisfying 
\[
\lim_{t\rightarrow 0+}t\cdot \overline{\lambda}_2(t) = +\infty
\]
as well as the derivative type bounds 
\[
\big|\overline{\lambda}_2^{(k)}(t)\big|\leq C_{k,\nu}\cdot \overline{\lambda}_2^{1+\nu k}(t),\,k\geq 1,\,\nu>0.
\]
Here the $C_k$ are positive constants. Then there exists $t_1>0,\,t_1\leq t_0$, depending on finitely many of the $C_k$, such that the problem 
\[
\frac{\lambda_1''}{\lambda_1^3} - 2 \big(\frac{\lambda_1'}{\lambda_1^2}\big)^2 = -\frac{\overline{\lambda}_2^2}{\lambda_1^2}
\]
admits a positive, decreasing solution $\lambda_1\in C^\infty\big((0,t_1]\big)$, satisfying the estimates 
\[
\lambda_1(t)\geq e^{c\int_t^{t_2}\overline{\lambda}_2(s)\,ds},\,t\in (0, t_1],
\]
for a suitable constant $c>0$. Furthermore, we have the estimates 
\begin{align}\label{eq:app_lam_1}
\big|\lambda_1^{(k)}(t)\big|\leq D_{k,\nu}\cdot \lambda_1^{1+\nu k}(t),\,k\geq 1,\,\nu>0
\end{align}
for suitable constants $D_{k,\nu}$. Alternatively, we have the estimate 
\begin{align}\label{eq:app_lam_2}
\big|\lambda_1^{(k)}(t)\big|\leq E_{k}\cdot\bar{\lambda}_2^k(t)\lambda_1(t). 
\end{align}
\end{lem}
\begin{proof} Formally setting $\lambda_1(t) = e^{\alpha(t)}$, we deduce 
\[
\alpha'' - \big(\alpha'\big)^2 = -\overline{\lambda}_2^2.
\]
We shall show that picking $0<t\leq t_1$ small enough, we can select $\alpha'<0$ satisfying the preceding relation. For this we pass instead to the variable $\zeta(t): = (\alpha')^{-1}(t)$, which formally satisfies  
\[
\zeta' = \big(\overline{\lambda}_2\zeta\big)^2 - 1. 
\]
Write 
\begin{align*}
\zeta = -\frac{1}{\overline{\lambda}_2} - \frac{\overline{\lambda}_2'}{2\overline{\lambda}_2^3} + \frac{\overline{\lambda}_2''}{4\overline{\lambda}_2^4} - \frac{5(\overline{\lambda}_2')^2}{8\overline{\lambda}_2^5} + w. 
\end{align*}
Then we infer the following equation for $w$:
\begin{equation}\label{eq:weqn}
w' + 2\overline{\lambda}_2(1+o(1))w = E + \big(\overline{\lambda}_2 w\big)^2,
\end{equation}
where $E = E(t)$ is a smooth function satisfying the bounds 
\[
\big|E^{(k)}\big|\lesssim_k \overline{\lambda}_2^{-3+},\,k\geq 0,
\]
according to our assumptions. Furthermore, the function 
\begin{align*}
o(1): =  \frac{\overline{\lambda}_2'}{2\overline{\lambda}_2^2} -\frac{\overline{\lambda}_2''}{4\overline{\lambda}_2^3} + \frac{5(\overline{\lambda}_2')^2}{8\overline{\lambda}_2^4}
\end{align*}
satisfies the estimates 
\begin{align*}
\big|\big(o(1)\big)^{(k)}\big|\lesssim_k \overline{\lambda}_2^{-1+}.
\end{align*}
Writing 
\[
 2\overline{\lambda}_2(1+o(1)) =:\tilde{\lambda}_2, 
 \]
we can characterize $w$ as a solution of the following fixed-point problem:
\begin{equation}\label{eq:wfixedpoint}
w(t) = e^{\int_t^{t_0}\tilde{\lambda}_2(s)\,ds}\cdot \int_0^t e^{-\int_{t'}^{t_0}\tilde{\lambda}_2(s)\,ds}\cdot \big[E(t') + \big(\overline{\lambda}_2 w(t')\big)^2\big]\,dt'.
\end{equation}
Noting that by our assumption, we have the bound 
\[
e^{-\int_{t'}^{t}\tilde{\lambda}_2(s)\,ds}\leq 1
\]
provided we restrict $t\leq t_1$ sufficiently small, and furthermore, we have the bootstrapping bound 
\begin{align*}
\big\|\overline{\lambda}_2^3(t')\big(\overline{\lambda}_2 w(t')\big)^2\big\|_{L^\infty(0,t_1]}\ll \big\|\overline{\lambda}_2^3 w\big\|_{L^\infty(0,t_1]}, 
\end{align*}
we infer from the above bounds on $E$ the existence of a continuous solution on $(0, t_1]$, $t_1$ sufficiently small, by means of a standard fixed point argument. Furthermore, taking advantage of the derivative bound
\begin{align*}
\big|\partial_t^k\big(e^{-\int_{t'}^{t}\tilde{\lambda}_2(s)\,ds}\big)\big|\lesssim_k \overline{\lambda}_2^k,\,k\geq 0,\,t'\leq t,
\end{align*}
we can also obtain inductively the bounds 
\begin{equation}\label{eq:wderivative}
\big|w^{(k)}\big|\lesssim_k \overline{\lambda}_2^{-3+k},\,k\geq 0. 
\end{equation}
The preceding implies that 
\[
\zeta =  -\frac{1}{\overline{\lambda}_2}\cdot \big(1+g(t)\big),
\]
where we have the bounds 
\[
\big|g^{(k)}\big|\lesssim_k \overline{\lambda}_2^{-1+k+},\,k\geq 0.
\]  
We can now set 
\begin{align*}
\alpha(t) = \int_t^{t_1}\overline{\lambda}_2(s)\cdot\big(1+g(s)\big)^{-1}\,ds    
\end{align*}
for $0<t_1\ll 1$ sufficiently small, and deduce from the above the bounds 
\begin{equation}\label{eq:lambdaonederivativebounds1}
\big|\lambda_1^{(k)}\big|\lesssim_k \overline{\lambda}_2^k\cdot \lambda_1, \qquad k\geq 0.
\end{equation}
This corresponds exactly to \eqref{eq:app_lam_2}. 
Clearly $\lambda_1$ is positive, smooth, and also decreasing. Finally, we prove the second-to-last estimate \eqref{eq:app_lam_1}. Due to the above bounds on $g$, we have the lower bound 
\[
\lambda_1(t)\geq e^{c\int_t^{t_1}\overline{\lambda}_2(s)\,ds}
\]
for any given $c>0$, provided $t_1$ is sufficiently small. Further, for any $\nu>0$, we have that 
\begin{align*}
\frac{d}{dt}\big(\nu c\int_t^{t_1}\overline{\lambda}_2(s)\,ds - \log \overline{\lambda}_2(t)\big) = -\nu c\overline{\lambda}_2(t) + \frac{\overline{\lambda}_2'(t)}{\overline{\lambda}_2(t)}<-\frac{\nu c}{2}\overline{\lambda}_2(t)
\end{align*}
provided $0<t<t_*(\nu, c)$ small enough, and our assumption on $\overline{\lambda}_2(t)$ implies that 
\[
\lim_{t\rightarrow 0+}\int_t^{t_*}\overline{\lambda}_2(s)\,ds = +\infty. 
\]
It follows that 
\begin{equation}\label{eq:integraldominance}
\lim_{t\rightarrow 0+}\nu c\int_t^{t_1}\overline{\lambda}_2(s)\,ds - \log \overline{\lambda}_2(t) = +\infty, 
\end{equation}
and so 
\begin{align*}
\overline{\lambda}_2^k(t)\lesssim_{k,\nu}\lambda_1^{\nu k}(t).     \end{align*}
In light of \eqref{eq:lambdaonederivativebounds1}, the desired derivative bound for $\lambda_1$ stated in the lemma follows. 
\end{proof}

Repeated application of the lemma implies the following conclusion:
\begin{prop}\label{prop:exponentialtowers1} Given $n\geq 2$, $\beta>0$, there is $t_0 = t_0(n,\beta)>0$, such that the inductively defined system of scaling parameters 
\begin{align*}
&\lambda_n(t): = \frac{|\log t|^{\beta}}{t},\\
&\frac{\lambda_k''}{\lambda_k^3} - 2 \big(\frac{\lambda_k'}{\lambda_k^2}\big)^2 = -\frac{16}{\pi}\cdot\frac{\sum_{j=k+1}^n(-1)^{j-k-1}\lambda_j^2}{\lambda_k^2}, \quad 1\leq k\leq n-1,
\end{align*}
admits a smooth solution $\big(\lambda_j(t)\big)_{j=1}^n$ on $(0, t_0]$, such that the $\lambda_j$, $1\leq j<n$, satisfy the bounds 
\[
\lambda_j(t) \geq e^{c\int_t^{t_0}\lambda_{j+1}(s)\,ds},\qquad 1\leq j<n,
\]
for suitable $c>0$, as well as 
\begin{align*}
&\big|\lambda_j^{(k)}\big|\leq E_{k,\nu,n} \lambda_j^{1+k\nu},\\
&\big|\lambda_j^{(k)}\big|\leq F_{k}\lambda_{j+1}^k\lambda_j. 
\end{align*}
for any $\nu>0$. In particular, the $\lambda_j$ display tower exponential growth. The same asymptotics hold provided we replace $\lambda_n$ by $\lambda_n\cdot (1+d_n)$ where 
\[
(t\partial_t)^kd_n(t)\lesssim_k |d_n(t)|\lesssim (|\log t|)^{-C}\,\forall k\geq 0,\,C>0,
\]
and we replace $\sum_{j=k+1}^n(-1)^{j-k-1}\lambda_j^2$ by 
\[
\big(\sum_{j=k+1}^n(-1)^{j-k-1}\lambda_j^2\big)\cdot (1+e_{n,k}(t))
\]
where $e_n$ satisfies the same type of bounds as $d_n$. 
\end{prop}
\begin{proof} It was shown in \cite{JenKri} that one obtains for $\lambda_{n-1}\geq e^{c|\log t|^{\beta+1}}$ the bounds asserted in the proposition, and for the following $\lambda_j$, $j<n-1$, the assertion follows by noting that 
\begin{align*}
\overline{\lambda}_{k+1}: = \sqrt{\sum_{j=k+1}^n(-1)^{j-k-1}\lambda_j^2}
\end{align*}
satisfies the same bounds as $\lambda_{k+1}$ on $(0, t_0]$, for $t_0$ sufficiently small. Indeed, inductive application of the preceding lemma leads to this conclusion. The last modification of the lemma follows similarly. 
\end{proof}

In the following, we shall have to modify slightly the equation \eqref{eq:lambdaoneleadingorder} defining $\lambda_1$ from the preceding scaling rates. So we state here a 'perturbed version' of Lemma~\ref{lem:simpleinductivelambdaone}, as follows:
\begin{lem}\label{lem:simpleinductivelambdaonewithm} Let $\overline{\lambda}_2$ be as in Lemma~\ref{lem:simpleinductivelambdaone}, and set 
\[
\tau: = e^{\int_t^{t_0}\overline{\lambda}_2(s)\,ds},\,t\in (0, t_1],
\]
with $0<t_1\leq t_0$. Further, let $m\in C^\infty\big((0, t_0]\big)$ satisfy the bound 
\[
\big|m(t)\big|\leq \tau^{-\frac12},\,t\in (0, t_0]. 
\]
Then the perturbed equation 
\[
\frac{\lambda_1''}{\lambda_1^3} - 2 \big(\frac{\lambda_1'}{\lambda_1^2}\big)^2 = -\frac{\big(\overline{\lambda}_2 + m\big)^2}{\lambda_1^2},
\]
admits a solution $\lambda_1\in C^\infty\big((0, t_1]\big)$, provided $0<t_1$ is sufficiently small. This solution satisfies the bounds\footnote{Here we use the operator $t\partial_t$ instead of $\partial_t$, as it will appear naturally at later stages of the proof.} (with any $\nu>0$)
\begin{align*}
&\big|(t\partial_t)^k\lambda_1\big|\lesssim_{\nu,k}\lambda_1^{1+\nu k}\cdot\big(1 + \tau^{-\frac12}\big\|m\big\|_{\frac12,(k-2)_+}\big)^k,\\
&\big|(t\partial_t)^k\lambda_1\big|\lesssim_{\nu,k}\bar{\lambda}_2^{k}\lambda_1\cdot\big(1 + \tau^{-\frac12}\big\|m\big\|_{\frac12,(k-2)_+}\big)^k,
\end{align*}
where we utilize the norms (with $p>0, l\geq 0$)
\[
\big\|m\big\|_{p,l}: = \sum_{j=0}^l\big\|\tau^{p}(t\partial_t)^jm(t)\big\|_{L^\infty((0,t_0])}.
\]
\end{lem}
\begin{proof} Let $\zeta$ be as in the proof of the preceding Lemma~\ref{lem:simpleinductivelambdaone}. We shall then strive to construct 
\begin{equation}\label{eq:lambdabetaprime}
\lambda_1(t) = e^{\int_{t_0}^{t}\beta'(s)\,ds}, 
\end{equation}
where we set 
\begin{equation}\label{eq:betaprimezetanu}
\big(\beta'(s)\big)^{-1} = \zeta(s) + \nu(s),
\end{equation}
with $\nu(s)$ playing a purely perturbative role for $0<s$ small enough. Then $\nu$ needs to satisfy the following equation:
\begin{equation}\label{eq:nueqn}
\nu' = \big(m^2 + 2\overline{\lambda}_2m\big)\zeta ^2+ \big(\nu^2 + 2\nu\zeta\big)\big(\overline{\lambda}_2 + m\big)^2.
\end{equation}
Writing 
\[
\tilde{\lambda}_2: = 2\zeta\big(\overline{\lambda}_2 + m\big)^2, \qquad G(t,\nu): = \big(m^2 + 2\overline{\lambda}_2m\big)\zeta ^2 + \nu^2\big(\overline{\lambda}_2 + m\big)^2, 
\]
we can characterize $\nu$ as a solution of the following fixed-point problem:
\begin{equation}\label{eq:nufixedpoint}
\nu(t) = e^{-\int_t^{t_0} \tilde{\lambda}_2(s)\,ds}\cdot \int_0^{t} e^{\int_{t'}^{t_0} \tilde{\lambda}_2(s')\,ds'}\cdot G(t', \nu)\,dt'.
\end{equation}
Note that since $\zeta(t)<0$ for $0<t\ll 1$, we have 
\[
e^{-\int_t^{t_0} \tilde{\lambda}_2(s)\,ds}\cdot  e^{\int_{t'}^{t_0} \tilde{\lambda}_2(s')\,ds'} =  e^{\int_{t'}^{t} \tilde{\lambda}_2(s')\,ds'}\leq 1,\qquad 0<t'\leq t. 
\]
Then we observe that due to \eqref{eq:integraldominance} we have the estimate 
\begin{align*}
\big\|\nu^2\big(\overline{\lambda}_2 + m\big)^2\big\|_{\frac12,0}\ll \big\|\nu\big\|_{\frac12,0}^2,     
\end{align*}
provided $t_0>0$ is small enough, and furthermore the proof of Lemma~\ref{lem:simpleinductivelambdaone} implies that 
\[
\big\|\big(m^2 + 2\overline{\lambda}_2m\big)\zeta^2\big\|_{\frac12,0}\lesssim \big\|m\big\|_{\frac12,0}. 
\]
It then follows that the map 
\[
\nu\longrightarrow e^{-\int_t^{t_0} \tilde{\lambda}_2(s)\,ds}\cdot \int_0^{t} e^{\int_{t'}^{t_0} \tilde{\lambda}_2(s')\,ds'}\cdot G(t', \nu)\,dt'
\]
is a contraction with respect to the norm $\|\cdot\|_{\frac12,0}$, provided $t_0>0$ is sufficiently small. In turn we infer the existence of a fixed point $\nu\in C^0\big((0,t_0]\big)$ satisfying 
\[
\big\|\nu\big\|_{\frac12,0}\lesssim \big\|m\big\|_{\frac12,0}.
\]
For the derivative bounds, as in the proof of Lemma~\ref{lem:simpleinductivelambdaone} we have 
\begin{align*}
\big|(t\partial_t)^k\big( e^{\int_{t'}^{t} \tilde{\lambda}_2(s')\,ds'}\big)\big|\lesssim_{\nu, k}\lambda_1^{\nu k}\cdot\big(1 + \tau^{-\frac12}\big\|m\big\|_{\frac12,k-1}\big)^k,\,k\geq 0. 
\end{align*}
We conclude first that 
\begin{align*}
\big\|(t\partial_t)\nu\big\|_{\frac12,0}\lesssim_{\nu}\lambda_1^{\nu}\cdot\big(1 + \tau^{-\frac12}\big\|m\big\|_{\frac12,0}\big), 
\end{align*}
where we also took advantage of the estimate for $\big\|\nu\big\|_{\frac12,0}$. Proceeding inductively on $k$, we then infer the more general higher-order bounds 
\begin{align*}
\big\|(t\partial_t)^k\nu\big\|_{\frac12,0}\lesssim_{\nu,k}\lambda_1^{\nu k}\cdot\tau^{-\frac12}\big\|m\big\|_{\frac12,k-1}\big(1 + \tau^{-\frac12}\big\|m\big\|_{\frac12,k-1}\big)^k, 
\end{align*}
Recalling the definition of $\lambda_1$ and $\beta'$ at the beginning of this proof, as well as the fact that 
\[
\overline{\lambda}_2(t)\ll \tau^{0+}
\]
for sufficiently small $t>0$, we infer that 
\[
\beta'(s)\sim \zeta^{-1}(s),\,s\ll 1,
\]
and furthermore, we deduce the bounds 
\begin{align*}
\big|(t\partial_t)^k\big(\lambda_1(t)\big)\big| = \big|(t\partial_t)^k\big(e^{\int_{t_0}^{t}\beta'(s)\,ds}\big)\big|\lesssim_{\nu, k}\lambda_1^{1+\nu k}\cdot\big(1 + \tau^{-\frac12}\big\|m\big\|_{\frac12,k-1}\big)^k,\,k\geq 0,
\end{align*}
provided $0<t\leq t_0$ with $t_0$ sufficiently small. The bound with $\bar{\lambda}_2$ instead of $\lambda_1^{\nu}$ is obtained similarly. 
\end{proof}

For later purposes, it shall be important to understand how $\lambda_1$ changes upon modifying $m$. Letting $\phi = \phi(t;m)$ a function, we use the notation
\begin{equation}\label{eq:differencingnotation}
\triangle \phi(t;m,n) : = \phi(t;m+n) - \phi(t;m). 
\end{equation}
In particular, we set 
\[
\triangle\big(e^{\int_{t'}^{t} \tilde{\lambda}_2(s)\,ds}\big)\big(t;m,n\big) = e^{\int_{t'}^{t} \tilde{\lambda}_2(s;m+n)\,ds} - e^{\int_{t'}^{t} \tilde{\lambda}_2(s;m)\,ds},
\]
where we use the notation $\tilde{\lambda}_2(s;m)$ to emphasize the dependency of the function on $m$. Letting
\[
H(t;m,n;\alpha): = e^{\int_{t'}^{t} \tilde{\lambda}_2(s;m+\alpha n)\,ds},\,\alpha\in [0,1], 
\]
we can write 
\[
\triangle\big(e^{\int_{t'}^{t} \tilde{\lambda}_2(s)\,ds}\big)\big(t;m,n\big) = \int_0^1\frac{\partial}{\partial\alpha}H(t;m,n;\alpha)\,d\alpha.
\]
Exploiting the fact that $\tilde{\lambda}_2(s;m+\alpha n)\leq 0$ for $t\ll 1$, as well as the explicit form of $\tilde{\lambda}_2$ and the estimates from Lemma~\ref{lem:simpleinductivelambdaone}, we deduce that (for $p\geq \frac12$, say)
\begin{equation}\label{eq:differencingbound1}
\big\|\triangle\big(e^{\int_{t'}^{t} \tilde{\lambda}_2(s)\,ds}\big)\big(t;m,n\big)\big\|_{p,0}\lesssim \tau^{-p}\big\|n\big\|_{p,0},
\end{equation}
provided we impose the a priori bound 
\begin{equation}\label{eq:mnbasicbound}
\big\|m\big\|_{p,0} + \big\|n\big\|_{p,0}\leq 1, 
\end{equation}
and we assume that $t_0$ (as appearing in the definition of $\|\cdot\|_{p,l}$) is small enough. Furthermore, applying the differencing operator $\triangle$ to the fixed point relation \eqref{eq:nufixedpoint}, and further the Leibniz type rule 
\begin{equation}\label{eq:Leibnizdifferencing}\begin{split}
&\triangle\big(f\cdot g\big)(t;m,n)\\
& = \big(\triangle f\cdot  \triangle g\big)(t;m,n) + f(t;m)\cdot \triangle g(t;m,n) +  g(t;m) \cdot \triangle f(t;m,n),
\end{split}\end{equation}
we derive via a bootstrap argument that 
\begin{align*}
\big\|\triangle \nu(\cdot;m,n)\big\|_{p,0}\lesssim \big\|n\big\|_{p,0},\,p\geq \frac12. 
\end{align*}
As far as higher derivatives are concerned, recalling the basic assumption on $\overline{\lambda}_2$ as formulated in Lemma~\ref{lem:simpleinductivelambdaone}, as well as using induction on the number of derivatives and bootstrap arguments as before, we deduce that (for $p\geq \frac12$)
\begin{equation}\label{eq:nuderivativebounds}
\big|(t\partial_t)^l\triangle \nu(t;m,n)\big|\lesssim_{p,l}\tau^{-p+}\cdot\big(1+\|m\|_{p,l}+ \|n\|_{p,l}\big)^{l-1}\cdot \big\|n\big\|_{p,l}. 
\end{equation}
Recalling the relations \eqref{eq:betaprimezetanu}, \eqref{eq:lambdabetaprime}, we then find that, with $\lambda_1$ determined as in Lemma~\ref{lem:simpleinductivelambdaonewithm}, we have the differencing bound 
\begin{equation}\label{eq:lambdaonedifferencing}
\big|\lambda_1^{-1}\cdot\big(t\partial_t\big)^l\triangle\lambda_1(t;m,n)\big|\lesssim_{l} \tau^{-p+}\cdot\big(\Big\|m\Big\|_{p,l} +\Big\|n\Big\|_{p,l}+1\big)^{l-1}\cdot\Big\|n\Big\|_{p,l}
\end{equation}
We also make the following simple observation concerning the 'time variable' associated to the new scale $\lambda_1$, defined by 
\[
\tau_1: = \int_t^{t_1}\lambda_1(s)\,ds. 
\]
Note that under the assumptions of the preceding lemma, we have
\begin{align*}
 \int_t^{t_1}\lambda_1(s)\,ds& = \int_t^{t_1}e^{-\int_{s}^{t_0}\beta'(s_1)\,ds_1}\,ds\\&   = \big(-\beta'(t)\big)^{-1}\cdot e^{-\int_{t}^{t_0}\beta'(s_1)\,ds_1} + C_1\\
 & - \int_t^{t_1}\frac{d}{ds}\big[\big(\beta'(s)\big)^{-1}\big]e^{-\int_{s}^{t_0}\beta'(s_1)\,ds_1}\,ds,
\end{align*}
where we let $C_1 = \big(\beta'(t_1)\big)^{-1}\cdot e^{-\int_{t_1}^{t_0}\beta'(s_1)\,ds_1}$. But then the relation \eqref{eq:betaprimezetanu}, and the fact that 
\[
\zeta+\nu \sim -\overline{\lambda}_2^{-1},\qquad \Big|\frac{d}{ds}\big(\zeta+\nu\big)\Big|\lesssim \overline{\lambda}_2^{-1+}
\]
imply that 
\begin{equation}\label{eq:importanttauoneasymptotic}
 \tau_1 = \int_t^{t_1}\lambda_1(s)\,ds\sim \frac{\lambda_1}{\overline{\lambda}_2}
\end{equation}
provided $t\ll 1$.

%%%%%%%%%%%%%%%%%%%%%%%%%%%%%%%%%%
%%%%%%%%%%%%%%%%%%%%%%%%%%%%%%%%%%
%%%%%%%%%%%%%%%%%%%%%%%%%%%%%%%%%%
%%%%%%%%%%%%%%%%%%%%%%%%%%%%%%%%%%

%\newpage
\bigskip
\section{Bounds for the linearisation around the outer $n-1$ bubble $\bfcq_{n-1}$}\label{sec:linout}

In the next sections, we will need a priori bounds for the propagator associated with the linearisation around the full outer profile $\bfcq_{n-1}$. More specifically, we will need bounds for the inhomogeneous linear equation with suitably decaying source terms. Consider the equation \begin{align}\label{linout1}
-\partial_t^2 h + \partial_r^2 h + \dfrac{1}{r} \partial_r h - \dfrac{4}{r^2}\cos\big( 2\bfcq_{n-1}\big)h  = F
\end{align}
We introduce the time variables \begin{align}\label{linout20}
    \tau &:= \exp\bigg( c \int_t^{t_0} \lam_2(s)ds\bigg), \qquad  \tau_2  := \int_t^{t_0} \lam_2(s)ds, \qquad c>0.
\end{align}
One can think of $\tau$ as being an approximation of $\lambda_1$, which we use here for simplicity. A particularly useful observation that we will use later is that 
\[
\tau_2 \sim \dfrac{\lambda_2(t)}{\lambda_3(t)}\bigg(1 + O \Big( \dfrac{  \lambda_3'(t)}{\big(\lambda_3(t)\big)^2} \Big)\bigg)  \quad \hbox{ for } \quad t\sim 0+,
\]
which can easily be seen from Laplace Method, provided \[
\dfrac{\vert \lambda_3'(t)\vert}{\lambda_3^2(t)} \to 0 \quad \hbox{ as } \quad t\to 0+,
\]
which is our present case (see Lemma \ref{prop:exponentialtowers1}). The following Lemma is the main result of this section. 
\begin{lem}\label{lem:linout_1}
Let $M\geq 2$ be fixed. Suppose that the source term $F$ satisfies \begin{align}\label{linout16}
\lambda_2^{-1}\Vert S^{\ell}F\Vert_{L^2_{rdr}}\lesssim \tau^{-p}, \qquad p>0, \quad 0\leq \ell \leq M.
\end{align}
Then, there exists $t_0=t_0(p,M,\beta)$ such that, for any $0<\varepsilon\ll1$, \eqref{linout1} admits a solution on $(0,t_0]$ satisfying the bounds
\begin{align}\label{linout10}
    \big\Vert S^\ell h(t,\cdot)\big\Vert_{\boldsymbol{\mathcal{H}}^1_{rdr}} \lesssim_{\ell,\varepsilon} \tau^{-(p-\ell \varepsilon)}, \qquad 0\leq \ell \leq M,
\end{align}
where the energy norm $\Vert \cdot\Vert_{\boldsymbol{\mathcal{H}}^1_{rdr}}$ is defined as \begin{align}\label{linout18}
\Vert h\Vert_{\boldsymbol{\mathcal{H}}^1_{rdr}} := \Vert \mathcal{L}^{1/2}_t h\Vert_{L^2_{rdr}} + \Vert \partial_t h\Vert_{L^2_{rdr}}+\lambda_2(t) \,\Vert h\Vert_{L^2_{rdr}},
\end{align}
having denoted by $\mathcal{L}_t:=-\partial_r^2-\tfrac{1}{r}\partial_r+\tfrac{4}{r^2}\cos(Q_2)$. Moreover, the following estimate hold \begin{align}\label{linout12}
\Big\Vert \dfrac{1}{r^2}S^\ell h\Big\Vert_{L^2_{rdr}} + \Big\Vert \dfrac{1}{r^2}S^{\ell}h \Big\Vert_{L^4_{rdr}} \lesssim_{\ell,\beta,p,\varepsilon} \tau^{-(p-(\ell+1)\varepsilon)}, \qquad 0\leq \ell\leq M-2.
\end{align}
\end{lem}
\begin{proof}
First of all, rewrite equation \eqref{linout1} as \begin{equation}\begin{aligned}\label{linout7}
& -\partial_t^2 h + \partial_r^2 h + \dfrac{1}{r} \partial_r h - \dfrac{4}{r^2}\cos\big( 2Q_2\big)h %
\\ & = F + \dfrac{4}{r^2}\Big( \cos\big(2\bfcq_{n-1}\big) - \cos(2 Q_2) \Big)h.
\end{aligned}\end{equation}
We begin by showing the desired bound \eqref{linout10} in the case $\ell=0$, and then proceed inductively to prove the general case $\ell\geq 1$.

Now, to handle the above equation, it is convenient to switch to coordinates in which the Schr\"odinger operator is no longer time dependent. Concretely, changing variables $R = \lambda_2(t) r$, and defining the re-scaled unknown  \[
\wt h(\tau_2,R) = R^{1/2} h\big( t(\tau_2),\, \lambda_2^{-1} R\big),
\]
using that \[
\partial_th(t,r)= \dfrac{(-\lambda_2)}{R^{1/2}} \partial_{\tau_2} \widetilde{h}  + \dfrac{(-\lambda_{2,\tau_2})}{R^{1/2}}R\partial_R \widetilde{h} + \dfrac{\lambda_{2,\tau_2}}{2R^{1/2}}\widetilde{h},
\]
and similarly for the second derivative, we find that  \begin{equation}\begin{aligned}\label{linout2}
& \bigg( - \Big( \partial_{\tau_2} + \dfrac{\lambda_{2,\tau_2}}{\lambda_2} R\partial_R \Big)^2 + \dfrac{1}{4}\Big(\dfrac{\lambda_{2,\tau_2}}{\lambda_2}\Big)^2 + \dfrac{1}{2} \partial_{\tau_2} \Big(\dfrac{\lambda_{2,\tau_2}}{\lambda_2}\Big) \bigg) \wt h - \mathcal{H} \wt h 
\\ & = R^{1/2}\bigg( \dfrac{1}{\lambda_2^2}F + \dfrac{4}{R^2}\Big( \cos(2\bfcq_{n-1}) - \cos(2Q_2)\Big) h  \bigg),
\end{aligned}\end{equation}
where the first large parenthesis on the left of \eqref{linout2} corresponds the contribution of $-\partial_t^2$ in the new variables. Next, for the sake of notation, set \begin{align*}
\omega_2 := \dfrac{ \partial_{ \tau_2   } \lambda_2}{\lambda_2} =\dfrac{1}{\lambda_2}\, \dfrac{d\lambda_2 / dt}{d\tau_2  /dt} = -\dfrac{\lambda_2'}{\lambda_2^2}, \qquad \hbox{and} \qquad \dot{\omega}_2=\partial_{\tau_2  } \omega_2.
\end{align*}
It is worth noting that \[
\omega_2 \sim \dfrac{1}{\tau_2} \qquad \hbox{ and that } \qquad \vert \dot{\omega}_2 \vert \sim \dfrac{1}{\tau_2^2}.
\]
Finally, define the operator $D_{\tau_2  } := \partial_{\tau_2  } -\omega_2 \mathcal{K}_d$, and \[
\mathcal{K}_d:= \begin{pmatrix}
1 & 0
\\ 0 & 2\xi \partial_\xi + 2 + \tfrac{\xi \rho'}{\rho} 
\end{pmatrix}, \quad \mathcal{K}_{nd}:= \begin{pmatrix}
0 & \mathcal{K}_{pc} 
\\ \mathcal{K}_{cp} & \mathcal{K}_0
\end{pmatrix} .
\]
We point out that the constant term $+2$ in the bottom right entry in the definition of $\mathcal{K}_d$, instead of $\tfrac{3}{2}$ as in Lemma \ref{dft5}, arises partly because the last two terms in the large parenthesis on the left-hand side of \eqref{linout2}, with coefficients $\tfrac14$ and $\tfrac12$; together they contribute an additional $\tfrac12$ to this constant term. The same observation applies to the $1$ in the $(1,1)$ entry, where $\tfrac12$ of it comes from \eqref{dft8}. This allows everything to factor neatly in terms of $D_{\tau_2}$. Concretely, taking the distorted Fourier transform on both sides of equation \eqref{linout2}, \begin{equation}\label{linout3}\begin{aligned}
\Big( - D_{\tau_2  }^2 - \omega_2 D_{\tau_2  } - \xi \Big) \mathcal{F} \wt h & =  \mathcal{F}\widetilde{G} - 2 \omega_2 \mathcal{K}_{nd} D_{\tau_2  } \mathcal{F}\wt h  -  \omega_2 [ \mathcal{K}_{nd}, \mathcal{K}_d]  \mathcal{F}\wt h 
\\ &   + \omega_2^2 \big( \mathcal{K}_{nd}^2 - \mathcal{K}_{nd}\big)  \mathcal{F}\wt h - \dot\omega_2 \mathcal{K}_{nd}  \mathcal{F}\wt h ,
\end{aligned}\end{equation}
where we have denoted by $\widetilde{G}$ the right-hand side of \eqref{linout2}, \[
\wt G := R^{1/2}\bigg( \dfrac{1}{\lambda_2^2}F - \dfrac{4}{R^2}\Big( \cos(2\bfcq_{n-1}) - \cos(2Q_2)\Big) h  \bigg).
\]
We will solve \eqref{linout3} by means of a fixed-point argument, which in turn yields the desired bound for $\ell=0$. Recalling the vector notation coming from the discrete and continuous components of the distorted Fourier transform, namely, \[
\mathcal{F}[\wt h](\xi) = \int_0^\infty \phi(R,\xi) \wt h(R) dR + \langle \phi_{0}, \wt h \rangle_{L^2_{dR}}.
\]
For the sake of notation let us denote by \[
\widehat{h} := \mathcal{F}[\wt h], \qquad \hbox{or,} \hbox{ according to the above identity,} \qquad  \widehat{h} = \begin{pmatrix} \widehat{h}_p \\ \widehat{h}_c
\end{pmatrix}.
\]
We introduce the norm \begin{align*}
    \Vert \widehat{h}\Vert_{X_2} &:= \sup_{0 < t \leq t_0} e^{c q \tau_2} \Big( \vert \widehat{h}_p(\tau_2)\vert + \tau_2 \vert \dot{\widehat{h}}_p(\tau_2)\vert \Big)   + \sup_{0 < t \leq t_0 } e^{c q \tau_2}  \Vert \rho^{1/2} \widehat{h}_c (\tau_2,\cdot) \Vert_{L^2_{d\xi}} 
    \\  & \ \, +  \sup_{0 < t \leq t_0 } \tau_2\, e^{c q \tau_2} \Vert \rho^{1/2} D_{\tau_2} \widehat{h}_c (\tau_2,\cdot) \Vert_{L^2_{d\xi}}  + \sup_{0 < t \leq t_0 } \tau_2 \, e^{c q \tau_2}  \Vert \rho^{1/2} \xi^{1/2}\widehat{h}_c (\tau_2,\cdot) \Vert_{L^2_{d\xi}}  , 
\end{align*}
whereas for the source terms we define \[
\Vert f\Vert_{Y_2} := \sup_{0 < t\leq t_0} \tau_2 \, e^{c  q  \tau_2}\big\vert f_p(\tau_2)\big\vert + \sup_{0 < t\leq t_0} \tau_2 \, e^{c  q  \tau_2} \big\Vert \rho^{1/2} f_c(\tau_2,\cdot) \big\Vert_{L^2_{d\xi}},
\]
with $q<p$ as close to $p$ as desired. Observe that, from Plancherel Theorem for the distorted Fourier Transform, we infer that \[
\sup_{0< t\leq t_0} e^{cq\tau_2}\Vert h\Vert_{\boldsymbol{\mathcal{H}}^1_{rdr}} \lesssim \Vert \mathcal{F}[\wt h]\Vert_{X_2}.
\]
In order to solve equation \eqref{linout3} let us denote by $g_p$ and $g_c$ the discrete and continuous components of the right-hand side of \eqref{linout3}, so for instance \[
\mathcal{F}\wt G = \begin{pmatrix} \langle \wt G, \phi_0\rangle_{L^2_{dR}} 
\\ \langle \wt G, \phi(\cdot,\xi)\rangle_{L^2_{dR}} \end{pmatrix},
\] 
and similarly for the other terms, accounting for the matrix multiplications introduced by the operators $\mathcal{K}$. With these definitions on hand, going back to equation \eqref{linout3}, we are led to study the system \begin{equation}\label{linout4}\begin{aligned}
-\partial_{\tau_2} \Big( \partial_{\tau_2} - \omega_2\Big) \widehat{h}_p & = g_p,
\\ \Big( - D_{\tau_2}^2 - \omega_2 D_{\tau_2} - \xi \Big) \widehat{h}_c & =  g_c.
\end{aligned}\end{equation}
The first equation in \eqref{linout4} can be solved directly by means of variation of constants. Note that the homogeneous solutions are given by \[
\widehat{h}_{p,\text{hom}}= c_1 \lambda_2(\tau_2) + c_2 \lambda_2(\tau_2) \int_{\tau_2}^\infty \dfrac{1}{\lambda_2(t')} dt' =: c_1 \widehat{h}_{p,\text{hom}}^1 + c_2 \widehat{h}_{p,\text{hom}}^2, 
\]
and their Wronskian is exactly $W(\tau_2)=\lambda_2(\tau_2)$. Note that \[
\dfrac{1}{W(\tau_2)}\Big( \widehat{h}_{p,\text{hom}}^1(\tau_2)\widehat{h}_{p,\text{hom}}^2(\sigma_2) - \widehat{h}_{p,\text{hom}}^1(\sigma_2)\widehat{h}_{p,\text{hom}}^2(\tau_2)\Big) = \lambda_2(\tau_2)\int_{\tau_2}^{\sigma_2} \dfrac{ds}{\lambda_2(s)} .
\]
Then, solving backwards from infinity, one easily concludes that the first equation in \eqref{linout4} admits a solution \[
\widehat{h}_p = - \lambda_2(\tau_2) \int_{\tau_2}^\infty \Big( \int_{\tau_2}^{\sigma_2} \dfrac{ds}{\lambda_2(s)}  \Big) \, g(\sigma_2)d\sigma_2. 
\]
We see that for $\tau_2 < \sigma_2$,
\begin{align*}
\bigg \vert \lambda_2(\tau_2) \int_{\tau_2}^{\sigma_2} \dfrac{ds}{\lambda_2(s)} \bigg\vert & \lesssim  \tau_2\ln\Big(\dfrac{\sigma_2}{\tau_2}\Big),
\\ \bigg \vert \partial_{\tau_2} \bigg( \lambda_2(\tau_2) \int_{\tau_2}^{\sigma_2} \dfrac{ds}{\lambda_2(s)} \bigg) \bigg\vert & \lesssim \ln\Big(\dfrac{\sigma_2}{\tau_2}\Big) .
\end{align*}
from where we conclude that the first equation in \eqref{linout4} admits a solution satisfying \begin{align}\label{linout5}
    \sup_{0 < t \leq t_0} e^{c q \tau_2} \Big( \vert \widehat{h}_p\vert + \tau_2 \vert \dot{\widehat{h}}_p\vert \Big) \lesssim \sup_{0 < t \leq t_0} \tau_2 \, e^{cq\tau_2} \big\vert g_p\big\vert ,
\end{align}
Next, we seek to solve the second equation in \eqref{linout4}. To write down its fundamental system of solutions, we first note that this is equivalent to \[
\bigg[ -\Big( \partial_{\tau_2} - 2\omega_2 \xi\partial_\xi\Big)^2  -  \omega_2 \Big( \partial_{\tau_2} - 2\omega_2 \xi\partial_\xi\Big)  - \xi  \bigg] \dfrac{1}{\lambda_2^2}\big( \rho^{1/2}(\xi) \widehat{h}_c\big) = \dfrac{1}{\lambda_2^2} \rho^{1/2}(\xi) g_c,
\]
and hence, changing variables \[
\widehat{h}_c =: \rho^{1/2}(\xi) y(\tau_2,\xi), \qquad \wt g:=\rho^{1/2}(\xi) g_c,
\] 
we can remove the $\rho$ dependent terms in $D_{\tau_2}$ and simplify the second equation in \eqref{linout4}, to obtain \[
\bigg[ -\Big( \partial_{\tau_2} -2 \omega_2 \xi \partial_\xi\Big)^2  - \Big( \partial_{\tau_2} -2 \omega_2 \xi \partial_\xi\Big)  - \xi \bigg] \lambda_2^{-2} y = \lambda_2^{-2} \wt g.
\]
The advantage of this is that now we can easily find the solutions of this equation. Indeed, note that the characteristics associated with the operator on the left-hand side of the last equation are given by $\big(\tau_2, \lambda_2^{-2}(\tau_2)\xi_0\big)$, which is the same as saying that \[
\Big( \partial_{\tau_2} - 2\omega_2 \xi\partial_\xi \big) f(\tau_2,\xi) = \dfrac{d}{d\tau_2} f\big( \tau_2,\xi(\tau_2)\big), \qquad \xi(\tau_2)= \dfrac{1}{\lambda_2^2(\tau_2)}\xi_0.
\]
Consequently, we infer that \[
\bigg( -\partial_{\tau_2}^2 - \omega_2 \partial_{\tau_2} - \dfrac{1}{\lambda_2^2}\xi_0\bigg) \exp\bigg( \pm i\xi_0^{1/2} \int_{\tau_2}^\infty \lambda_2^{-1}(s)ds \bigg) = 0.
\]
By direct calculations, one also checks that their Wronskian is exactly $2i\xi_0^{1/2}\lambda_2^{-1}(\tau_2)$. Therefore, by variation of parameters, we conclude that the second equation in \eqref{linout4} admits a solution of the form \begin{align*}
    \widehat{h}_c(\tau_2,\xi) = \int_{\tau_2}^\infty U\big( \tau_2,\sigma_2,\xi\big)    g_c\Big( \tau_2, \dfrac{\lambda_2^2(\tau_2)}{ \lambda_2^{2}(\sigma_2)}  \xi \Big) d\sigma_2,
\end{align*}
where the Green function is given by  \begin{align*}
    U(\tau_2,\sigma_2,\xi) := \xi^{-1/2} \dfrac{\rho^{1/2}\Big(\tfrac{\lambda_2^2(\tau_2)}{\lambda_2^2(\sigma_2)}\xi\Big) }{\rho^{1/2}(\xi\big)} \, \dfrac{\lambda_2(\tau_2)}{\lambda_2(\sigma_2)} \sin \bigg( \xi^{1/2} \lambda_2(\tau_2) \int_{\tau_2}^{\sigma_2} \lambda_2^{-1}(s)ds\bigg) .
\end{align*}
It is worth noting that \[
\bigg\vert \sin \bigg( \xi^{1/2} \lambda_2(\tau_2) \int_{\tau_2}^{\sigma_2} \lambda_2^{-1}(s)ds\bigg)\bigg\vert \lesssim \min\bigg\{1,\, \xi^{1/2} \lambda_2(\tau_2) \int_{\tau_2}^{\sigma_2} \lambda_2^{-1}(s)ds  \bigg\}
\]
and hence, the propagator satisfies the estimates \begin{align*}
    & \big\vert U(\tau_2,\sigma_2,\xi)\big\vert  \lesssim  \dfrac{\rho^{1/2}\big(\tfrac{\lambda_2^2(\tau_2)}{\lambda_2^2(\sigma_2)}\xi\big) }{\rho^{1/2}( \xi\big)} \, \dfrac{\lambda_2(\tau_2)}{\lambda_2(\sigma_2)}  \, \min\Big\{ \xi^{-1/2}, \, \tau_2 \ln\Big( \dfrac{\sigma_2}{\tau_2} \Big) \Big\} , 
    \\ & \big\vert \partial_{\tau_2} U(\tau_2,\sigma_2,\xi)\big\vert  \lesssim \dfrac{\rho^{1/2}\big(\tfrac{\lambda_2^2(\tau_2)}{\lambda_2^2(\sigma_2)}\xi\big) }{\rho^{1/2}( \xi\big)}  \, \ln\Big( \dfrac{\sigma_2}{\tau_2} \Big) , 
\end{align*}
which in turn imply that \begin{equation}\label{linout6}\begin{aligned}
    & \sup_{0 < t \leq t_0 } e^{c q \tau_2}  \Vert \rho^{1/2} \widehat{h}_c (\tau_2,\cdot) \Vert_{L^2_{d\xi}}   +  \sup_{0 < t \leq t_0 } \tau_2\, e^{c q \tau_2} \Vert \rho^{1/2} D_{\tau_2} \widehat{h}_c (\tau_2,\cdot) \Vert_{L^2_{d\xi}}
    \\ & + \sup_{0 < t \leq t_0 } \tau_2 \, e^{c q \tau_2}  \Vert \rho^{1/2} \xi^{1/2}\widehat{h}_c (\tau_2,\cdot) \Vert \, \lesssim \, \sup_{0 < t \leq t_0 } \tau_2 \, e^{c q \tau_2} \Vert \rho^{1/2} g_c(\tau_2,\cdot) \Vert_{L^2_{d\xi}} .
\end{aligned}\end{equation}
Combining \eqref{linout5} and \eqref{linout6} we conclude that \[
\Vert \widehat{h}\Vert_{X_2} \lesssim \Vert g\Vert_{Y_2}.
\]
With these estimates on hand, we now seek to solve \eqref{linout3}, that is, we must now estimate its right-hand side in the $Y_2$ norm and obtain smallness for those terms depending on $\wt h$. We split the analysis into multiple steps.

\medskip

\noindent
$\bullet$ \textit{Estimate for $\mathcal{F}[\wt G]$}: First, note that \[
\big\Vert \mathcal{F}\wt G\big\Vert_{Y_2} \lesssim \sup_{0 < t\leq t_0} \tau_2 \, e^{c q\tau_2} \lambda_2^{-1} \Vert F(t,\cdot)\Vert_{L^2_{rdr}} + \sup_{0 < t \leq t_0} \tau_2 \, e^{c q \tau_2} \Vert G_1(t,\cdot)\Vert_{L^2_{RdR}} ,
\]
from where we immediately obtain that
\begin{align*}
    \sup_{0 < t \leq t_0} \tau_2 \, e^{c q \tau_2} \lambda_2^{-1} \, \Vert F(t,\cdot)\Vert_{L^2_{rdr}}  \lesssim \sup_{0 < t \leq t_0} e^{c p \tau_2} \lambda_2^{-1} \Vert F(t,\cdot)\Vert_{L^2_{rdr}} ,
\end{align*}
having denoted by $G_1$ the function \begin{align*}
   G_1 := \dfrac{4}{R^2}\Big( \cos(2\bfcq_{n-1}) - \cos(2Q_2)\Big)  \,\wt h.
\end{align*}
By direct manipulations we can write 
\begin{equation}\label{linout8}\begin{aligned}
    & \cos(2\bfcq_{n-1}) - \cos(2Q_2) 
    \\ & = -2\bigg(\sin\Big( \bfq_{n-1}+Q_2\Big) \cos(w_{n-1})+ \cos\Big(\bfq_{n-1}+Q_2\Big)\sin(w_{n-1}) \bigg) \cdot 
    \\ & \cdot  \bigg( \sin\Big(\sum_{j=3}^n(-1)^{j+1}Q_j\Big)\cos(w_{n-1})+\cos\Big(\sum_{j=3}^n (-1)^{j+1}Q_j\Big) \sin(w_{n-1}) \bigg) ,
\end{aligned}\end{equation}
and hence, decomposing $w_{n-1}$ as \[
w_{n-1}=w_{n-1,2}+w_{n-1,3}+...+w_{n-1,n}=:w_{n-1,2}+W,
\]
using the inductive hypothesis \eqref{eq:intro_wn-1_decay_indhyp} for one of the $w_{n-1,2}$, and \eqref{eq:intro-vn-1refined}-\eqref{eq:intro-vn-1boundsrefined} for $W$ and the second $w_{n-1,2}$ in the case we have two factors $w_{n-1,2}$, we conclude that
\begin{align}\label{linout19}
    \bigg\vert \dfrac{4}{R^2}\Big( \cos(2\bfcq_{n-1}) - \cos(2Q_2)\Big) \bigg\vert \lesssim \tau_2^{-2+}.
\end{align}
It follows then that \begin{align*}
    \sup_{0 < t\leq t_0}\tau_2 \, e^{c q \tau_2} \big\Vert G_1 (t,\cdot)\big\Vert_{L^2_{RdR}}  \ll \sup_{0 < t \leq t_0} e^{c q \tau_2} \big\Vert \wt h(t,\cdot) \big\Vert_{L^2_{RdR}} \lesssim \big\Vert \mathcal{F}\wt h\big\Vert_{X_2}.
\end{align*}

\medskip

\noindent
$\bullet$ \textit{Estimate for the terms with at least one occurence of $\mathcal{K}_{nd}$ or $\mathcal{K}_d$}: Estimating these terms follow identically to those in Lemma $4.1$ in \cite{JenKri}, so we omit them.

\medskip

In order to conclude the proof, we now have to show analogous bounds for $S^\ell h$, as well as estimates for the singular weights $r^{-k}$, $k=1,2$. First, we apply the scaling operator 
$S=t\partial_t +r\partial_r$ to equation \eqref{linout7}, from where we obtain that
\begin{align}\label{linout11}
& -\partial_t^2(Sh) + \partial_r^2(Sh) + \dfrac{1}{r}\partial_r(Sh) - \dfrac{4}{r^2}\cos(2Q_2) Sh 
\\ & = 2F +  SF + \dfrac{4}{r^2} S\bigg[ \Big( \cos(2\bfcq_{n-1})-\cos(2Q_2)\Big) h\bigg] + \dfrac{4}{r^2}S \Big( \cos(2Q_2)\Big) h. \nonumber
\end{align}
We think of the last equation as a fixed point for $Sh$. Before going any further, observe that, by direct calculations,
\begin{align}
    & \Big\vert \sin(Q_j)\Big\vert\lesssim \dfrac{\lambda_j^2r^2}{\langle \lambda_j r\rangle^4} ,  \label{linout9}
    \\ & \Big\vert S \sin(Q_j)\Big\vert  \lesssim \dfrac{\lambda_jr^2 (\lambda_j + t \vert \lambda_j'\vert ) }{\langle \lambda_j r\rangle^4}, \nonumber 
    \\ & \Big\vert S \cos(Q_j)\Big\vert  \lesssim \dfrac{\lambda_j^3r^4 (\lambda_j+t \vert \lambda_j'\vert ) }{\langle \lambda_j r\rangle^8}, \nonumber
\end{align}
for all $j\geq 1$. Then, for $j\in\{2,...,n\}$, $k\in\{3,...,n\}$, with $j\leq k$, \[
\bigg\vert \dfrac{1}{r^2} S\Big[ \sin(Q_j)\sin(Q_k) \Big] \bigg\vert \lesssim \dfrac{t \lambda_j\vert \lambda_j'\vert \lambda_k^2 r^2}{ \langle \lambda_j r\rangle^4 \langle \lambda_k r\rangle^4} \lesssim t \lambda_{j+1} \lambda_k^2.
\] 
An obvious but important observation is that whenever $S$ acts on one of these trigonometric functions (each of which satisfies slightly worse bounds compared to \eqref{linout9}), then such a term has a multiplicative factor $h$ in front, rather than $Sh$. Similarly, whenever $Sh$ appears on the right-hand side, then we have the stronger bound \eqref{linout9}. More concretely, the terms depending on $h$ on the right-hand can either be bound by \[
t\lambda_2^2 \lambda_3\vert h\vert \qquad \hbox{ or } \qquad  \lambda_2^2 \vert Sh\vert.
\]
This better bound for the term on the right-hand side involving $Sh$, together with the fact that we only require a decay of order $\tau^{-(p-\varepsilon)}$ for $Sh$, see \eqref{linout10}, is what enables us to close the fixed-point argument for $Sh$.  Up to obvious modifications, the same observation holds if we apply $S^\ell$ to the equation, with $\ell \geq2$. 

Indeed, let us call $H(t)$ all the terms on the right-hand side of \eqref{linout11} except the one containing $Sh$. Thus, according to \eqref{linout8} as well as the already proven bounds for $h$ in the case $\ell=0$, we conclude that  
\begin{align*}
    & \sup_{0 < t \leq t_0} \tau_2 \,  e^{cq\tau_2} \lambda_2^{-1} \Vert H(t,\cdot)\Vert_{L^2_{rdr}} 
    \\ & \lesssim \sup_{0 < t\leq t_0} e^{cp \tau_2}\Vert F(t,\cdot) \Vert_{L^2_{rdr}} + \sup_{0 < t\leq t_0} e^{cp \tau_2}\Vert S F(t,\cdot)\Vert_{L^2_{rdr}}.
\end{align*}
We can now repeat the same argument as before, with $Sh$ taking the role of $h$, $H$ that of $F$, the term depending on $Sh$ on the right-hand side of \eqref{linout11} corresponding to the last term on the right-hand side of \eqref{linout7}, and with $\tilde{q}:=q-\varepsilon<q$ chosen arbitrarily close to $q$. Therefore, by induction on $\ell$, we conclude that \[
\big\Vert S^\ell h\Vert_{\boldsymbol{\mathcal{H}}^1_{rdr}}  \lesssim_{\ell,\varepsilon} \tau^{-(p-\ell \varepsilon)}.
\]
Finally, it only remains to prove the weighted estimates \eqref{linout12}. This is proven by a direct estimate and no fixed point argument is needed here (in particular no smallness is needed, in contrast with the previous steps). We prove this only for the case $\ell=0$. The general case $\ell \geq1$ then follows by differentiating the equation as before. To this end, we first recall the local elliptic estimate near $r=0$ in \cite{KST2}, namely, \begin{align}\label{linout15}
\Vert r^{-2}h\Vert_{L^2_{rdr}} \lesssim \Big\Vert \dfrac{1}{t} \partial_rh\Big\Vert_{L^2_{rdr}}+\Big\Vert \dfrac{1}{t r}h\Big\Vert_{L^2_{rdr}}+\Big\Vert \Big( \dfrac{t^2-r^2}{t^2}\partial_r^2+\dfrac{1}{r}\partial_r-\dfrac{4}{r^2}\Big)h\Big\Vert_{L^2_{rdr}}.
\end{align}
Now, to express the last term in the estimate above in a more convenient way, we write $t\partial_t = S- r \partial_r$, and hence \[
t^2\partial_t^2 h = -2t \partial_t h - S^2 h  + Sh + r^2 \partial_r^2 h + 2 t \partial_t S h,
\]
which in turn implies that \begin{align*}
\bigg( \dfrac{t^2-r^2}{t^2} \partial_r^2 + \dfrac{1}{r}\partial_r  - \dfrac{4}{r^2} \bigg) h & = \dfrac{1}{t^2}\bigg( -2t \partial_t h - S^2 h  + Sh  + 2 t \partial_t S h \bigg)
\\ & + \bigg(-\partial_t^2+\partial_r^2+\dfrac{1}{r}\partial_r - \dfrac{4}{r^2}\bigg) h.
\end{align*}
From this last identity it follows that \begin{align*}
    \dfrac{1}{\lambda_2}\Big\Vert \Big( \dfrac{t^2-r^2}{t^2} \partial_r^2 + \dfrac{1}{r}\partial_r  - \dfrac{4}{r^2}\Big) h\Big\Vert_{L^2_{rdr}} & \lesssim \dfrac{1}{t^2\lambda_2^2} \Big( \lambda_2 \Vert S^2h\Vert_{L^2_{rdr}} + \lambda_2 \Vert Sh\Vert_{L^2_{rdr}} \Big)
    \\ & + \dfrac{1}{t \lambda_2}\Big( \Vert \partial_t S h\Vert_{L^2_{rdr}} + \Vert \partial_t h\Vert_{L^2_{rdr}}\Big)
    \\ & + \dfrac{1}{\lambda_2}\Big\Vert \Box h - \dfrac{4}{r^2}h\Big\Vert_{L^2_{rdr}},
\end{align*}
where we recall that $\Box h = -\partial_t^2h+\partial_r^2h+\tfrac1r\partial_rh$. Moreover, from  \eqref{linout1} and the already proven bounds we conclude that \begin{align*}
    \dfrac{1}{\lambda_2(t)}\Big\Vert \Box h - \dfrac{4}{r^2}h\Big\Vert_{L^2_{rdr}} & \lesssim  \dfrac{1}{\lambda_2(t)}\Vert F\Vert_{L^2_{rdr}} + \dfrac{1}{\lambda_2(t)}\Big\Vert \dfrac{4}{r^2}\Big(\cos(2\bfcq_{n-1})-1\Big)h\Big\Vert_{L^2_{rdr}}
    \\ & \lesssim \dfrac{1}{\lambda_2(t)} \Vert F\Vert_{L^2_{rdr}} + \lambda_2(t) \Vert h\Vert_{L^2_{rdr}} \lesssim \lambda_2(t) \tau^{-p}
\end{align*}
Therefore, \begin{align}\label{linout13}
    \Big\Vert \Big( \dfrac{t^2-r^2}{t^2} \partial_r^2 + \dfrac{1}{r}\partial_r  - \dfrac{4}{r^2}\Big) h\Big\Vert_{L^2_{rdr}} \lesssim \tau^{-(p-\varepsilon)}.
\end{align}
On the other hand, we recall the elementary bound (see for instance \cite{KST3}) \begin{align}\label{linout17}
\Vert \partial_r h\Vert_{L^2_{rdr}} + \Vert r^{-1}h\Vert_{L^2_{rdr}} \lesssim \Vert \mathcal{L}_{t}^{1/2} h\Vert_{L^2_{rdr}} + \lambda_2(t)\Vert h\Vert_{L^2_{rdr}}.
\end{align}
and hence \begin{align}\label{linout14}
    \Vert t^{-1}\partial_r h\Vert_{L^2_{rdr}}+\Vert t^{-1}r^{-1} h\Vert_{L^2_{rdr}} \lesssim \tau^{-(p-\varepsilon)}.
\end{align}
Plugging \eqref{linout13} and \eqref{linout14} into \eqref{linout15} we conclude the desired estimate. The $L^4_{rdr}$ bound in \eqref{linout12} follows from interpolating the previous bound and the trivial bound \[
\Vert h\Vert_{L^\infty} \lesssim \Vert \partial_r h\Vert_{L^2_{rdr}} + \Vert \dfrac{1}{r}h\Vert_{L^2_{rdr}} \lesssim \tau^{-(p-\varepsilon)}.
\]
The proof is complete.
\end{proof}

In order to recover the Taylor type decomposition \eqref{eq:intro-vn-1refined}, we shall take advantage of the following lemma, which is based on commuting vector field methods
\begin{lem}\label{lem:Tayloerexpansion} Let $\epsilon$ solve the equation 
\begin{equation}\label{eq:epsilonwave10}
-\epsilon_{tt} + \epsilon_{rr} + \frac{1}{r}\epsilon_r - \frac{4\cos \bfcq_{n-1}}{r^2}\epsilon = f. 
\end{equation}
on the region $(0, t_0]\times \mathbb{R}^2$, and satisfy the conclusions of the preceding lemma. 
Assume that $\bfcq_{n-1}$ satisfies the inductive assumptions encapsulated by \eqref{eq:intro_wn-1_decay_indhyp}, \eqref{eq:intro-vn-1refined}, and \eqref{eq:intro-vn-1boundsrefined}, and that the source term $f$ satisfies the bounds
\begin{align*}
&\big\|S^kf(t,\cdot)\big\|_{L^2_{r\,dr}}\lesssim_k \tau^{-2+}, \\
&\big\|\chi_{r\ll t}S^kf(t,\cdot)\big\|_{L^\infty_{r\,dr}}\lesssim_k \tau^{-1+}, \\
&\big\|r^{-2}\chi_{r\ll t}S^kf(t,\cdot)\big\|_{L^\infty_{r\,dr}}\lesssim_k \tau,
\end{align*}
where $\tau_2 = \int_{t}^{t_0}\lambda_2(s)\,ds,\,\tau = e^{c\tau_2}$. Then have the representation 
\[
\epsilon(t, r) = c(t)r^2 + g(t, r)
\]
where we have the bounds 
\[
\big|(t\partial_t)^kc(t)\big|\lesssim_k \tau^{-1+},\,\big|r^{-4} S^kg(t, r)\big|\lesssim_k\tau. 
\]
\end{lem}
\begin{proof} This proceeds in several steps:
\\

{\it{(1): We have the bound 
\begin{equation}\label{eq:Htwobound1}
\big\|\epsilon_{rr}\big\|_{L^2_{r\,dr}(r\ll t)} + \big\|\frac{\epsilon_r}{r}\big\|_{L^2_{r\,dr}}\lesssim \tau^{-2+},
\end{equation}
and similarly for $S^k\epsilon$.}} Taking advantage of the preceding lemma, we see that it suffices to establish the estimate
\[
\big\|\big(\chi_{r\ll t}\epsilon\big)_{rr}\big\|_{L^2_{r\,dr}} + \big\|\frac{\big(\chi_{r\ll t}\epsilon\big)_r}{r}\big\|_{L^2_{r\,dr}}\lesssim \tau^{-2+}.
\]
for a smooth cutoff function $\chi_{r\ll t}$. For this consider the equation 
\begin{equation}\label{eq:trickidentity1}\begin{split}
\Big(\frac{t^2 - r^2}{t^2}\partial_{rr} + \frac{1}{r}\partial_r - \frac{4}{r^2}\Big)\big(\chi_{r\ll t}\epsilon\big) &= t^{-2}\cdot\chi_{r\ll t}\Big({-}S^2\epsilon + 2t\partial_t S\epsilon + S\epsilon - 2t\partial_t\epsilon\Big)\\
& +\chi_{r\ll t}\big[ \Box\epsilon - \frac{4\epsilon}{r^2}\big]\\
& + \Big[\Big(\frac{t^2 - r^2}{t^2}\partial_{rr} + \frac{1}{r}\partial_r - \frac{4}{r^2}\Big), \chi_{r\ll t}\Big]\epsilon\\
&=:(I) + (II) + (III).
\end{split}\end{equation}
Due to the identity (for compactly supported $g$)
\begin{align*}
\int_0^\infty \big(g_{rr} + \frac{1}{r}g_r\big)^2 r\,dr = \int_0^\infty\big(g_{rr}^2 + \frac{g_r^2}{r^2}\big)r\,dr +  2\int_0^\infty g_{rr}g_r\,dr =  \int_0^\infty\big(g_{rr}^2 + \frac{g_r^2}{r^2}\big)r\,dr,  
\end{align*}
we deduce that 
\begin{align*}
\big\|\frac{r^2}{t^2}\big(\chi_{r\ll t}\epsilon\big)_{rr}\big\|_{L^2_{r\,dr}}\leq c\big\|\big(\partial_{rr} + \frac{1}{r}\partial_r\big)\chi_{r\ll t}\epsilon\big\|_{L^2_{r\,dr}},
\end{align*}
for a constant $c\ll 1$. Then we deduce from \eqref{eq:trickidentity1} that 
\begin{align*}
\big\|\big(\partial_{rr} + \frac{1}{r}\partial_r\big)\chi_{r\ll t}\epsilon\big\|_{L^2_{r\,dr}}&\leq c\big\|\big(\partial_{rr} + \frac{1}{r}\partial_r\big)\chi_{r\ll t}\epsilon\big\|_{L^2_{r\,dr}} + \big\|\frac{1}{r^2}\chi_{r\ll t}\epsilon\big\|_{L^2_{r\,dr}} + \big\|(I)\big\|_{L^2_{r\,dr}} \\
& + \big\|(II)\big\|_{L^2_{r\,dr}} + \big\|(III)\big\|_{L^2_{r\,dr}}, 
\end{align*}
as well as 
\[
\big\|\big(\chi_{r\ll t}\epsilon\big)_{rr}\big\|_{L^2_{r\,dr}} + \big\|\frac{\big(\chi_{r\ll t}\epsilon\big)_r}{r}\big\|_{L^2_{r\,dr}}\lesssim \big\|\big(\partial_{rr} + \frac{1}{r}\partial_r\big)\chi_{r\ll t}\epsilon\big\|_{L^2_{r\,dr}}
\]
Since $\big\|\frac{1}{r^2}\chi_{r\ll t}\epsilon\big\|_{L^2_{r\,dr}}\lesssim \tau^{-2+}$, the proof of \eqref{eq:Htwobound1} reduces to bounding $\big\|(I), (II), (III)\big\|_{L^2_{r\,dr}}$
Thanks to the preceding lemma, we infer the estimate 
\[
 \big\|(III)\big\|_{L^2_{r\,dr}}\lesssim \big\|\frac{\epsilon}{t^2}\big\|_{L^2_{r\,dr}} +  \big\|\frac{\epsilon_r}{t}\big\|_{L^2_{r\,dr}}\lesssim \tau^{-2+}, 
\]
For $(I)$, we observe that 
\begin{align*}
 \big\|(I)\big\|_{L^2_{r\,dr}}\lesssim t^{-2}\big\|S^2\epsilon\big\|_{L^2_{r\,dr}} + t^{-1}\big\|\partial_t S\epsilon\big\|_{L^2_{r\,dr}} +  t^{-1}\big\|\partial_t \epsilon\big\|_{L^2_{r\,dr}}\lesssim \tau^{-2+}.
\end{align*}
Finally, for $(II)$, we estimate it by 
\begin{align*}
&\Big\|\chi_{r\ll t}\big[ \Box\epsilon - \frac{4\epsilon}{r^2}\big]\Big\|_{L^2_{r\,dr}}\\&\leq \Big\|\chi_{r\ll t}\big[ \Box\epsilon - \frac{4\cos(\bfcq_{n-1})\epsilon}{r^2}\big]\Big\|_{L^2_{r\,dr}} + \Big\|\chi_{r\ll t}\frac{4(\cos(\bfcq_{n-1})-1)\epsilon}{r^2}\Big\|_{L^2_{r\,dr}}. 
\end{align*}
The assumptions on $\epsilon, f$ and the preceding lemma allow us to bound the last expression by $\lesssim \tau^{-2+}$. The analogous bounds for $S^k\epsilon$ follow by applying $S^k$ to the equation satisfied by $\epsilon$ and using induction on $k$.
\\

{\it{(2): Proof of the estimate $\big|c(t)\big|\lesssim \tau^{-1+}$.}} Write \eqref{eq:trickidentity1} in the form 
\[
\Big(\partial_{rr} + \frac{1}{r}\partial_r - \frac{4}{r^2}\Big)\big(\chi_{r\ll t}\epsilon\big) = \phi, 
\]
which we solve via variation of constants. It follows that 
\begin{equation}\label{eq:epsilonformula}
\chi_{r\ll t}\epsilon = c(t)r^2 +  \frac14 r^2\int_0^r s^{-1}\phi(s)\,ds - \frac14 r^{-2}\int_0^r s^{3}\phi(s)\,ds, 
\end{equation}
where $\phi$ is the sum of the terms on the right hand side in \eqref{eq:trickidentity1}  and the term $\frac{r^2}{t^2}\big(\chi_{r\ll t}\epsilon\big)_{rr}$. We next show the bound 
\begin{equation}\label{eq:ccontrol1}
\big|\frac14 \int_0^r s^{-1}\phi(s)\,ds - \frac14 r^{-4}\int_0^r s^{3}\phi(s)\,ds\big|\lesssim \tau^{-1+},\,r\lesssim t. 
\end{equation}
Assuming this bound, we deduce from \eqref{eq:epsilonformula} that 
\begin{align*}
\big\|c(t)\cdot 1\big\|_{L^2_{r\,dr}(r\lesssim t)}\lesssim \big\|r^{-2}\chi_{r\ll t}\epsilon\big\|_{L^2_{r\,dr}} + t\cdot \tau^{-1+}, 
\end{align*}
which results in 
\[
|c(t)|\lesssim t^{-1}\tau^{-1+} + \tau^{-1+}\lesssim \tau^{-1+},
\]
implying $\big|c(t)\big|\lesssim \tau^{-1+}$. \\
In order to establish \eqref{eq:ccontrol1}, we estimate the contributions of the various constituents of $\phi$,  starting with $\frac{r^2}{t^2}\big(\chi_{r\ll t}\epsilon\big)_{rr}$. This leads to the term 
\begin{align*}
&\frac{1}{4t^2} \int_0^r s\big(\chi_{s\ll t}\epsilon\big)_{ss}\,ds - \frac{1}{4t^2} r^{-4}\int_0^r s^{5}\big(\chi_{s\ll t}\epsilon\big)_{ss}\,ds\\
& = \frac{1}{t^2}\chi_{r\ll t}\epsilon-  \frac{5}{t^2} r^{-4}\int_0^r s^{3}\chi_{s\ll t}\epsilon\,ds
\end{align*}
The absolute value of this term is bounded by $\lesssim t^{-2}\tau^{-2+}\lesssim \tau^{-2+}$, which is better than what we need. \\
As for the contribution of the term $(III)$ in \eqref{eq:trickidentity1}, we can bound it in absolute value by 
\[
\lesssim t^{-2}\big|\tilde{\chi}_{r\sim t}\epsilon\big| + t^{-1}\big|\tilde{\chi}_{r\sim t}\epsilon_r\big|
\]
where $\tilde{\chi}$ is a suitable smooth cutoff function localising to the indicated region. 
It follows that the corresponding contribution to \eqref{eq:ccontrol1} is bounded in absolute value by $\lesssim \tau^{-2+}$ as well. 
\\
Next, we consider the contribution of the term $(I)$ in \eqref{eq:trickidentity1}. Specifically, we shall treat the contribution of the term $t^{-1}\chi_{r\ll t}\partial_t(S\epsilon)$, the other terms there being handled similarly. Write
\[
t^{-1}\partial_t(S\epsilon) = t^{-2}S^2\epsilon - t^{-2}r\partial_r(S\epsilon). 
\]
Then we get 
\begin{align*}
&\big|\frac{1}{4t^2} \int_0^r s^{-1}S^2\epsilon\,ds - \frac{1}{4t^2} r^{-4}\int_0^r s^{3}S^2\epsilon\,ds\big|\\
&\lesssim t^{-1}\big\|\frac{S^2\epsilon}{r}\big\|_{L^\infty(r\ll t)} + t^{-2}\big\|\frac{S^2\epsilon}{r}\big\|_{L^2(r\ll t)} 
\end{align*}
The first term on the last line in turn is bounded by 
\begin{align*}
t^{-1}\big\|\frac{S^2\epsilon}{r}\big\|_{L^\infty(r\ll t)}\lesssim  t^{-1}\big\|\frac{S^2\epsilon}{r^2}\big\|_{L^2_{r\,dr}(r\ll t)} + t^{-1}\big\|\partial_r\big(\frac{S^2\epsilon}{r}\big)\big\|_{L^2_{r\,dr}(r\ll t)}\lesssim \tau^{-2+},
\end{align*}
thanks to {\it{(1)}}. As for the contribution of $ t^{-2}r\partial_r(S\epsilon)$ to the variation of constants formula, it equals 
\begin{align*}
&\frac{1}{4t^2} \int_0^r \partial_s(S\epsilon)\,ds - \frac{1}{4t^2} r^{-4}\int_0^r s^{4}\partial_s(S\epsilon)\,ds = \frac{1}{t^2} r^{-4}\int_0^r s^{3}S\epsilon\,ds\lesssim t^{-2}\big\|S\epsilon\big\|_{L^\infty(r\lesssim t)}, 
\end{align*}
which in turn is bounded by $\lesssim \tau^{-2+}$. 
\\
It remains to deal with the contribution of the term $(II)$. Writing
\[
\chi_{r\ll t}\big[ \Box\epsilon - \frac{4\epsilon}{r^2}\big] = \chi_{r\ll t} f + \chi_{r\ll t}\frac{4(\cos(\bfcq_{n-1})-1)\epsilon}{r^2},
\]
For the latter term, we use the estimate 
\[
\big|r^{-1}\chi_{r\ll t}\frac{4(\cos(\bfcq_{n-1})-1)\epsilon}{r^2}\big|\lesssim \lambda_2^2 r\big|\frac{\epsilon}{r^2}\big|.
\]
We conclude that (letting $\chi_{r\ll t}\frac{4(\cos(Q_2)-1)\epsilon}{r^2}=:\psi$)
\begin{align*}
\big|\frac14 \int_0^r s^{-1}\psi(s)\,ds - \frac14 r^{-4}\int_0^r s^{3}\psi(s)\,ds\big|&\lesssim \big\|\lambda_2^2\cdot 1\big\|_{L^2_{s\,ds}(s\lesssim t)}\cdot \big\|\frac{\epsilon}{s^2}\big\|_{L^2_{s\,ds}}\\
&\lesssim (\lambda_2^2 t)\cdot \tau^{-2+}\lesssim \tau^{-2+}. 
\end{align*}
As for the contribution of the inhomogeneous source term $\chi_{r\ll t} f$, we use that  the assumptions of the lemma imply the estimate 
\[
\big\|r^{0-}\chi_{r\ll t}f\big\|_{L^\infty_{r\,dr}}\lesssim \tau^{-1+}.
\]
It follows that 
\[
\big|\frac14 \int_0^r s^{-1}f(s)\,ds - \frac14 r^{-4}\int_0^r s^{3}f(s)\,ds\big|\lesssim \tau^{-1+}. 
\]
This completes the proof of \eqref{eq:ccontrol1}. 
\\

{\it{(3): We have the estimate $\big\|\frac{\epsilon}{r^2}\big\|_{L^\infty_{r\,dr}}\lesssim \tau^{-1+}$.}} In fact, this follows from \eqref{eq:epsilonformula}, \eqref{eq:ccontrol1}, as well as {\it{(2)}}. 
\\

{\it{(4): We have the estimate $\big\|\frac{S^k\epsilon}{r^2}\big\|_{L^\infty_{r\,dr}}\lesssim_k \tau^{-1+}$.}} This follows by applying $S^k$ to the equation for $\epsilon$ and repeating the steps {\it{(1) - (3)}} for $S^k\epsilon$, using induction on $k$.
\\

{\it{(5): We have the bound 
\begin{equation}\label{eq:sharpgbound}
\big|\frac14 \int_0^r s^{-1}\phi(s)\,ds - \frac14 r^{-4}\int_0^r s^{3}\phi(s)\,ds\big|\lesssim \tau r^2,\,r\lesssim t. 
\end{equation}
}}
To accomplish this, we revisit the argument for \eqref{eq:ccontrol1}. Precisely, we estimate the contributions of the various terms constituting $\phi$, but now also taking advantage of {\it{(3), (4)}}. Repeating the steps in {\it{(2)}}, we start with the contribution of  $\frac{r^2}{t^2}\big(\chi_{r\ll t}\epsilon\big)_{rr}$ to the variations of constants formula, which we saw can be re-expressed in the form
\begin{align*}
&\frac{1}{t^2}\chi_{r\ll t}\epsilon-  \frac{5}{t^2} r^{-4}\int_0^r s^{3}\chi_{s\ll t}\epsilon\,ds\\
& = r^2\big[\frac{1}{t^2}\chi_{r\ll t}r^{-2}\epsilon-  \frac{5}{t^2} r^{-6}\int_0^r s^{5}\chi_{s\ll t}s^{-2}\epsilon\,ds]
\end{align*}
Using {\it{(4)}}, we can bound this by $\lesssim r^2\tau^{-1+}$.  
\\
Continuing with the contribution of the term $(III)$, we can exploit that it is supported at $r\sim t$, which means applying $r^{-2}$ in front of the variation of constants expression will cost at most $\lesssim t^{-2}\lesssim \tau^{0+}$. Hence this contribution is also bounded by $\lesssim r^2\tau^{-1+}$. 
\\
Next, we consider the contribution of the term $(I)$. Again proceeding as in {\it{(2)}}, we can estimate 
\begin{align*}
&\big|\frac{1}{4t^2} \int_0^r s^{-1}S^2\epsilon\,ds - \frac{1}{4t^2} r^{-4}\int_0^r s^{3}S^2\epsilon\,ds\big|\\
&\lesssim r^2\cdot \big\|\frac{S^2\epsilon}{r^2}\big\|_{L^\infty_{r\,dr}}\lesssim r^2\tau^{-1+},
\end{align*}
thanks to {\it{(4)}}. Further, we find that 
\begin{align*}
\frac{1}{4t^2} \int_0^r \partial_s(S\epsilon)\,ds - \frac{1}{4t^2} r^{-4}\int_0^r s^{4}\partial_s(S\epsilon)\,ds = \frac{1}{t^2} r^{-4}\int_0^r s^{3}S\epsilon\,ds&\lesssim \frac{r^2}{t^{2}}\big\|\frac{S\epsilon}{r^2}\big\|_{L^\infty(r\lesssim t)}\\
&\lesssim r^2\tau^{-1+},
\end{align*}
again due to {\it{(2)}}. 
Finally, we turn to the contribution of $(II)$, and more specifically, the contribution of the terms 
\[
\chi_{r\ll t} f,\, \chi_{r\ll t}\frac{4(\cos(\bfcq_{n-1})-1)\epsilon}{r^2}
\]
to the variation of constants formula. For the latter, which we call $\psi_1$, we observe that 
\[
\big|\psi_1\big|\lesssim r^2\lambda_2^2\cdot \big|\frac{\epsilon}{r^2}\big|\lesssim r^2\tau^{-1+},
\]
which easily results in the desired bound 
\begin{align*}
&\big|\frac{1}{4t^2} \int_0^r s^{-1}\psi_1\,ds - \frac{1}{4t^2} r^{-4}\int_0^r s^{3}\psi_1\,ds\big|\lesssim r^2\tau^{-1+}.
\end{align*}
For the former term, the desired bound for this contribution is a direct consequence of the third bound for $f$ stated in the assumptions of this lemma. 
\\

In conjunction with \eqref{eq:epsilonformula} and {\it{(2)}}, this entails that we have the decomposition 
\[
\epsilon(t, r) = c(t)r^2 + g(t, r)
\]
and the bound stated at the end of the lemma with $k = 0$. 
\\

{\it{(6): Higher order bounds for $c, g$}}. Applying $S^k$ to the equation for $\epsilon$, we indictively derive analogous decompositions
\[
S^k\epsilon(t, r) = c_k(t)r^2 + g_k(t, r)
\]
satisfying 
\[
\big|c(t)\big|\lesssim_k \tau^{-1+},\,\big|g_k(t, r)\big|\lesssim_k r^4\tau.  
\]
Then we take advantage of the identity
\[
c_k(t) = (t\partial_t +2)^kc(t)
\]
to derive the higher order derivative bounds for $c$, while those for $g$ follow from 
\[
S^k g = g_k. 
\]

\end{proof}

%%%%%%%%%%%%%%%%%%%%%%
%%%%%%%%%%%%%%%%%%%%%%
%%%%%%%%%%%%%%%%%%%%%%
%%%%%%%%%%%%%%%%%%%%%%

%\newpage
\bigskip
\section{Construction of accurate approximate solution}\label{sec:construction_acc}

In this section, our point of departure is the $n-1$-bubble solution $\bfcq_{n-1}$, constructed on a time interval $(0, t_0]$ for some small $t_0>0$, which we assume to be of the form \eqref{eq:outerbubblesoln}. Here we assume that the scales $\lambda_j$, $2\leq j\leq n$, $n\geq 3$, satisfy \eqref{eq:lambdandef}, as well as the bounds stated in Proposition \ref{prop:exponentialtowers1}
\begin{equation}\label{eqlambdajassumptions}
\lambda_j(t)\geq e^{c\int_t^{t_0}\lambda_{j+1}(s)\,ds},\quad \big|\lambda_j^{(k)}\big|\leq E_{k,\nu,n} \lambda_j^{1+k\nu},\quad 2\leq j<n,\quad \nu>0,
\end{equation}
and $\nu>0$ can be chosen arbitrarily. 
We furthermore assume that the radiation part $w_{n-1}$ satisfies the bounds 
\begin{equation}\label{eq:vn-1bounds}\begin{split}
&\sum_{k=0}^K\big\|\nabla_{t,r}S^k w_{n-1}\big\|_{L^2_{r\,dr}(r\lesssim t)}\lesssim_K |\log t|^{-1},\quad \, S = t\partial_t + r\partial_r,\\
&\sum_{k=1}^K\big\|\frac{S^k w_{n-1}(t,\cdot)}{r^2}\big\|_{L^\infty_{r\,dr}(r\lesssim t)}\lesssim_K \lambda_2.
\end{split}\end{equation}
Furthermore, we require there to be a decomposition 
\begin{equation}\label{eq:vn-1refined}
w_{n-1}(t, r) = c_{n-1}(t)r^2 + \tilde{w}_{n-1}(t,r),
\end{equation}
where we have the bounds (throughout $t\in (0, t_0]$ with $t_0\ll 1$)
\begin{equation}\label{eq:vn-1boundsrefined}\begin{split}
&\big|(t\partial_t)^kc_{n-1}(t)\big|\lesssim_k \lambda_2^{1+k \varepsilon}\\
&\big\|r^{-4}S^k \tilde{w}_{n-1}\big\|_{L^\infty_{r\,dr}(r\lesssim t)}\lesssim_k \lambda_2^{3+k \varepsilon}.  
\end{split}\end{equation}
We note here that the very weak decay rate of the energy type norm in the first inequality of \eqref{eq:vn-1bounds} is due to the contribution of the outermost bubble solution. On the other hand, the scaling parameter $\lambda_2$ occurring in the remaining inequalities is the dominant one, corresponding to the innermost bubble in $\bfcq_{n-1}$. 
\\

Our goal now is the following 
\begin{prop}\label{prop:approximateinnerbubble} Given $N\geq 1$, there exists $t_1 = t_1(N,\beta, n)\in (0, t_0]$ and a positive function 
\[
\lambda_1\in C^\infty\big((0, t_1]\big)
\]
satisfying the bounds 
\[
\lambda_1\geq e^{c\int_t^{t_0}\lambda_{2}(s)\,ds},\qquad \big|\lambda_1^{(k)}\big|\leq F_{k,\nu,n, N} \lambda_1^{1+k\nu},\qquad \big|\lambda_1^{(k)}\big|\leq E_{k,n,N}\bar{\lambda}_2^k\lambda_1,
\]
where $\nu>0$ can be chosen arbitrarily, and $c = c(\beta, n, N)>0$ is a suitable constant, and such that the following holds: setting 
\begin{align}\label{eq:construction_app_tau1}
\tau_1: = \int_t^{t_1}\lambda_1(s)\,ds, 
\end{align}
there exists an approximate solution 
\[
 u_N(t, r) = Q\big(\lambda_1(t)r\big) - \bfcq_{n-1}(t, r) + v_N(t, r)
\]
satisfying the bounds
 \begin{align*}
&\big\|{-}u_{N, tt} + u_{N, rr} + \frac{1}{r}u_{N,r} - 2\frac{\sin (2u_N)}{r^2}\big\|_{L^2_{r\,dr}(r\lesssim t)}\lesssim \tau_1^{-N},\\
&\big\|S^k v_N\big\|_{H^1_{rdr}}\lesssim_k \tau_1^{-2+},\qquad k\geq 0. 
\end{align*}
Moreover, there is a splitting 
\[
v_N = c_N(t)r^2 + \tilde{v}_N
\]
such that we have the bounds
\begin{equation}\label{eq:vNbounds}\begin{split}
&\big|(t\partial_t)^kc_N(t)\big|\lesssim_k \lambda_1^{1+}\\
&\big\|S^k\tilde{v}_{N}\big\|_{L^\infty_{r\,dr}(r\lesssim t)}\lesssim_k \lambda_1^{3+}.  
\end{split}\end{equation}
\end{prop}

\begin{proof}[Proof of Proposition \ref{prop:approximateinnerbubble}] Let $t_1>0$ be a small number, whose size will eventually be chosen small enough depending on $n, \beta, N$. In order to construct $\lambda_1, v_N$, we shall rely on an auxiliary function $m\in C^0\big((0, t_1]\big)$, which we shall choose at the end, and which will determine $\lambda_1$ via Lemma~\ref{lem:simpleinductivelambdaonewithm}; note that $\overline{\lambda}_2$ in turn is determined by
\begin{equation}\label{eq:barlambdatwo}
\overline{\lambda}_2 = \frac{4}{\sqrt{\pi}}\cdot \Big[\sqrt{\sum_{j=2}^n(-1)^j\lambda_j^2 + c_{n-1}}\Big].
\end{equation}
 For this we shall require throughout that  
\begin{equation}\label{eq:basicmbound}
\big|m(t)\big|\leq \tau_1^{-\frac12},\,t\in (0, t_1]. 
\end{equation}
For now we leave $m$ undetermined as an additional parameter, only requiring the preceding bound in addition to higher derivative bounds specified later on, and determine $\lambda_1$ in accordance with Lemma~\ref{lem:simpleinductivelambdaonewithm}.
The correction $v_{N}$ shall result from adding $N$ increments to a first approximation $h_0$, i.e. we shall set 
\[
v_N = \sum_{j=0}^{N}h_j. 
\]
In order to motivate the choice of the $h_j$, we first recall the equation which $v_N$ would have to satisfy in order to result in an exact solution of \eqref{eq:keq2corotational}, given by \eqref{eq:veqn}, where we need to replace $v$ by $v_N$. 
\\

{\bf{Step 0}}: {\it{Construction of the correction $h_0$.}} Here, we replace the wave operator on the left of \eqref{eq:veqn} by the simplified ``inner" elliptic operator 
\[
\partial_{rr} + \frac{1}{r}\partial_r - \frac{4\cos\big(2Q_1\big)}{r^2}. 
\]
Furthermore, we only retain the term $E_2$ on the right, and we introduce a further correction term, depending on the parameter $m$, which will be required to compensate for other corrections later in the inductive procedure. We then arrive at the following equation for $h_0$ (recall $E_2$ was defined in \eqref{eq:veqn}): 
\begin{equation}\label{eq:hzeroeqn}
 h_{0, rr} + \frac{1}{r}h_{0, r} - \frac{4\cos\big(2Q_1\big)}{r^2}h_0 = E_2 - 8(2m\overline{\lambda}_2 + m^2)\cdot \big[\cos(2Q_1) - 1\big]
\end{equation}
We stress that this step is not part of the inductive process we shall perform later.
The base case of the inductive process will be the construction of the correction $h_1$ (see Step 1). This isolated Step 0 aims specifically to cancel $E_2$, which, as mentioned before, contains the largest error terms involving the strong interactions between the bubbles as well as the fact that $Q_1$ is not a solution.

In analogy to \cite{JenKri}, we have the following 
  \begin{lem}\label{lem:hzerocorrection} The equation \eqref{eq:hzeroeqn} admits a solution $h_0\in H^3_{r\,dr}$ satisfying the uniform bound 
  \begin{align*}
  \big|h_0\big|\lesssim \tau_1^{-2+},\qquad r\lesssim t. 
  \end{align*}
  Denoting the scaling vector field $S = t\partial_t + r\partial_r$, we have the uniform bounds 
  \begin{align*}
  \big\|S^k h_0\big\|_{H^1_{rdr}}\lesssim_k  \tau_1^{-2+},\quad k\geq 0,\quad r\lesssim t. 
  \end{align*}
  provided $\big|(t\partial_t)^jm\big|\lesssim 1$, $0\leq j\leq k$. More precisely, interpreting $h_0$ as a function depending on $m$, and correspondingly writing 
\[
h_0(t, r) = h_0(t,r; m),
\]
we can estimate
  \begin{align*}
  \big|S^k h_0(\cdot,\cdot;m)\big|\lesssim _k \tau_1^{-2+}\cdot \big(1 + \tau_1^{-p}\cdot\Big\|m\Big\|_{p,k}\big)^{k-1},\,r\lesssim t,\,p\geq \frac12,
  \end{align*}
  provided we have the a priori bound $\big\|m\big\|_{p,0}\leq 1$.
  We can also include an additional factor 
  \[
  \min\{1, (r\lambda_1)^4\}
  \]
  on the right-hand side. 
    \end{lem}
\begin{proof}[Proof of Lemma \ref{lem:hzerocorrection}] We shall use the notation 
\begin{equation}\label{eq:Rdef}
R: = \lambda_1(t)\cdot r. 
\end{equation}
Then we can write 
 \begin{align*}
  E_2 &=  \frac{\lambda_1''}{\lambda_1}\cdot\Phi(R) + \big(\frac{\lambda_1'}{\lambda_1}\big)^2\cdot \big(R\Phi'(R) - \Phi(R)\big)\\& - 2\frac{\sin\big(2\bfcq_{n-1}\big)}{r^2}\cdot \big[\cos(2Q_1\big) - 1\big]\\
& +  2\frac{\big[\cos\big(2 \bfcq_{n-1}\big) - 1\big]}{r^2}\cdot \sin\big(2Q_1\big), 
  \end{align*}
and recall that $\Phi(R) = 4R^2/(1+R^4)$. Let us introduce the modified source term 
 \begin{align*}
  \tilde{E}_2: = \frac{\lambda_1''}{\lambda_1^3}\cdot\Phi(R) + \big(\frac{\lambda_1'}{\lambda_1^2}\big)^2\cdot \big(R\Phi'(R) - \Phi(R)\big) - 8\cdot\big(\frac{\tilde{\lambda}_2}{\lambda_1}\big)^2\cdot \big[\cos\big(2Q(R)\big) - 1\big].
  \end{align*}
Here we have introduced
\begin{equation}\label{eq:tildelambdatwo}
\tilde{\lambda}_2: = \sqrt{\sum_{j=2}^n(-1)^j\lambda_j^2 +c_{n-1}}.    
\end{equation}
Next introduce the operator 
  \begin{equation}\label{eq:mathcalL}
\mathcal{L}: = \partial_{RR} + \frac{1}{R}\partial_R - \frac{4\cos(2Q(R))}{R^2},\qquad R = \lambda_1(t)r. 
\end{equation}
and its fundamental system 
\[
\frac14\Phi(R) = \frac{R^2}{1+R^4},\qquad \Theta(R): = \frac{-1 + 8R^4\log R + R^8}{4R^2(1+R^4)}. 
\]
Then we can solve the equation \eqref{eq:hzeroeqn} by means of the following variation of constants formula:
\begin{equation}\label{eq:hformula}\begin{split}
h_0(t, r) &= \frac14\Theta(R)\cdot \int_0^R f\big( \tfrac{s}{\lambda_1}  \big)\Phi(s)\,s ds\\
& -  \frac14\Phi(R)\cdot \int_0^R f\big(  \tfrac{s}{\lambda_1}  \big)\Theta(s)\,s ds,
\end{split}\end{equation}
where we have used the shorthand 
\[
f = \lambda_1^{-2}E_2 - 8 \lambda_1^{-2}\cdot (2m\bar{\lambda}_2 + m^2)\cdot \big[\cos(2Q_1) - 1\big].
\]
As in \cite{JenKri}, we decompose $h_0$ into three parts, where the middle one, $h_2$, turns out to be the most delicate one. Specifically, we write 
\[
h_0 = \sum_{j=1}^3 h_j,
\]
where we set 
\begin{align*}
&h_1 ( r ) = -  \frac14\Phi(R)\cdot \int_0^R f\big(  \tfrac{s}{\lambda_1}  \big)\Theta(s)\,s ds,\\
&h_2 ( r ) = \frac14\Theta(R)\cdot \int_0^R \tilde{f}\big(  \tfrac{s}{\lambda_1}  \big)\Phi(s)\,s ds,
\end{align*}
and where we have introduced the modified source term $\tilde{f}$ to define $h_2$, in turn defined by (recall the definition of $\tilde{E}_2$ in \eqref{eq:tildeE2}) \[
\tilde{f} = \tilde{E}_2 - 8 \lambda_1^{-2}\cdot (2m\lambda_2 + m^2)\cdot \big[\cos(2Q_1) - 1\big].
\]
It is worth noting that, according to Lemma \ref{lem:simpleinductivelambdaonewithm}, \[
\int_0^\infty \tilde{f}\big(\tfrac{s}{\lambda_1}\big) \Phi(s) sds =0.
\]
The final term $h_3$ accounts for the error produced by replacing $f$ by $\tilde{f}$, and is hence given by 
\begin{align*}
h_3(r) = \frac14\Theta(R)\cdot \int_0^R (f-\tilde{f})\big(  \tfrac{s}{\lambda_1}  \big) \Phi(s)\,s ds.
\end{align*}
We next observe that our a priori assumption on $\bfcq_{n-1}$ as described in \eqref{eq:vn-1bounds}, as well as the relation \eqref{eq:outerbubblesoln}, result in the pointwise bound 
\begin{equation}\label{eq:bfcqn-1bound1}
\big|\bfcq_{n-1}(t,r)\big|\lesssim \big(\overline{\lambda}_2r\big)^2,\,r\lesssim t,    
\end{equation}
provided $0<t\leq t_0$ with $t_0$ sufficiently small. This implies that 
\begin{equation}\label{eq:bfcqn-1bound2}
(\lambda_1 r)^{-2}\big|\bfcq_{n-1}(t,r)\big|\lesssim \big(\frac{\overline{\lambda}_2}{\lambda_1}\big)^2.
\end{equation}
We also note the improved (for suitably small $r$) error bound 
\begin{equation}\label{eq:bfcqn-1bound3}
\big|\bfcq_{n-1}(t,r) -  \bfq_{n-1}(t,r) - c_{n-1}(t)r^2\big|\lesssim \bar{\lambda}_2^4 r^4. 
\end{equation}
Now we estimate each of the terms $h_j$ separately: 
\\

{\it{(1): The estimate for $h_1$.}} Taking advantage of \eqref{eq:bfcqn-1bound2}, we deduce the bounds
\begin{align*}
&\lambda_1^{-2}\cdot \big|2\frac{\sin\big(2\bfcq_{n-1} \big)}{r^2}\cdot \big[\cos(2Q_1\big) - 1\big]\big|\lesssim \big(\frac{\bar{\lambda}_2}{\lambda_1}\big)^2\cdot \frac{R^4}{1+R^8},\\
&\lambda_1^{-2}\cdot \big| 2\frac{\big[\cos\big(2\bfcq_{n-1} \big) - 1\big]}{r^2}\cdot \sin\big(2Q_1\big)\big|\lesssim \big(\frac{\bar{\lambda}_2}{\lambda_1}\big)^2\cdot\frac{R^2}{1+R^4}.
\end{align*}
Further, in light of Lemma~\ref{lem:simpleinductivelambdaonewithm} and our choice of $\lambda_1$, we have the bound
\begin{equation}\label{eq:h1technical}\begin{split}
\Big|\frac{\lambda_1''}{\lambda_1^3}\cdot\Phi(R) + \big(\frac{\lambda_1'}{\lambda_1^2}\big)^2\cdot \big(R\Phi'(R) - \Phi(R)\big)\Big|&\lesssim\big(\frac{\bar{\lambda}_2}{\lambda_1}\big)^2\cdot \frac{R^2}{1+R^4}.
\end{split}\end{equation}
We suppress the implicit dependence on $m$ here but this will be made more explicit at later stages. In the preceding estimates, we can always replace $\frac{\bar{\lambda}_2}{\lambda_1}$ by $\lambda_1^{-1+}$. Moreover, in light of \eqref{eq:importanttauoneasymptotic}, we can also replace $\frac{\bar{\lambda}_2}{\lambda_1}$ by $\tau_1^{-1}$.
\\
Finally, we have the (crude) bound 
\[
\big|8 \lambda_1^{-2}\cdot (2m\bar{\lambda}_2 + m^2)\cdot \big[\cos(2Q_1) - 1\big]\big|\lesssim \lambda_1^{-2}\cdot \frac{R^4}{1+R^8}.
\]

The desired bound for this contribution to $h_0$, but still without the operator $S^k$, then easily follows from the asymptotic bounds for $\Phi(R)$, $\Theta(R)$. In fact, we obtain an improved estimate with $\tau_1^{-2}$ instead of $\tau_1^{-2+}$, and due to the quadratic vanishing of $\Phi(R)$ at $R = 0$, we can include a factor $\min\{1, (\lambda_1 r)^4\}$ on the right hand side. 
\\

{\it{(2): The estimate for $h_2$}}. Recall that we chose $\lambda_1$ by applying Lemma~\ref{lem:simpleinductivelambdaonewithm} to the equation
\[
\frac{\lambda_1''}{\lambda_1^3} - 2 \big(\frac{\lambda_1'}{\lambda_1^2}\big)^2 = -\frac{\big(\overline{\lambda}_2 + m\big)^2}{\lambda_1^2}.
\]
Keeping in mind the definition \eqref{eq:barlambdatwo} as well as the numerical identitites \eqref{eq:explicitintegrals}, we can rewrite the term $h_2$ in the following way: 
\begin{align*}
h_2 &= \chi_{R\lesssim 1}\cdot  \frac14\Theta(R)\cdot \int_0^R\tilde{f}(s)\Phi(s)\,s ds\\
& - \chi_{R\gtrsim 1}\cdot  \frac14\Theta(R)\cdot \int_R^\infty\tilde{f}(s)\Phi(s)\,s ds.
\end{align*}
Taking advantage of \eqref{eq:h1technical}, as well as the pointwise estimate 
\begin{align*}
\Big|4\cdot\big(\frac{\bar{\lambda}_2+m}{\lambda_1}\big)^2\cdot \big[\cos\big(2Q(R)\big) - 1\big]\Big|\lesssim \big(\frac{\bar{\lambda}_2}{\lambda_1}\big)^2\cdot \frac{R^4}{1+R^8},
\end{align*}
and further taking advantage of the estimates 
\begin{align*}
&\chi_{R\gtrsim 1}\cdot\big|\Theta(R)\cdot\int_R^\infty \frac{s^2}{1+s^4}\cdot \Phi(s)s\,ds\big|\lesssim  1\\
&\chi_{R\lesssim 1}\cdot\big|\Theta(R)\cdot\int_0^R \frac{s^2}{1+s^4}\cdot \Phi(s)s\,ds\big|\lesssim (r\lambda_1)^4,
\end{align*}
we infer the estimate 
\begin{align*}
\big|h_2(t,r)\big|\lesssim \tau_1^{-2}\min\{1, (\lambda_1 r)^4\}.    \end{align*}

{\it{(3): The estimate for $h_3$}}. Computing the difference of $E_2$ and $\tilde{E}_2$, we see that the following integrals need to be bounded: 
\begin{align*}
&(I): = 2\Theta(R)\int_0^R\frac{\big[\cos\big(2 \bfcq_{n-1}\big) - 1\big]}{S^2}\cdot \sin\big(2Q_1\big)\cdot \Phi(S)S\,dS,\\
&(II): = 2\Theta(R)\int_0^R2\frac{\big[\sin\big(2\bfcq_{n-1} \big) - 4\big(\frac{\tilde{\lambda}_2}{\lambda_1}S\big)^2\big]}{S^2}\cdot \big[\cos(2Q_1\big) - 1\big]\cdot \Phi(S)S\,dS,
\end{align*}
where we recall \eqref{eq:tildelambdatwo}. Also, $=\lambda_1(t)s$. Taking advantage of the bound 
\[
\big|\sin\big(2Q_1\big)\big|\lesssim \frac{S^2}{1+S^4}, 
\]
and further taking advantage of \eqref{eq:bfcqn-1bound1}, we derive the pointwise estimate 
\begin{align*}
\Big|\frac{\big[\cos\big(2 \bfcq_{n-1}\big) - 1\big]}{S^2}\cdot \sin\big(2Q_1\big)\cdot \Phi(S)S\Big|\lesssim \big(\frac{\bar{\lambda}_2}{\lambda_1}\big)^4\langle S\rangle^{-1}\min\{1, (s\lambda_1)^4\}.    
\end{align*}
Now if we restrict $r\lesssim t$, then we have $R = \lambda_1r\lesssim \lambda_1 t$. It follows that 
\begin{align*}
\Big|(I)\Big|&\lesssim \big(\lambda_1 t\big)^2\cdot |\log (\lambda_1 t)|\cdot\big(\frac{\bar{\lambda}_2}{\lambda_1}\big)^4\min\{1, (r\lambda_1)^4\}\\
&\lesssim \big(\frac{\bar{\lambda}_2}{\lambda_1}\big)^{2-}\min\{1, (r\lambda_1)^4\}, 
\end{align*}
where we have taken advantage of \eqref{eqlambdajassumptions}, \eqref{eq:lambdandef}. This gives the desired bound for this contribution to $h_3$, still without derivatives $S^k$. 
\\
As for the term $(II)$, recalling \eqref{eq:outerbubblepart}, \eqref{eq:outerbubblesoln}, as well as \eqref{eq:vn-1refined}, \eqref{eq:vn-1bounds} and the definition \eqref{eq:tildelambdatwo}, we infer the estimate 
\[
\Big|2\frac{\big[\sin\big(2\bfcq_{n-1} \big) - 4\big(\frac{\tilde{\lambda}_2}{\lambda_1}S\big)^2\big]}{S^2}\cdot \big[\cos(2Q_1\big) - 1\big]\cdot \Phi(S)S\Big|\lesssim \frac{\bar{\lambda}_2^3}{\lambda_1^4}\langle S\rangle^{-1}\min\{1, (s\lambda_1)^4\}.
\]
From here we infer the same type of estimate for the term $(II)$ as for the term $(I)$. 
\\

In order to complete the proof of the lemma, we also need to control the differentiated expression $S^kh_0$. Here we proceed in analogy to \cite{JenKri}, relying on the following kind of observation: 
\begin{lem}\label{lem:symbolbounds1} Let $H$ be a smooth function on $\mathbb{R}_+$ satisfying symbol bounds of the type 
\[
\big|(R\partial_R)^kH(R)\big|\lesssim_k \frac{R^{\alpha}}{\langle R\rangle^{\beta}},\,(\alpha,\beta)\in \mathbb{R}\times \mathbb{R}. 
\]
Then the composite function (with $\lambda_1,\tau_1$ as before)
\[
G(t, r): = H(\lambda_1(t)r)
\]
satisfies the bounds (as usual we set $p\geq \frac12$)
\begin{align*}
\Big|S^kG(t, r)\Big|\lesssim_k \tau_1^{0+}\cdot\big(1 + \tau_1^{-p}\big\|m\big\|_{p,k-1}\big)^k\cdot \frac{R^{\alpha}}{\langle R\rangle^{\beta}},\,R = \lambda_1(t)r,\,k\geq 0.     
\end{align*}
The same conclusion applies when $S$ is replaced by $t\partial_t$. 
\end{lem}
This lemma is a straightforward consequence of the chain rule and Lemma~\ref{lem:simpleinductivelambdaonewithm}.
\\

To see how this implies the desired derivative bounds for the terms $h_j$, let us consider for example the term $(I)$ contributing to the component $h_3$. Applying the Leibniz rule, we can write 
\begin{align*}
&S^k(I) \\&= \sum_{k_1+k_2=k,\,k_1<k}C_{k_{1,2}}S^{k_1}\big(\Theta(R)\big)S^{k_2-1}\Big(S(R)\cdot\frac{\big[\cos\big(2 \bfcq_{n-1}\big) - 1\big]}{R}\sin\big(2Q_1\big)\Phi(R)\Big)\\
&+2S^k\big(\Theta(R)\big)\int_0^R\frac{\big[\cos\big(2 \bfcq_{n-1}\big) - 1\big]}{S^2}\cdot \sin\big(2Q_1\big)\cdot \Phi(S)S\,dS.
\end{align*}
Then the desired bound follows for the second term by following the same argument as for the undifferentiated term $(I)$, together with the preceding lemma, and for the first term, by invoking the induction hypotheses \eqref{eq:vn-1bounds}. The remaining contributions to $h_0$ can be handled similarly. 
\end{proof}

In order to enable the construction of the desired function $m$ at the end, we shall also require differencing bounds with respect to $m$ for all the functions $h_j$, and in particular for $h_0$.
Writing 
\[
h_0 = h_0(t,r;m)
\]
and denoting (in accordance with \eqref{eq:differencingnotation})
\[
\triangle h_0(t,r;m,n) = h_0(t,r;m+n) - h_0(t,r;m), 
\]
we derive the following bound by exploiting \eqref{eq:lambdaonedifferencing}:     
\begin{equation}\label{eq:nzerodifferencingbounds}
\big|S^l\triangle h_0(t,r;m,n)\big|\lesssim_{l} \tau_1^{-2-p+}\cdot\big(\Big\|m\Big\|_{p,l} + \Big\|n\Big\|_{p,l}+1\big)^{l}\cdot\Big\|n\Big\|_{p,l},\,r\lesssim t. 
\end{equation}
Keeping in mind the equation \eqref{eq:veqn}, we see that the approximation $v_N = h_0$ generates the following error term:
 \begin{equation}\label{eq:e0error}\begin{split}
 e_0 &= -h_{0,tt} + E_1(h_0) + E_3(h_0) - 8(2m\lambda_2 + m^2)\cdot \big[\cos(2Q_1) - 1\big]\\
 & + \frac{4\cos(2Q_1) - 4\cos\big(2Q_1 - 2\bfcq_{n-1}\big)}{r^2}h_0.
\end{split}\end{equation}
Here $E_{1,3}$ are as defined on the right-hand side in \eqref{eq:veqn}. Then we can formulate the following error estimate:
 \begin{lem}\label{lem:ezerobound} We have the estimate 
 \begin{align*}
 &\lambda_2^{-1}(t)\cdot \big\|S^k \big(e_0 +  8(2m\lambda_2 + m^2)\cdot \big[\cos(2Q_1) - 1\big]\big)\big\|_{L^2_{r\,dr}(r\lesssim t)}\\&\lesssim_{k,l} \tau_1^{-2+}\cdot\big(1 + \tau_1^{-p}\|m\|_{p,k+2}\big)^{k+2},\,k\geq 0,
 \end{align*}
 where we restrict to $p\geq \frac12$. Furthermore, setting $E_0: = e_0 +  8(2m\lambda_2 + m^2)\cdot \big[\cos(2Q_1) - 1\big]$ and writing $E_0 = E_0(t,r;m)$, we have the differencing bounds 
 \begin{align*}
&\lambda_2^{-1}(t)\cdot \big\|S^l\triangle E_0(t,r;m,n)\big\|_{L^2_{r\,dr}}\lesssim_{l} \tau_1^{-2-p+}\cdot\big(\Big\|m\Big\|_{p,l+2} +\Big\|n\Big\|_{p,l+2}+1\big)^{l}\cdot\Big\|n\Big\|_{p,l+2}.
 \end{align*}
 Here we restrict to $p\geq \frac12$. Finally, we also have the estimate 
 \begin{align*}
  &\lambda_2^{-4}(t)\cdot \big\|r^{-2}S^k \big(e_0 +  8(2m\lambda_2 + m^2)\cdot \big[\cos(2Q_1) - 1\big]\big)\big\|_{L^\infty_{r\,dr}(r\lesssim t)}\\&\lesssim_{k,l} \tau_1^{0+}\cdot\big(1 + \tau_1^{-p}\|m\|_{p,k+2}\big)^{k+2},\,k\geq 0.
\end{align*}
 \end{lem}
\begin{proof} We verify this for the different terms constituting
$E_0$. 
\\

\noindent
{$\bullet$ \it{The estimate for $h_{0,tt}$.}} Write 
\begin{align*}
h_0 &= \chi_{R\lesssim 1}\cdot  \Big[\frac14\Theta(R)\cdot \int_0^R\tilde{f}(s)\Phi(s)\,s ds - \frac14\Phi(R)\cdot \int_0^R\tilde{f}(s)\Theta(s)\,s ds\Big]\\
& + \chi_{R\gtrsim1}\cdot  \Big[-\frac14\Theta(R)\cdot \int_R^\infty\tilde{f}(s)\Phi(s)\,s ds - \frac14\Phi(R)\cdot \int_0^R\tilde{f}(s)\Theta(s)\,s ds\Big]\\
& + \frac14\Theta(R)\cdot \int_0^R[f-\tilde{f}](s)\Phi(s)\,s ds - \frac14\Phi(R)\cdot \int_0^R[f-\tilde{f}](s)\Theta(s)\,s ds\\
& = : \sum_{j=1}^3 H_j. 
\end{align*}
We note that each of the three variation of constants expressions gain two derivatives of smoothness with respect to the integrand $\tilde{f}, f - \tilde{f}$, respectively. Then we note that 
\begin{align*}
&(t\partial_t)H_3\\
&=  \frac14(t\partial_t)\big(\Theta(R)\cdot \lambda^2(t)\big)\int_0^r[f-\tilde{f}](t,\lambda(t)s)\Phi(\lambda(t)s)s\, ds\\& - \frac14(t\partial_t)\big(\Phi(R)\cdot \lambda^2(t)\big)\int_0^r[f-\tilde{f}](t,\lambda(t)s)\Theta(\lambda(t)s)\,s ds\\
& + \frac14\Theta(R)\cdot \lambda^2(t)\int_0^r(t\partial_t)\big([f-\tilde{f}](t,\lambda(t)s)\Phi(\lambda(t)s)\big)s\, ds\\& - \frac14\big(\Phi(R)\big)\cdot \lambda^2(t)\int_0^r(t\partial_t)\big([f-\tilde{f}](t,\lambda(t)s)\Theta(\lambda(t)s)\big)\,s ds.
\end{align*}
 The first two expressions on the right can be bounded in close analogy to the term $h_3$, also relying on Lemma~\ref{lem:symbolbounds1}. To estimate the last two expressions, we write 
 \begin{align*}
 (t\partial_t)\big([f-\tilde{f}](t,\lambda(t)s)\Phi(\lambda(t)s)\big) &= S\big([f-\tilde{f}](t,\lambda(t)s)\Phi(\lambda(t)s)\big)\\
 & - (s\partial_s)\big([f-\tilde{f}](t,\lambda(t)s)\Phi(\lambda(t)s)\big).
 \end{align*}
 The first term on the right leads to a contribution to the remaining two variation of constant formula contributions which can be treated just like the term $h_3$ in the proof of Lemma~\ref{lem:hzerocorrection}, due to the a priori bounds \eqref{eq:vn-1bounds} - \eqref{eq:vn-1boundsrefined}. For the last term on the right, we perform integration by parts with respect to $s$, leading to the term 
 \begin{align*}
&-\frac12\Theta(R)\cdot \lambda^2(t)\int_0^r\big([f-\tilde{f}](t,\lambda(t)s)\Phi(\lambda(t)s)\big)s\, ds\\&\hspace{2cm} +\frac12\big(\Phi(R)\big)\cdot \lambda^2(t)\int_0^r\big([f-\tilde{f}](t,\lambda(t)s)\Theta(\lambda(t)s)\big)\,s ds.
 \end{align*}
 This can again be treated like the term $h_3$ in the proof of Lemma~\ref{lem:hzerocorrection}.
 \\
 As for the term $(t\partial_t)^2H_3$, one proceeds as before, performing integration by parts to suppress the operator $s\partial_s$. The only new feature is the appearance of the boundary term 
 \begin{align*}
 &\frac14(t\partial_t)\big(\Theta(R)\cdot r^2\lambda^2(t)\big)[f-\tilde{f}](t,\lambda(t)r)\Phi(\lambda(t)r)\\&\hspace{2cm} - \frac14(t\partial_t)\big(\Phi(R)\cdot r^2\lambda^2(t)\big)[f-\tilde{f}](t,\lambda(t)r)\Theta(\lambda(t)r)
 \end{align*}
 This term is seen to satisfy the same bounds as $h_3$. Moreover, we can apply arbitrary powers of the operator $S$ to it (using Lemma~\ref{lem:simpleinductivelambdaonewithm}), as well as to the preceding expressions, using integration by parts as before for the expressions involving integrals, as needed. 
 In light of the identity 
 \[
 \partial_t^2 = t^{-2}\big[(t\partial_t)^2 - (t\partial_t)\big], 
 \]
 The first inequality of the lemma for the contribution of $h_{0,tt}$ then follows from the preceding and the inequality 
 \[
 t^{-2}\lesssim \tau_1^{0+}. 
 \]
 The second inequality of the lemma for the contribution of $h_{0,tt}$ follows by straightforward differencing in anaolgy to \eqref{eq:nzerodifferencingbounds}, and the last inequality is a consequence of the fourth order vanishing of the contribution of $h_{0,tt}$ at $r= 0$, in turn inherited from Lemma~\ref{lem:hzerocorrection}.  
 \\

 Continuing with the proof of Lemma~\ref{lem:ezerobound}, we turn to 
\\

\noindent
{$\bullet$ \it{The estimates for $E_{1,3}(h_0)$.}} Recalling \eqref{eq:veqn}, in addition to \eqref{eq:outerbubblesoln} as well as the a priori bounds \eqref{eq:vn-1bounds}, we infer the bounds\footnote{We also use that $\lambda_1\lesssim \tau_1^{1+}$.} 
\begin{align*}
&\big\|S^{l}\big( \frac{2\sin\big(2Q_1- 2\bfcq_{n-1}\big)}{r^2}\big)(t,r)\big\|_{L^2_{r\,dr}}\lesssim_{l} \tau_1^{1+},\,l\geq 0.     
\end{align*}
Also taking advantage of Lemma~\ref{lem:hzerocorrection} and writing 
 \begin{align*}
 &S^l\big(E_3(h_0)\big)\\& = \sum_{l_1+l_2 = l}C_{l_{1,2}}\cdot S^{l_1}\big( \frac{2\sin\big(2Q_1- 2\bfcq_{n-1}\big)}{r^2}\big)\cdot  S^{l_2}\big(\cos(2h_0) - 1\big),
\end{align*}
we infer the bound 
\begin{align*}
\big\|S^lE_3(h_0)\big\|_{L^2_{r\,dr}(r\lesssim t)}\lesssim_l\tau_1^{-3+}\big(1 + \tau_1^{-p}\big\|m\big\|_{p,l}\big)^l.    
\end{align*}
For the differencing bound, we take advantage of the estimate 
\begin{align*}
&\big\|S^{l}\triangle\big( \frac{2\sin\big(2Q_1 - 2\bfcq_{n-1}\big)}{r^2}\big)(t,r;m,n)\big\|_{L^2_{r\,dr}}\\&\lesssim_{l} \tau_1^{1-p+}\cdot \big(1+ \Big\|m\Big\|_{p,l} +  \Big\|n\Big\|_{p,l}\big)^{l}\cdot\Big\|n\Big\|_{p,l},\,l\geq 0,
\end{align*}
in turn a consequence of Lemma~\ref{lem:simpleinductivelambdaonewithm}, and the differencing estimate \eqref{eq:nzerodifferencingbounds}, yielding the bound
 \begin{align*}
 \big\|S^l\triangle \big(E_3(h_0)(\cdot,\cdot;m,n)\big)\big\|_{L^2_{r\,dr}}\lesssim_l\tau_1^{-3-p+}\cdot \big(\Big\|m\Big\|_{p,l} +\Big\|n\Big\|_{p,l}+1\big)^{l}\cdot\Big\|n\Big\|_{p,l}.
 \end{align*}
 The final bound asserted in the lemma follows for this contribution again from the vanishing properties of $h_0$ at the origin. 
 The term $E_1(h_0)$ is handled analogously. 
 \\

\noindent
{$\bullet$ \it{The estimates for the term $ \frac{4\cos(2Q_1) - 4\cos\big(2Q_1 - 2\bfcq_{n-1}\big)}{r^2}h_0$.}} Here we exploit the bounds
 \[
 \big|S^l\big(\frac{4\cos(2Q_1) - 4\cos\big(2Q_1 - 2\bfcq_{n-1}\big)}{r^2}\big)\big|\lesssim_l \bar{\lambda}_2^{2+},\,l\geq 0,
 \]
 which is a consequence of our assumption \eqref{eqlambdajassumptions}. Recalling Lemma~\ref{lem:hzerocorrection} and the fact that $\bar{\lambda}_2\lesssim \tau_1^{0+}$, we infer from application of Leibniz' rule the bounds 
\begin{align*}
\Big\|S^l\Big( \frac{4\cos(2Q_1) - 4\cos\big(2Q_1 - 2\bfcq_{n-1}\big)}{r^2}h_0\Big)\Big\|_{L^2_{r\,dr(r\lesssim t)}}\lesssim_l\tau_1^{-2+}\cdot \big(1 + \tau_1^{-p}\cdot\Big\|m\Big\|_{p,l}\big)^{l-1}
\end{align*}
The fourth order vanishing at the origin $r = 0$ of these terms is also a consequence of Lemma~\ref{lem:hzerocorrection}, and all additional factors $\bar{\lambda}_2$ can be absorbed by $\tau_1^{0+}$. The differencing bounds are more of the same. 
\end{proof}

\end{proof}

%%%%%%%%%%%%%%%%%%%%%%%%%%%%%%%%%%%%%%%%%%%
%%%%%%%%%%%%%%%%%%%%%%%%%%%%%%%%%%%%%%%%%%%
%%%%%%%%%%%%%%%%%%%%%%%%%%%%%%%%%%%%%%%%%%%
%%%%%%%%%%%%%%%%%%%%%%%%%%%%%%%%%%%%%%%%%%%
%%%%%%%%%%%%%%%%%%%%%%%%%%%%%%%%%%%%%%%%%%%
%%%%%%%%%%%%%%%%%%%%%%%%%%%%%%%%%%%%%%%%%%%

%\newpage
\medskip
{\bf{Step 1}}: {\it{Construction of the correction $h_1$.}} We now begin the inductive process for the construction of the correction. This step corresponds to the base case. The purpose of this step is to cancel the error $e_{0,2}$ (see below). More precisely, this first step consists of obtaining an approximate solution to the following wave equation \begin{align}\label{eq:step1_1}
-\partial_t^2 h_1 + \partial_r^2 h_1 +\dfrac{1}{r}\partial_r h_1  - \dfrac{4}{r^2} \cos\big(2Q_1-2\bfcq_{n-1}\big)h_1 = e_{0,2} + \boxed{ e_{1,\co}} 
\end{align}
where we define \begin{align}\label{eq:step1_11}
e_{0}:= e_{0,1} - e_{0,2}, \, \quad e_{0,1}:= E_{2,\co}, \quad \hbox{and} \quad e_{1,\co} = m_1\Big(\cos(2Q_1)-1\Big).
\end{align}
We recall that \begin{equation}\label{eq:step1_12}\begin{aligned}
    e_{0,2}&=  \partial_t^2 h_0 + E_1(h_0)+E_2(h_0)+E_3(h_0) ,
    \\  E_1(h_0) & =  \dfrac{2}{r^2}\cos\big( 2Q_1-2\bfcq_{n-1} \big) \Big( \sin(2h_0)-2h_0 \Big)  ,
    \\ E_2(h_0) &  =  \dfrac{4}{r^2}\Big(\cos\big(2Q_1-2\bfcq_{n-1}\big)- \cos(2Q_1)\Big) h_0,
    \\  E_3(h_0)  & =  \dfrac{2}{r^2}\sin\big(2Q_1-2\bfcq_{n-1}\big)\Big( \cos(2h_0)-1\Big).
\end{aligned}\end{equation}
It is worth stressing that we are not aiming to cancel the first correction term from the previous step, namely, \begin{align}\label{eq:step1_13}
e_{0,1}=E_{2,\co}  :=  - 8 (2m\overline{\lambda}_2 + m^2)\cdot \Big(\cos(2Q_1) - 1\Big). 
\end{align}
On the other hand, the main purpose of the boxed term $e_{1,\co}$ in \eqref{eq:step1_1} is to enforce the additional vanishing condition required to eliminate the growing part of the solution of said equation (similar to the introduction of $E_{2,\co}$ in Step 0 before). We anticipate that, to ensure the same vanishing condition for all $h_k$, $k\geq1$, a similar correction term will be introduced at every step of the inductive process, ultimately requiring that (at least approximately) \[
-8\big(2\lambda_2 m+m^2) + \sum_{k=1}^N m_k =0,
\]
where the $m_k$ are the factors chosen to guarantee the subsequent vanishing conditions at every step of the iteration. 

\medskip

Now, the strategy to construct $h_1$ will be via a two-stage process: first, we solve an outer
wave equation involving the potential term (associated with the $n-1$ outer bubbles) \[
-\dfrac{4}{r^2}\cos\big(2\bfcq_{n-1}\big),
\]
then we solve an inner elliptic equation akin to the one satisfied by $h_0$ (with a potential term associated with the innermost bubble $Q_1$). The first step has a smoothing effect at the outer scale $\lambda_2^{-1}$, while the second step has a shrinking effect. The role of the boxed term in \eqref{eq:step1_1} will be to ensure good bounds for the second step.

That being said, we decompose the solution as \[
h_1=:h_1^{\out}+h_1^{\inn},
\]
where $h_{1}^{\out}$ will be chosen to satisfy an outer wave equation, whereas $h_{1}^{\inn}$ will solve an inner elliptic equation. Concretely, we choose $h_{1}^{\out}$ and $h_{1}^{\inn}$ satisfying  \begin{align}\label{eq:step1_2}
-\partial_t^2 h_{1}^{\out} + \partial_r^2 h_{1}^{\out} +\dfrac{1}{r}\partial_r h_{1}^{\out} - \dfrac{4}{r^2}\cos(2\bfcq_{n-1}) h_{1}^{\out}  =  e_{0,2},
\end{align}  
and \begin{equation}\label{eq:step1_3}\begin{aligned}
& \partial_r^2 h_{1}^{\inn} +\dfrac{1}{r}\partial_r h_{1}^{\inn} - \dfrac{4}{r^2}\cos(2Q_1) h_{1}^{\inn} 
\\ & = e_{1,\co} + \dfrac{4}{r^2}\Big(\cos(2Q_1-2\bfcq_{n-1})-\cos(2\bfcq_{n-1}) \Big) h_{1}^{\out}. 
\end{aligned}\end{equation}
The lemma below contains the key information of this step. We point out that the parameter $m_1$ (appearing in the equation for $h_1^{\inn}$) will depend mildly on $h_1^{\out}$.
\begin{lem}\label{lem:step1_2}
Let $m$ be as before. Then, there exists a function $m_1$ on $(0,t_0]$ satisfying the bound \begin{align}\label{eq:step1_4}
\vert m_1(t)\vert \ll_{t_0} \tau_1^{-1/2},
\end{align}
such that equation \eqref{eq:step1_1} admits an approximate solution on $(0,t_0]\times (0,\infty)$ satisfying the bounds \[
\big\Vert S^\ell h_1(t,\cdot)\big\Vert_{H^1_{rdr}}\lesssim_{k,l} \tau_1^{-(2-)}\bigg(1+\tau_1^{-\ell}\sum_{0\leq j \leq k+4}\sup_{0 < t \leq t_0} \tau_1^\ell \big\vert (t\partial_t)^j m(t)\big\vert \bigg),
\]
for $k\geq0$ and $\ell \geq \tfrac12$. We also have a representation 
\[
h_1(t, r) = c_1(t)r^2 + g_1(t, r)
\]
with the bounds 
\[
\big|(t\partial_t)^kc_1(t)\big|\lesssim_k \tau_1^{-1+},\,\big|S^kg_1(t, r)\big|\lesssim_k \tau_1 r^4. 
\]
Moreover, the difference between the right- and left-hand sides of \eqref{eq:step1_1} - recalling that this is only an approximate solution - satisfies \[
\lambda_2^{-1}\big\Vert S^k e_{1,0}\big\Vert_{L^2_{rdr}(r\lesssim t)} \lesssim_k \tau_1^{-(3-)}\bigg(1+\tau_1^{-\ell} \sum_{0\leq j \leq k+6} \sup_{0 < t \leq t_0} \tau_1^\ell \big\vert (t\partial_t)^j m(t)\big\vert \bigg),
\]
having called $e_{1,0}$ this difference. We also have the pointwise bounds
\[
\big|r^{-2}S^k e_{1,0}(t,r)\big|\lesssim_k 1,\,r\lesssim t. 
\]

Furthermore, we have the differencing bounds \begin{align*}
    \big\Vert S^k\triangle h_1(t,r;m,n)\big\Vert_{H^1_{rdr}} \lesssim_\ell \tau_1^{-(3+p-\varepsilon)}\Big( 1+ \Vert m\Vert_{p,\ell+4} + \Vert n\Vert_{p,\ell+4}\Big)^\ell  \Vert n\Vert_{p,\ell+4},
\end{align*}
and \begin{align*}
    \lambda_2^{-1}\big\Vert S^k\triangle e_{1,0}(t,r;m,n)\big\Vert_{L^2_{rdr}} \lesssim_\ell \tau_1^{-(3+p-\varepsilon)}\Big( 1+ \Vert m\Vert_{p,\ell+6} + \Vert n\Vert_{p,\ell+6}\Big)^\ell  \Vert n\Vert_{p,\ell+6}.
\end{align*}
\end{lem}

\begin{proof}
    As stated before, we decompose $h_1=h_1^{\out}+h_1^{\inn}$, each of which satisfy \eqref{eq:step1_2}-\eqref{eq:step1_3} respectively. We split the analysis into multiple steps. Before going any further, we record that $e_{1,0}$ is given by \begin{align}\label{eq:step1_14}
        e_{1,0}:= \partial_t^2 h_1^{\inn} + \dfrac{4}{r^2}\Big( \cos(2Q_1-2\bfcq_{n-1})-\cos(2Q_1) \Big) h_1^{\inn}.
    \end{align}

    \medskip

    \noindent
    \textit{(i) Control of the solution to the outer wave equation \eqref{eq:step1_2}.} We have that equation \eqref{eq:step1_2} is exactly of the form \eqref{linout1}. Hence, applying Lemma \ref{lem:linout_1} with $F=e_{0,2}$, yields the desired decay, provided that the right-hand side satisfies the bound \eqref{linout16}. On the other hand, Lemma \ref{lem:ezerobound} guarantees $e_{0,2}$ satisfies said bound with $p=2-$, and hence, we infer the existence of a solution satisfying the bounds \begin{align}\label{eq:step1_7}
    \Vert S^k h_1^{\out}\Vert_{\bch^1_{rdr}} \lesssim \tau_1^{-(2-(k+1)\varepsilon)}\bigg(1+ \tau_1^{-\ell} \sum_{0 \leq j \leq k+2} \sup_{0 < t \leq t_0} \tau_1^{\ell} \big\vert (t\partial_t)^jm(t)\big\vert \bigg),
    \end{align}
    with $k\geq0$, $\ell \geq \tfrac12$. Moreover, we have that the solution $h_1^\out$ depends continuously on $m$. Specifically, write $h_1^{\out}=h_1^{\out}(t,r;m)$. Recall the notation $\triangle f(t,r;m,n)  = f(t,r;m+n) - f(t,r;m)$, the fact that $\lambda_1$ depends on $m$ through Lemma \ref{lem:simpleinductivelambdaonewithm} that and $\lambda_2=O(\tau_1^{0+})$. Then, from Lemma \ref{lem:ezerobound} it follows that \[
    \big\Vert S^k \triangle h_1^{\out}(t,r;m,n)\big\Vert_{\bch^1_{rdr}} \lesssim_k  \tau_1^{-2-p+} \Big( \Vert m\Vert_{p,\ell+2}+\Vert n\Vert_{p,\ell+2}+1\Big) \Vert n\Vert_{p,\ell+2},
    \]
    where $p\geq \tfrac{1}{2}$. Finally, taking advantage of Lemma~\ref{lem:Tayloerexpansion} and the last inequality in Lemma~\ref{lem:ezerobound}, we deduce that 
    \[
    h_1^{\out}(t, r) = c_1^{\out}(t)r^2 + g_1^{\out}(t, r)
    \]
    where we have 
    \[
    \big|(t\partial_t)^kc_1^{\out}(t)\big|\lesssim_k \tau_1^{-1+},\,\big|S^kg_1(t,r)\big|\lesssim_k \tau_1 r^4.  
    \]
    \bigskip

    \noindent
    \textit{(ii) Choice of the parameter $m_1(t)$ in the inner elliptic equation \eqref{eq:step1_3}.} Similarly to the case of $h_0$ in \eqref{eq:hzeroeqn} (see also \eqref{eq:hformula}), we will solve equation \eqref{eq:step1_3} by means of variations of constants formula. Letting $S=\lambda_1(t)s$, we choose  \begin{align*}
        4 m_1(t):= - \lambda_1^2\int_0^\infty \dfrac{4}{S^2}\Big(\cos(2Q_1-2\bfcq_{n-1})-\cos(2\bfcq_{n-1})\Big) h_1^{\out}(t,s)\Phi(S) SdS,
    \end{align*}
    where the factor $4$ in front of $m_1(t)$ comes from \eqref{eq:explicitintegrals}. We point out that with this definition we automatically obtain the following orthogonality property    \begin{equation}\label{eq:step1_5}\begin{aligned}
        & \bigg\langle \dfrac{4}{r^2}\Big(\cos(2Q_1-2\bfcq_{n-1})-\cos(2\bfcq_{n-1})\Big) h_1^{\out}, \,  \Phi(\lambda_1\cdot) \bigg\rangle_{L^2_{rdr}} 
        \\ & + m_1 \Big\langle \cos(2Q_1)-1, \,  \Phi(\lambda_1\cdot)\Big\rangle_{L^2_{rdr}}=0. 
    \end{aligned}\end{equation}
    having used \eqref{eq:explicitintegrals}. The following Lemma ensures, in particular, that \eqref{eq:step1_4} holds.
    \begin{lem}\label{lem:step1_1}
        The function $m_1(t)$ satisfies the bounds \begin{align*}
            \big\vert (t\partial_t)^\ell m_1(t)\big\vert \lesssim_\ell \tau_1^{-(1-\varepsilon)}\Big(1 + \Vert m\Vert_{p,\ell+4}\Big)^\ell, \qquad p\geq \dfrac{1}{2}.
        \end{align*}
        Moreover, interpreting $m_1(t)=m_1(t;m)$ we have the estimate for the difference \begin{align}\label{eq:step1_8}
            \big\vert (t\partial_t)^\ell \triangle m_1(t;m,n)\big\vert  \lesssim_{\ell} \tau_1^{-1-p+} \Big(1+\Vert m\Vert_{p,\ell+4} + \Vert n\Vert_{p,\ell+4}\Big)^\ell \Vert n\Vert_{p,\ell+4},
        \end{align}
        with $p\geq \tfrac12$.
    \end{lem}
    \begin{proof}[Proof of Lemma \ref{lem:step1_1}]
    We begin by showing the base case $\ell=0$. Indeed, first recall that  \begin{align*}
        & \cos(2Q_1-2\bfcq_{n-1})-\cos(2\bfcq_{n-1})=-2\sin(Q_1)\sin\big(Q_1-2\bfcq_{n-1}\big),
    \end{align*}
    and \begin{align*}
        \sin(Q_1)=\dfrac{2S^2}{1+S^4}.
    \end{align*}
    Hence, for $t_0\in (0,1)$ small enough (recall \eqref{eq:construction_app_tau1}), \begin{align}\label{eq:step1_6}
    \bigg\vert \dfrac{4}{S^2}\Big( \cos(2Q_1-2\bfcq_{n-1})-\cos(2\bfcq_{n-1}) \Big)   s^2 \bigg\vert \lesssim \dfrac{1}{\langle S\rangle^2} \, \dfrac{1}{\lambda_1^2}\lesssim \dfrac{1}{\langle S\rangle^2} \, \dfrac{1}{\tau_1^2},
    \end{align}
    having bounded $\vert \sin(Q_1-2\bfcq_{n-1}) \vert $ by $1$. In particular, the last inequality implies that the integral defining $m_1$ is bounded by \begin{align*}
    & \vert \lambda_1^{-2} m_1(t)\vert 
    \\ & \lesssim \bigg\Vert \dfrac{4}{S^2} \Big( \cos(2Q_1-2\bfcq_{n-1})-\cos(2\bfcq_{n-1}) \Big) \big( S\Phi(S)\big) s^2 \bigg\Vert_{L^2_{SdS}} \bigg\Vert \dfrac{h_1^{\out}}{s^2}\bigg\Vert_{L^2_{SdS}}
    \\ & \lesssim \dfrac{1}{\tau_1^2} \bigg\Vert \dfrac{h_1^{\out}}{s^2}\bigg\Vert_{L^2_{SdS}}
    \end{align*}
    On the other hand, Lemma \ref{lem:linout_1} (recall that $h_1^{\out}$ satisfies equation \eqref{eq:step1_2}) implies that \begin{align}\label{eq:step1_10}
        \bigg\Vert \dfrac{h_1^{\out}}{s^2}\bigg\Vert_{L^2_{SdS}} \lesssim \lambda_1 \bigg\Vert \dfrac{h_1^{\out}}{s^2}\bigg\Vert_{L^2_{sds}} \lesssim \lambda_1 \lambda_2^2 \, \tau_1^{-(2-\varepsilon)}\Big(1+\Vert m \Vert_{p,4}\Big), \qquad p\geq \dfrac{1}{2}.
    \end{align}
    We conclude then that \[
    \vert \lambda_1^{-2} m_1(t)\vert \lesssim \tau_1^{-(3-2\varepsilon)} \Big(1+\Vert m \Vert_{p,4}\Big), \qquad p\geq \dfrac{1}{2}.
    \]
    Next, we prove inequality \eqref{eq:step1_8} for the difference. This follows from the same estimates as before, from where we obtain that \begin{equation}\label{eq:step1_9}\begin{aligned}
       & \bigg\vert \triangle\bigg(\dfrac{4}{R^2}\Big(\cos(2Q_1-2\bfcq_{n-1})-\cos(2\bfcq_{n-1})\Big) R \Phi(R) r^2\bigg)(t,r;m,n)\bigg\vert 
       \\ & \lesssim \langle R\rangle^{-2} \tau_1^{-2-p+} \Vert n\Vert_{p,0},
    \end{aligned}\end{equation}
    having used the bound (see \eqref{eq:lambdaonedifferencing})
    \[
    \big\vert \triangle\lambda_1(t;m,n)\big\vert\lesssim_{l} \tau_1^{-p+}\lambda_1\Vert n\Vert_{p,0},
    \]
    as well as the fact that, in general, for a function $F(R)$, we have \begin{align*}
    & F\big(R(m+n)\big)-F\big(R(m)\big)
    \\ & = \Big(R(m+n)-R(m)\Big)\int_0^1 F'\Big(sR(m+n)+(1-s)R(m)\Big)ds,
    \end{align*}
    which in turn implies that \begin{align*}
        \big\vert \triangle F(R)\big\vert \lesssim \vert \triangle R\vert \,\sup_{\xi\in[R(m),R(m+n)]}\big\vert F'(\xi)\big\vert, \qquad \vert R F'(R)\vert\lesssim \dfrac{1}{\lambda_1^2} \dfrac{1}{\langle R\rangle^2},
    \end{align*}
    where in the last inequality we have taken $F$ as being the function inside the big parenthesis in \eqref{eq:step1_9}. Finally, the derivative bounds for the difference follow identical lines as before (up to obvious modifications), and hence we omit them. 
    \end{proof}

    \bigskip

    \noindent
    \textit{(iii) Control of the solution to the inner elliptic equation \eqref{eq:step1_3}.} This is a consequence of the following lemma.
    \begin{lem}
        Equation \eqref{eq:step1_3} admits a solution satisfying the bounds \begin{align*}
            \big\Vert S^\ell \triangle h_{1}^\inn (t,r;m,n)\big\Vert_{H^1_{rdr}} \lesssim_{\ell} \tau_1^{-3-p+}\Big(1 + \Vert m \Vert_{p,\ell+4} + \Vert n\Vert_{p,\ell+4}\Big)^\ell \Vert n\Vert_{p,\ell+4},
        \end{align*}
        where $k,\ell\geq0$ and $p\geq \tfrac12$. We also have the estimates 
        \[
        \big|S^k h_{1}^\inn (t,r;m)\big|\lesssim_k \tau_1^{1+}r^4. 
        \]
    \end{lem}
    \begin{proof}
        First, for simplicity let us define \[
        G:=e_{1,\co} + \dfrac{4}{r^2}\Big(\cos(2Q_1-2\bfcq_{n-1})-\cos(2\bfcq_{n-1}) \Big) h_{1}^{\out}.
        \]
        It is enough to estimate \[
        \Phi(R)\int_0^R \lambda_1^{-2}  \Theta(s) G\big(t,\tfrac{s}{\lambda_1}\big) sds - \Theta(R)\int_0^R  \lambda_1^{-2} \Phi(s) G\big(t,\tfrac{s}{\lambda_1}\big)sds. 
        \]
        Then, due to the orthogonality in \eqref{eq:step1_5}, we can write the last term above as \[
        -\chi_{R\lesssim 1} \, \Theta(R)\int_0^R  \lambda_1^{-2} \Phi(s) G\big(t,\tfrac{s}{\lambda_1}\big)sds + \chi_{R\gtrsim 1}\, \Theta(R)\int_R^\infty  \lambda_1^{-2} \Phi(s) G\big(t,\tfrac{s}{\lambda_1}\big)sds,
        \]
        which prevents this term from growing quadratically. Indeed, using estimate \eqref{eq:step1_7} and \eqref{eq:step1_10} for $h_1^{\out}$, along with  \eqref{eq:step1_6}, and Cauchy-Schwarz inequality, we infer that \begin{align*}
        & \bigg\Vert \dfrac{4}{R^2}\Big(\cos(2Q_1-2\bfcq_{n-1})-\cos(2\bfcq_{n-1})\Big) h_1^{\out}\big(\tfrac{R}{\lambda_1}\big) \Phi(R)R\bigg\Vert_{L^1_{RdR}}
    \\ & \lesssim \tau_1^{-3+}\Big(1+\Vert m\Vert_{p,2}\Big),
        \end{align*}
        where, again, the reason to use the singular $L^2$ norm \eqref{eq:step1_10} in the last inequality is to make appear $\lambda_1^{-2}$. Then, calling \[
        G_1:= e_{1,\co}, \ \quad G_2:= \dfrac{4}{r^2}\Big(\cos(2Q_1-2\bfcq_{n-1})-\cos(2\bfcq_{n-1}) \Big) h_{1}^{\out},
        \]
        we see that \begin{equation}\label{eq:step1_17}\begin{aligned}
            \bigg\Vert \chi_{R\lesssim1} \Theta(R)\int_0^R \lambda_1^{-2} G_2\big(t,\tfrac{s}{\lambda_1}\big) \Phi(s) sds  \bigg\Vert_{H^1_{rdr}}  & \lesssim_p \tau_1^{-3+} \Big( 1 + \Vert m \Vert_{p,2}\Big),
            \\ \bigg\Vert   \chi_{R\gtrsim 1} \Theta(R) \int_R^\infty \lambda_1^{-2} G_2\big(t,\tfrac{s}{\lambda_1}\big) \Phi(s) sds \bigg\Vert_{H^1_{rdr}}  & \lesssim _p \tau_1^{-3+} \Big( 1 + \Vert m \Vert_{p,2}\Big),
            \\ \bigg\Vert   \Phi(R) \int_0^R \lambda_1^{-2} G_2\big(t,\tfrac{s}{\lambda_1}\big) \Phi(s) sds \bigg\Vert_{H^1_{rdr}}  & \lesssim _p \tau_1^{-3+} \Big( 1 + \Vert m \Vert_{p,2}\Big).
        \end{aligned}\end{equation}
        The bounds for $G_1$ follow identical lines, using Lemma \ref{lem:step1_1}. The differentiated and differencing estimates follow mutatis mutandis from the preceding parts, so we omit them. The final proof follows from the fourth order vanishing of the variation of constants formula at $R = 0$. 
        The proof is complete.
    \end{proof}
    
    \begin{rem}\label{rem:step1_1}
     Note that from Lemma \ref{lem:step1_1} and the explicit formula for $h_1^{\inn}$ one can deduce that  \begin{align}\label{eq:step1_16}
    \Big\Vert \dfrac{1}{r^2}h_1^{\inn}\Big\Vert_{L^2_{rdr}}\lesssim \tau_1^{-2+}. 
     \end{align}
    Indeed, estimating in a similar way as for  \eqref{eq:step1_17}, now using the singular estimate \eqref{linout12} for $G_2$, and, in the case of $G_1$, the fact that $m_1(t)$ decays as $\tau^{-1+}$ (see Lemma \ref{lem:step1_1}), one easily obtains \eqref{eq:step1_16}. We omit the details.
    \end{rem}
    
    \bigskip

    \noindent
    \textit{(iv) Control of the error term $e_{1,0}$ in \eqref{eq:step1_14}} Finally, we estimate the remaining error term via the following lemma. \begin{lem}\label{lem:step1_4}
        The following estimate holds \begin{align*}
            \lambda_2^{-1} \big\Vert S^\ell \triangle e_{1,0}(t,r;m,n)\big\Vert_{L^2_{rdr}} \lesssim_\ell \tau_1^{-3-p+} \Big( 1 + \Vert m\Vert_{p,\ell+6} + \Vert n\Vert_{p,\ell+6} \Big)^\ell \Vert n\Vert_{p,\ell+6}.
        \end{align*}
        We also have the bounds 
        \[
        \big|r^{-2}S^k e_{1,0}(t, r)\big|\lesssim_k \tau_1^{-1+},\,r\lesssim t. 
        \]
    \end{lem}

    \begin{proof}
        This follows directly by rewriting \[
        \partial_t^2 = \Big[\dfrac{1}{t}(S-r\partial_r) \Big]^2, \qquad S=t\partial_t + r\partial_r.
        \]
        When this operator falls on the integrand in the variation of constants formula, it is enough to integrate by parts to avoid the term $r\partial_r$. We omit the details.
    \end{proof}

    Finally, the proof of Lemma \ref{lem:step1_2} is completed by combining steps $(i)$-$(iv)$. The proof is complete.
\end{proof}

Let us now summarize what we have accomplished so far. The approximate solution \[
u_1(t,r):= Q_1-\bfcq_{n-1} + v_1, \qquad v_1:= h_0+h_1,
\]
solves the equation (recall that $E_{2,\co}$ and $e_{1,\co}$ are defined in \eqref{eq:step1_11}-\eqref{eq:step1_13}; see also \eqref{eq:hzeroeqn})\begin{equation}\label{eq:step1_15}\begin{aligned}
    & -\partial_t^2 u_1 + \partial_r^2 u_1 + \dfrac{1}{r}\partial_r u_1 - \dfrac{2}{r^2}\sin(2u_1)=  e_1,   
    \\ & e_1 := e_{1,1} -  e_{1,2}, 
    \\ & e_{1,1} :=  E_{2,\co} + e_{1,\co} , 
    \\ & e_{1,2} := e_{1,0} + \sum_{k=1}^3 E_{1,k},
\end{aligned}\end{equation}
where (see \eqref{eq:step1_14})\[
    e_{1,0} = \partial_t^2 h_1^{\inn} + \dfrac{4}{r^2} \Big( \cos(2Q_1-2\bfcq_{n-1})-\cos(2Q_1)\Big) h_1^{\inn},
\]
and the error terms $E_{j,k}$ are given by the following explicit expressions \begin{align*}
E_{1,1} & := \dfrac{2}{r^2}\cos\big(2Q_1 - 2\bfcq_{n-1} + 2h_0\big) \Big( \sin(2h_{1}) - 2h_{1}\Big) ,
\\ E_{1,2} & := \dfrac{4}{r^2}  \Big( \cos\big( 2Q_1 - 2\bfcq_{n-1} + 2 h_0\big) - \cos \big(2Q_1 - 2\bfcq_{n-1}\big) \Big) h_{1} ,
\\ E_{1,3} & := \dfrac{2}{r^2} \sin\big(2Q_1 -2\bfcq_{n-1}+2h_0\big) \Big( \cos(2h_{1}) - 1 \Big).
\end{align*}
We stress that the terms in $e_{1,1}$ are precisely the sum of the correction terms required in Steps $0$ and $1$ (respectively) to enforce the additional vanishing condition to prevent quadratic growth. We never seek to cancel these correction terms at any stage of the inductive process; instead, we consistently carry over the sum of all correction terms to the next step and deal with them at the end. In this regard, it is useful to recall the bounds for $m_1$ proven in Lemma \ref{lem:step1_1}. We can finally formulate the following error bound.

\begin{lem}\label{lem:step1_3}
    We have the error bound \begin{align*}
        \lambda_2^{-1} \big\Vert S^\ell \triangle e_{1,2}(t,r;m,n)\big\Vert_{L^2_{rdr}} \lesssim_\ell \tau_1^{-3-p+} \Big( 1 + \Vert m\Vert_{p,\ell+6} + \Vert n\Vert_{p,\ell+6} \Big)^\ell \Vert n\Vert_{p,\ell+6}.
    \end{align*}
    We also have the bounds
    \[
    \big|r^{-2}S^ke_{1,2}(t, r)\big|\lesssim_k \tau_1^{-1+},\,r\lesssim t. 
    \]
\end{lem}
\begin{proof}
    First of all, recall \eqref{eq:step1_15}, and that the bounds for $e_{1,0}$ were already proven in Lemma \ref{lem:step1_4}, so it is enough to bound the nonlinear terms $E_{1,j}$. Now, write \begin{align*}
     \cos\big(2Q_1-2\bfcq_{n-1}+2h_0\big) & =\cos(2Q_1-2\bfcq_{n-1})
     \\ & + \cos(2Q_1-2\bfcq_{n-1})\Big(\cos(2h_0)-1\Big)
     \\ & -\sin(2Q_1-2\bfcq_{n-1})\sin(2h_0).
    \end{align*}
    It follows then that \begin{align*}
        & \big\Vert S^\ell \cos\big(2Q_1-2\bfcq_{n-1}+2h_0\big)(t,r;m)\big\Vert_{L^2_{rdr}}\lesssim \tau_1^{0+}\Big(1+\Vert m\Vert_{p,\ell}\Big)^\ell,
    \end{align*}
    where the ``slowly" growing factor $\tau_1^{0+}$ is due to the term with no $h_0$, and \begin{align*}
        & \big\Vert S^\ell \triangle \cos\big(2Q_1-2\bfcq_{n-1}+2h_0\big)(t,r;m,n)\big\Vert_{L^2_{rdr}}
        \\ & \lesssim \tau_1^{-p+}\Big(1+\Vert m\Vert_{p,\ell}+\Vert n\Vert_{p,\ell}\Big)^\ell \Vert n\Vert_{p,\ell},
    \end{align*}
    having used \eqref{eq:nzerodifferencingbounds} along with \eqref{eq:Leibnizdifferencing}. The same bounds hold if we replace $\cos(\cdot)$ by $\sin(\cdot)$. Moreover, using the singular bounds in \eqref{linout12} for the contribution of the outer wave component $h_1^{\out}$, and that in \eqref{eq:step1_16} for the contribution of the inner elliptic part $h_1^\inn$, we get that \begin{align*}
        & \Big\Vert S^\ell \Big(\dfrac{1}{r^2}\big(\cos(2h_1)-1\big)\Big)\Big\Vert_{L^2_{rdr}}
        \\ & \lesssim \Big( \sum_{\ell_1=0}^\ell \Big\Vert \dfrac{1}{r^2}S^{\ell_1}h_1\Big\Vert_{L^2_{rdr}}\Big) \Big( \sum_{\ell_2=0}^\ell \big\Vert S^{\ell_2}h_1\big\Vert_{L^\infty}\Big) \Big( 1+\sum_{\ell_3=0}^\ell \big\Vert S^{\ell_3}h_1\big\Vert_{L^\infty}\Big)^{\ell-2}
        \\ & \lesssim \tau_1^{-4+}\Big(1+\Vert m\Vert_{p,\ell+6}\Big)^{\ell^2}.
    \end{align*}
    Proceeding in a similar manner, we infer the differencing bound \begin{align*}
        & \Big\Vert S^\ell \triangle \Big( \dfrac{1}{r^2}\Big(\cos(2h_1)-1\Big)\Big)\Big\Vert_{L^2_{rdr}}
        \\ & \lesssim \tau_1^{-4-p+} \Big(1 + \Vert m\Vert_{p,\ell+6} + \Vert n\Vert_{p,\ell+6} \Big)^{\ell^2} \Vert n\Vert_{p,\ell+6}. 
    \end{align*}
    By combining the previous estimates and using the Leibniz rule along with \eqref{eq:Leibnizdifferencing}, we derive the desired bounds for $E_{1,3}(h_0,h_1)$. The pointwise bound for $E_{1,3}$ follows from the pointwise bounds near $r = 0$ for $h_{1}^{\out}, h_{1}^{\inn}$ stated in {\it{(i), (iii)}}. 
    The remaining terms $E_{1,j}(h_0,h_1)$, $j=1,2$, are treated analogously, so we omit them.
\end{proof}

\bigskip
\noindent
{\bf{Step 2}}: {\it{Inductive construction of the correction $h_{j+1}$.}} We now assume that we have already constructed the first $j+1$ corrections $h_k$, $k=0,1,...,j$. Define \[
v_j:=\sum_{k=0}^j h_k.
\]
Then, the approximate solution \[
u_j=Q_1-\bfcq_{n-1} + v_j
\]
satisfies the equation (recall $E_{2,\co}$ was defined in \eqref{eq:step1_13}) \begin{equation}\label{eq:step2_1}\begin{aligned}
    & -\partial_t^2 u_j + \partial_r^2 u_j + \dfrac{1}{r}\partial_r u_j - \dfrac{2}{r^2}\sin(2u_j) = e_j,
    \\ & e_j := e_{j,1} - e_{j,2},
    \\ & e_{j,1} := E_{2,\co} + \sum_{k=1}^j e_{k,\co},
    \\ & e_{k,\co} := m_k \Big( \cos(2Q_1)-1\Big),
    \\ & e_{j,2} := e_{j,0} + \sum_{k=1}^3 E_{j,k}, 
\end{aligned}\end{equation}
where \[
e_{j,0} := \partial_t^2 h_j^{\inn}   +   \dfrac{4}{r^2}\Big( \cos(2Q_1-2\bfcq_{n-1}) - \cos(2Q_1)\Big) h_{j}^{\inn},
\]
and the error terms \begin{align*}
    E_{j,1} & := \dfrac{2}{r^2}\cos\big(2Q_1 - 2\bfcq_{n-1} + 2h_0+...+2h_{j-1}\big) \Big( \sin(2h_{j}) - 2h_{j}\Big) ,
    \\ E_{j,2} & := \dfrac{4}{r^2}  \Big( \cos\big( 2Q_1 - 2\bfcq_{n-1} + 2 h_0+...+2h_{j-1}\big) - \cos \big(2Q_1 - 2\bfcq_{n-1}\big) \Big) h_{j} ,
    \\ E_{j,3} & := \dfrac{2}{r^2} \sin\big(2Q_1 -2\bfcq_{n-1}+2h_0+...+2h_{j-1}\big) \Big( \cos(2h_{j}) - 1 \Big).
\end{align*}
We will inductively cancel the term $e_{j,2}$. This will require adding, at each step of the induction, a new correction term $e_{k,\co}$, needed to enforce the vanishing condition. This is then added to the new correction term $e_{j+1,1}=e_{j,1}+e_{j+1,\co}$ which is carried forward to the next inductive step.

Again, the error term $e_{j,0}$ accounts for the fact that (just as what we did in \eqref{eq:step1_2}-\eqref{eq:step1_3}), a priori, we would like $h_j$ to satisfy the equation \[
- \partial_t^2 h_j + \partial_r^2 h_j + \dfrac{1}{r} \partial_r h_j - \dfrac{4}{r^2}\cos(2Q_1 - 2\bfcq_{n-1}) h_j =  e_{j-1,2} + e_{j,\co}.
\]
However, we will solve this equation only approximately. The difference between the actual equation satisfied by $h_j$ and the above ``ideal equation" is $e_{j,0}$, in other words, the error made by approximating the last equation.

In the sequel, we suppose that the increments $h_k$, the error $e_{j,2}$, and the correction terms $m_k$ obtained up to this stage satisfy (for all $1\leq k \leq j$; inductive hypothesis) \begin{align}\label{eq:step2_4}
    \big\Vert S^\ell h_k (t,\cdot;m) \big\Vert_{H^1_{rdr}} & \lesssim_{k,\ell} \tau_1^{-1-k+} \Big( 1+ \Vert m\Vert_{p,\ell+4k+2}\Big)^{C_{k,\ell}} , \nonumber
    \\  \big\Vert S^\ell \triangle h_k (t,\cdot;m,n) \big\Vert_{H^1_{rdr}} & \lesssim_{k,\ell} \tau_1^{-1-p-k+} \Big( 1+ \Vert m\Vert_{p,\ell+4k+2}\Big)^{C_{k,\ell}}  \Vert n\Vert_{p,\ell+4k+2} , \nonumber
    \\  \lambda_2^{-1}\big\Vert S^\ell e_{j,2} (t,\cdot;m) \big\Vert_{L^2_{rdr}} & \lesssim_{j,\ell} \tau_1^{-2-j+} \Big( 1+ \Vert m\Vert_{p,\ell+4j+2}\Big)^{C_{j,\ell}}  ,
    \\  \lambda_2^{-1}\big\Vert S^\ell \triangle e_{j,2} (t,\cdot;m,n) \big\Vert_{L^2_{rdr}} & \lesssim_{j,\ell} \tau_1^{-2-p-j+} \Big( 1+ \Vert m\Vert_{p,\ell+4j+2}\Big)^{C_{k,\ell}}  \Vert n\Vert_{p,\ell+4j+2}  , \nonumber
    \\  \big\vert (t\partial_t)^\ell m_k(t;m) \big\vert & \lesssim_{k,\ell} \tau_1^{-k+} \Big( 1+ \Vert m\Vert_{p,\ell+4k}\Big)^{C_{k,\ell}}  , \nonumber
    \\  \big\vert (t\partial_t)^\ell \triangle m_k(t;m,n) \big\vert & \lesssim_{k,\ell} \tau_1^{-p-k+} \Big( 1+ \Vert m\Vert_{p,\ell+4k}\Big)^{C_{k,\ell}}  \Vert n\Vert_{p,\ell+4k} . \nonumber
\end{align}
That being said, we introduce and decompose the subsequent increment $h_{j+1}$ as \[
h_{j+1} = h_{j+1}^{\out} + h_{j+1}^{\inn}
\]
where $h_{j+1}^{\out}$ will solve an outer wave equation, whereas $h_{j+1}^{\inn}$ an inner elliptic equation, namely, \begin{align}\label{eq:step2_2}
- \partial_t^2 h_{j+1}^{\out} + \partial_r^2 h_{j+1}^{\out} + \dfrac{1}{r} \partial_r h_{j+1}^{\out} - \dfrac{4}{r^2} \cos(2\bfcq_{n-1}) h_{j+1}^{\out}  =  e_{j,2}, 
\end{align}
and \begin{equation}\label{eq:step2_3}\begin{aligned}
& \partial_r^2 h_{j+1}^{\inn}  + \dfrac{1}{r} \partial_r h_{j+1}^{\inn} - \dfrac{4}{r^2} \cos(2Q_1) h_{j+1}^{\inn}
\\ & = e_{j+1,\co} + \dfrac{4}{r^2} \Big( \cos(2Q_1-2\bfcq_{n-1}) - \cos(2\bfcq_{n-1})\Big) h_{j+1}^{\out}. 
\end{aligned}\end{equation}
It is worth stressing that equation \eqref{eq:step2_2} can also be thought of as \begin{align*}
& - \partial_t^2 h_{j+1}^{\out} + \partial_r^2 h_{j+1}^{\out} + \dfrac{1}{r} \partial_r h_{j+1}^{\out} - \dfrac{4}{r^2} \cos(2Q_1 - 2\bfcq_{n-1})h_{j+1}^{\out}  
\\ & =  e_{j,2} - \dfrac{4}{r^2}\Big(\cos(2Q_1-2\bfcq_{n-1})- \cos(2\bfcq_{n-1}) \Big) h_{j+1}^{\out} , 
\end{align*}
which is how we typically write it to compute the error, once we put it together with the equation for $h_{j+1}^{\inn}$.

The main result of this step is the following lemma. 
\begin{lem}\label{lem:step2_1}
    Under the previous inductive assumptions there exist corrections $h_{j+1}$ and $m_{j+1}$ satisfying the bounds \begin{align*}
        & \Vert S^\ell h_{j+1}(t,\cdot; m)\Vert_{H^1_{rdr}}  \lesssim_{j,\ell} \tau_1^{-2-j+}\Big(1+\Vert m\Vert_{p,\ell+4(j+1)+2}\Big)^{C_{j+1,\ell}} ,
        \\  & \Vert S^\ell \triangle h_{j+1}(t,\cdot; m,n )\Vert_{H^1_{rdr}}  \lesssim_{j,\ell}  \tau_1^{-2-j-p+}\Big(1+\Vert m\Vert_{p,\ell+4(j+1)+2}\Big)^{C_{j+1,\ell}} \Vert n\Vert_{p,\ell+4(j+1)+2} ,
        \\ & \big\vert (t\partial_t)^\ell m_{j+1}(t;m)\big\vert \lesssim_{j,\ell}   \tau_1^{-j-1+}\Big(1+\Vert m\Vert_{p,\ell+4(j+1)}\Big)^{C_{j+1,\ell}} ,
        \\ & \big\vert (t\partial_t)^\ell \triangle m_{j+1}(t;m , n)\vert \lesssim_{j,\ell}  \tau_1^{-j-1-p+}\Big(1+\Vert m\Vert_{p,\ell+4(j+1)}\Big)^{C_{j+1,\ell}} \Vert n\Vert_{p,\ell+4(j+1)}  ,
    \end{align*}
    for some suitable constant $C_{j+1,\ell}$. The functions $h_{j+1}$ admit decompositions
    \begin{equation}\label{eq:hj+1Taylornearzero}
    h_{j+1}(t, r) = c_{j+1}(t)r^2 + g_{j+1}(t, r),\,r\lesssim t,
    \end{equation}
    satisfying the bounds 
    \[
\big|(t\partial_t)^k c_{j+1}(t)\big|\lesssim_k \tau_1^{-1-j+},\,\big|S^kg_{j+1}(t, r)\big|\lesssim_k \tau_1^{1-j}r^4.
    \]
    Furthermore, the new approximate solution 
    \[
    u_{j+1}:= Q_{1}-\bfcq_{n-1} + v_{j+1}, \qquad v_{j+1}= \sum_{k=0}^{j+1} h_k,
    \]
    satisfies the equation \begin{align}\label{eq:step2_8}
        & -\partial_t^2 u_{j+1} + \partial_r^2 u_{j+1} +\dfrac{1}{r} \partial_r u_{j+1}  - \dfrac{2}{r^2} \sin(2u_{j+1}) = e_{j+1},
    \end{align}
    with \begin{equation}\label{eq:step2_6}\begin{aligned}
        & e_{j+1} := e_{j+1,1} - e_{j+1,2},
        \\ & e_{j+1,1} := E_{2,\co} + \sum_{k=0}^{j+1} e_{k,\co},
        \\ & e_{k,\co} := m_k\Big(\cos(2Q_1)-1\Big),
        \\ & e_{j+1,2} := e_{j+1,0} + \sum_{k=1}^3 E_{j+1,k},
        \\ & e_{j+1,0} := \partial_{t}^2 h_{j+1}^{\inn} + \dfrac{4}{r^2}\Big( \cos(2Q_1-2\bfcq_{n-1}) - \cos(2Q_1)\Big) h_{j+1}^\inn,
    \end{aligned}\end{equation}
    where the error terms $E_{j+1,k}$ are given by \begin{align*}
        & E_{j+1,1}  := \dfrac{2}{r^2}\cos\big(2Q_1 - 2\bfcq_{n-1} + 2h_0+...+2h_{j}\big) \Big( \sin(2h_{j+1}) - 2h_{j+1}\Big) ,
        \\ & E_{j+1,2}  := \dfrac{4}{r^2}  \Big( \cos\big( 2Q_1 - 2\bfcq_{n-1} + 2 h_0+...+2h_{j}\big) - \cos \big(2Q_1 - 2\bfcq_{n-1}\big) \Big) h_{j+1} ,
        \\ & E_{j+1,3}  := \dfrac{2}{r^2} \sin\big(2Q_1 -2\bfcq_{n-1}+2h_0+...+2h_{j}\big) \Big( \cos(2h_{j+1}) - 1 \Big).
    \end{align*}
    Moreover, the error term $e_{j+1,2}$ satisfies bounds analogous to those of $e_{j,2}$ stated above, but with $j$ replaced by $j+1$, and we have that \begin{align}
        & \lambda_2^{-1} \Vert S^\ell e_{j+1,0}(t,\cdot;m)\Vert_{L^2_{rdr}} \lesssim_{\ell}  \tau_1^{-3-j+} \Big( 1 + \Vert m\Vert_{p,\ell+4(j+1)+2} \Big)^{C_{j+1,\ell}}  , \label{eq:step2_7}
        \\ &  \lambda_2^{-1} \Vert S^\ell \triangle e_{j+1,0}(t,\cdot;m,n)\Vert_{L^2_{rdr}} \lesssim_{\ell}  \tau_1^{-3-j-p+} \Big( 1 + \Vert m\Vert_{p,\ell+4(j+1)+2} \Big)^{C_{j+1,\ell}} \Vert n\Vert_{p,\ell+4(j+1)+2} . \nonumber
    \end{align}
\end{lem}

\begin{proof}
    Similar to the proof of Lemma \ref{lem:step1_2}, we split the analysis into multiple steps. The following steps $(i)$-$(iv)$ are devoted to proving both estimates in \eqref{eq:step2_7}.
    
    \bigskip

    \noindent
    \textit{(i) Control of the solution to the outer wave equation \eqref{eq:step2_2}}. This is a direct application of Lemma \ref{lem:linout_1} along with the inductive bounds in  \eqref{eq:step2_4}, from where we infer that equation \eqref{eq:step2_2} admits a solution satisfying the bounds \begin{align}\label{eq:step3_6}
        & \Vert S^\ell h_{j+1}^{\out}(t,\cdot; m)\Vert_{\bch^1_{rdr}}  \lesssim_{j,\ell} \tau_1^{-2-j+}\Big(1+\Vert m\Vert_{p,\ell+4(j+1)+2}\Big)^{C_{j+1,\ell}} ,
        \\  & \Vert S^\ell \triangle h_{j+1}^{\out}(t,\cdot; m,n )\Vert_{\bch^1_{rdr}}  \lesssim_{j,\ell}  \tau_1^{-2-j-p+}\Big(1+\Vert m\Vert_{p,\ell+4(j+1)+2}\Big)^{C_{j+1,\ell}} \Vert n\Vert_{p,\ell+4(j+1)+2} .  \nonumber
    \end{align}

    \medskip

    \noindent
    \textit{(ii) Choice of the parameter $m_{j+1}(t)$ in the inner elliptic equation \eqref{eq:step2_3}.} Once again, setting $S=\lambda_1(t) s$ \begin{align*}
        4m_{j+1}(t):= - \lambda_1^2\int_0^\infty \dfrac{4}{S^2} \Big(\cos(2Q_1-2\bfcq_{n-1})-\cos(2\bfcq_{n-1})\Big) h_{j+1}^{\out} \Phi(S)SdS.
    \end{align*}
    This is the exact analog of step $(ii)$ in Lemma \ref{lem:step1_2}. We recall that this definition then yields \begin{equation}\label{eq:step2_5}\begin{aligned}
        & \bigg\langle \dfrac{4}{r^2}\Big(\cos(2Q_1-2\bfcq_{n-1})-\cos(2\bfcq_{n-1})\Big) h_{j+1}^{\out}, \,  \Phi(\lambda_1\cdot) \bigg\rangle_{L^2_{rdr}} 
        \\ & + m_{j+1} \Big\langle \cos(2Q_1)-1, \,  \Phi(\lambda_1\cdot)\Big\rangle_{L^2_{rdr}}=0 , 
    \end{aligned}\end{equation}
    having used \eqref{eq:explicitintegrals}. Repeating the proof of Lemma \ref{lem:step1_1} verbatim, we obtain \begin{align}
        & \big\vert (t\partial_t)^\ell m_{j+1}(t;m)\big\vert \lesssim_{j,\ell}   \tau_1^{-j-1+}\Big(1+\Vert m\Vert_{p,\ell+4(j+1)}\Big)^{C_{j+1,\ell}} ,   \label{eq:step2_9}
        \\ & \big\vert (t\partial_t)^\ell \triangle m_{j+1}(t;m , n)\vert \lesssim_{j,\ell}  \tau_1^{-j-1-p+}\Big(1+\Vert m\Vert_{p,\ell+4(j+1)}\Big)^{C_{j+1,\ell}} \Vert n\Vert_{p,\ell+4(j+1)} . \nonumber
    \end{align}

    \medskip

    \noindent
    \textit{(iii) Control of the solution to the inner elliptic equation \eqref{eq:step2_3}.} Once again, using the same arguments as in the proof of step $(iii)$ of Lemma   \ref{lem:step1_2}, we infer the existence of $h_{j+1}^{\inn}$ satisfying the bounds (recall also the bounds at he beginning of proof of Lemma \ref{lem:step1_3} and \eqref{eq:step2_4}) \begin{align}\label{eq:step3_7}
        & \Vert S^\ell h_{j+1}^{\inn}(t,\cdot;m)\Vert_{H^1_{rdr}} \lesssim_{j,\ell} \tau_1^{-3-j+} \Big(1+\Vert m\Vert_{p,\ell+4(j+1)}\Big)^{C_{j,\ell+2}+\ell},
        \\ & \Vert S^\ell \triangle h_{j+1}^{\inn}(t,\cdot;m,n)\Vert_{H^1_{rdr}} \lesssim_{j,\ell} \tau_1^{-3-j-p+} \Big(1+\Vert m\Vert_{p,\ell+4(j+1)}\Big)^{C_{j,\ell+2}+\ell} \Vert n\Vert_{p,\ell+4(j+1)}.   \nonumber
    \end{align}
    It follows then that $h_{j+1}=h_{j+1}^\out+h_{j+1}^\inn$ satisfies the desired bounds, provided \[
    C_{j+1,\ell}\geq C_{j,\ell+2}+\ell.
    \]
    Moreover, arguing in the same fashion as in Remark \ref{rem:step1_1}, now using that $m_{j+1}$ decays as $\tau_1^{-1-j+}$ (see \eqref{eq:step2_9} above), together with the faster decay of $h_{j+1}^\out$ in \eqref{eq:step3_6}, one concludes the singular bound for the inner elliptic part \begin{align}\label{eq:step2_10}
    \Big\Vert \dfrac{1}{r^2} h_{j+1}^\inn\Big\Vert_{L^2_{rdr}}\lesssim \tau_1^{-2-j+}  .
    \end{align}
    The proof of the decomposition \eqref{eq:hj+1Taylornearzero} including the bounds following it is accomplished just like for the case $j = 0$, taking advantage of Lemma~\ref{lem:Tayloerexpansion}, and we omit it here. 
    \medskip

    \noindent
    \textit{(iv) Control of the error term $e_{j+1,0}$ in \eqref{eq:step2_6}}. This step is identical to the proof of step $(iv)$ of Lemma \ref{lem:step1_2}, so we omit it.

    This concludes the proof of \eqref{eq:step2_7}. We now seek to prove the bounds for $e_{j+1,2}$ in \eqref{eq:step2_6}. Recall that our goal is to show that \begin{align}
        \lambda_2^{-1}\big\Vert S^\ell e_{j,2} (t,\cdot;m) \big\Vert_{L^2_{rdr}} & \lesssim_{j,\ell} \tau_1^{-2-j+} \Big( 1+ \Vert m\Vert_{p,\ell+4j+2}\Big)^{C_{j,\ell}}  , \label{eq:step3_8}
    \\  \lambda_2^{-1}\big\Vert S^\ell \triangle e_{j,2} (t,\cdot;m,n) \big\Vert_{L^2_{rdr}} & \lesssim_{j,\ell} \tau_1^{-2-p-j+} \Big( 1+ \Vert m\Vert_{p,\ell+4j+2}\Big)^{C_{k,\ell}}  \Vert n\Vert_{p,\ell+4j+2}. \nonumber   
    \end{align}
    
    \medskip

    \noindent
    $\bullet$ \textit{Estimate for $E_{j+1,1}$}. Writing \begin{align*}
    & \cos\Big(2Q_1-2\bfcq_{n-1}+2\sum_{k=0}^j h_k\Big)
    \\ & =\cos(2Q_1-2\bfcq_{n-1})\cos\Big(2\sum_{k=0}^j h_k\Big)-\sin(2Q_1-2\bfcq_{n-1})\sin\Big(2\sum_{k=0}^j h_k\Big),
    \end{align*} 
    using the first bounds in the inductive hypothesis \eqref{eq:step2_4}, as well as a Taylor expansion, as well as Leibniz rule and \eqref{eq:Leibnizdifferencing}, we infer that \begin{align*}
        \Big\Vert S^\ell \cos\Big( 2Q_1-2\bfcq_{n-1}+\sum_{k=0}^j h_k\Big)\Big\Vert_{H^1_{rdr}} \lesssim_{\ell,j} \tau_1^{0+} \Big(1+\Vert m\Vert_{p,\ell+4j+2}\Big)^{C_{j,\ell}},
    \end{align*}
    where the $\tau_1^{0+}$ bound comes from the multiplication of both cosines above, and \begin{align*}
        & \Big\Vert S^\ell \triangle \cos\Big( 2Q_1-2\bfcq_{n-1}+\sum_{k=0}^j h_k\Big)(t,r;m,n)\Big\Vert_{H^1_{rdr}} 
        \\ & \lesssim_{\ell,j} \tau_1^{-p+} \Big(1+\Vert m\Vert_{p,\ell+4j+2}\Big)^{C_{j,\ell}}\Vert n\Vert_{p,\ell+4j+2},
    \end{align*}
    On the other hand, for the cubic in $h_{j+1}$ factor, we have that \begin{align*}
        & \Big\Vert S^\ell \Big( \dfrac{1}{r^2}\Big(\sin(2h_{j+1})-2h_{j+1}\Big)\Big)\Big\Vert_{L^2_{rdr}} 
        \\ & \lesssim_{\ell} \sum_{\ell_1,\ell_2\leq \ell} \Big\Vert \dfrac{S^{\ell_1} h_{j+1}}{r^2}\Big\Vert_{L^2_{rdr}}\, \big\Vert S^{\ell_2} h_{j+1}\big\Vert_{L^\infty} \Big( 1 + \sum_{r=0}^\ell \big\Vert S^r h_{j+1}\big\Vert_{L^\infty} \Big)^{\ell-2}.
    \end{align*}
    Then, it suffices to recall the already proven bounds for $h_{j+1}$ (see \eqref{eq:step3_6} and \eqref{eq:step3_7} for the regular bounds, and \eqref{linout12} for the singular bounds for the outer part and \eqref{eq:step2_10} for the inner part). More specifically, we can bound the preceding quantities by (here $\ell_1+\ell_2\leq \ell$) \begin{align*}
        & \Big\Vert \dfrac{S^{\ell_1} h_{j+1}}{r^2}\Big\Vert_{L^2_{rdr}}\, \big\Vert S^{\ell_2} h_{j+1}\big\Vert_{L^\infty} \Big( 1 + \sum_{r=0}^\ell \big\Vert S^r h_{j+1}\big\Vert_{L^\infty} \Big)^{\ell-2}
        \\ & \lesssim \tau_1^{-j-1+} \tau_1^{-j-2+}\Big( 1+\Vert m\Vert_{p,\ell+4j+4}\Big)^{D_{\ell,j}},
    \end{align*}
    having introduced, for the sake of simplicity, the notation \[
    D_{\ell,j} := \ell(C_{j,\ell+2}+\ell).
    \]
    Estimate \eqref{eq:step3_8} for $E_{j+1,1}$ follows then from the Leibniz rule, \begin{align*}
        \lambda_2^{-1}\big\Vert S^\ell E_{j+1,1}\big\Vert_{L^2_{rdr}} \lesssim_{j,\ell} \tau_1^{-3-2j+} \Big(1+\Vert m\Vert_{p,\ell+4j+4}\Big)^{D_{\ell,j}+C_{j,\ell}}.
    \end{align*}
    Mutatis mutandis, using now estimate \ref{eq:Leibnizdifferencing}, one derives the differencing bound \[
    \lambda_2^{-1}\big\Vert S^\ell \triangle E_{j+1,1}\big\Vert_{L^2_{rdr}} \lesssim_{j,\ell} \tau_1^{-3-2j-p+} \Big(1+\Vert m\Vert_{p,\ell+4j+4}\Big)^{D_{\ell,j}+C_{j,\ell}} \Vert n\Vert_{p,\ell+4j+4} .
    \]
    
    \medskip

    \noindent
    $\bullet$ \textit{Estimate for $E_{j+1,2}$}. As usual, we begin by expanding
    \begin{align*}
        & \cos\big(2Q_1-2\bfcq_{n-1}+\sum_{k=0}^j h_k\big)-\cos(2Q_1-2\bfcq_{n-1})
        \\ & = \cos(2Q_1-2\bfcq_{n-1})\left(\cos\Big(\sum_{k=0}^j h_k\Big)-1\right)-\sin(2Q_1-2\bfcq_{n-1})\sin\Big(\sum_{k=0}^j h_k\Big).
    \end{align*}
    Then, we can write
    \begin{align*}
        \Big\Vert S^\ell \Big( \dfrac{2}{r^2}\Big(\cos\big(2Q_1-2\bfcq_{n-1}+\sum_{k=0}^j h_k\big)-\cos(2Q_1-2\bfcq_{n-1})\Big)\Big)\Big\Vert_{L^2_{rdr}}\lesssim F_1+F_2,
    \end{align*}
    where \begin{align*}
        F_1 & \lesssim_\ell  \sum_{\ell_1+\ell_2=\ell} \big\Vert S^{\ell_1}\big(\cos(2Q_1-2\bfcq_{n-1})\big)\big\Vert_{L^\infty} \Big\Vert S^{\ell_2}\Big( \dfrac{1}{r^2}\Big( \cos\Big(\sum_{k=0}^j h_k\Big)-1\Big)\Big)\Big\Vert_{L^2_{rdr}},
        \\ F_2 & \lesssim_\ell \sum_{\ell_1+\ell_2=\ell} \big\Vert S^{\ell_1}\big(\sin(2Q_1-2\bfcq_{n-1})\big)\big\Vert_{L^\infty} \Big\Vert S^{\ell_2}\Big( \dfrac{1}{r^2}\sin\Big(\sum_{k=0}^j h_k\Big)\Big)\Big\Vert_{L^2_{rdr}}.
    \end{align*}
    Then, proceeding exactly as in the case of $E_{j+1,1}$, this time using the inductive hypothesis bounds \eqref{eq:step2_4} for the $h_k$, $k\leq j$, we obtain that (here $\ell_2\leq \ell$) \begin{align*}
        & \Big\Vert S^{\ell_2}\Big( \dfrac{1}{r^2}\Big( \cos\Big(\sum_{k=0}^j h_k\Big)-1\Big)\Big)\Big\Vert_{L^2_{rdr}},  \lesssim_{j,\ell} \tau_1^{-3+} \Big( 1+ \Vert m \Vert_{p,\ell+4j+2}\Big)^{C_{j,\ell}},
        \\ &  \Big\Vert S^{\ell_2}\Big( \dfrac{1}{r^2}\sin\Big(\sum_{k=0}^j h_k\Big)\Big)\Big\Vert_{L^2_{rdr}} \lesssim_{j,\ell} \tau_1^{-1+} \Big( 1+ \Vert m \Vert_{p,\ell+4j+2}\Big)^{C_{j,\ell}}.
    \end{align*}
    Therefore, we infer that (recall the extra factor $h_{j+1}$ in the definition of $E_{j+1,2}$) \begin{align*}
        \lambda_2^{-1}\big\Vert S^\ell E_{j+1,2}\big\Vert_{L^2_{rdr}} & \lesssim \big( F_1+F_2\big) \Big(\sum_{\ell_1\leq \ell} \big\Vert S^{\ell_1} h_{j+1}\Vert_{L^\infty}\Big)
        \\ &\lesssim_{\ell,j} \tau_1^{-3-j+} \Big( 1+\Vert m \Vert_{p,\ell+4j+2}\Big) ^{C_{j+1,\ell}},
    \end{align*}
    by increasing $C_{j+1,\ell}$ if necessary. The differencing bound follows exactly the same ideas, so we omit it. This concludes the proof of \eqref{eq:step3_8} for $E_{j+1,2}$.

    Finally, the estimate for $E_{j+1,3}$ follows similar lines to those of the previous case. In this case, it is enough to use the bound \begin{align*}
        \Big\Vert S^\ell \Big( \dfrac{2}{r^2} \Big(\cos(2h_{j+1})-1\Big)\Big)\Big\Vert_{L^2_{rdr}} \lesssim_{\ell,j} \tau_1^{-3-2j+} \Big( 1 + \Vert m\Vert_{p,\ell+4j+4}\Big)^{D_{\ell,j}},
    \end{align*}
    and proceed as before. The proof is complete.
\end{proof}

\bigskip
\noindent
{\bf{Step 3}}: {\it{Proof of Proposition \ref{prop:approximateinnerbubble}}}.  Repeated applications of Lemma \ref{lem:step2_1} sufficiently many times yield an approximate solution \[
u_N=Q_1-\bfcq_{n-1} + \sum_{j=0}^N h_j, \qquad N\gg1,
\]
which generates the error \begin{align*}
    -\partial_t^2 u_N + \partial_r^2 u_N + \dfrac{1}{r}\partial_r u_N - \dfrac{2}{r^2}\sin(2u_N)=e_N, \qquad e_N=e_{N,1}-e_{N,2}.
\end{align*}
We have the bounds \begin{align}
        & \lambda_2^{-1}\big\Vert S^\ell e_{N,2} (t,\cdot;m) \big\Vert_{L^2_{rdr}}   \lesssim_{N,\ell} \tau_1^{-2-N+} \Big( 1+ \Vert m\Vert_{p,\ell+4N+2}\Big)^{C_{N,\ell}}  ,   \label{eq:step3_5}
    \\  & \lambda_2^{-1}\big\Vert S^\ell \triangle e_{N,2} (t,\cdot;m,n) \big\Vert_{L^2_{rdr}}   \lesssim_{N,\ell} \tau_1^{-2-p-N+} \Big( 1+ \Vert m\Vert_{p,\ell+4N+2}\Big)^{C_{k,\ell}}  \Vert n\Vert_{p,\ell+4N+2}  \nonumber
    \end{align}
    On the other hand, recall that (recall \eqref{eq:step1_13} and \eqref{eq:step2_1}) \[
    e_{N,1}= \bigg( - 8 (2m\overline{\lambda}_2 + m^2)   + \sum_{k=1}^N m_k\bigg) \Big( \cos(2Q_1)-1\Big), 
    \]
    and that the $m_k$ depend on $m$ as specified in part $(ii)$ of the proof Lemma \ref{lem:step2_1}. The next lemma gives a precise control of $e_{N,1}$.
    \begin{lem}
        Given $N\gg1$, there exists $t_0=t_0(N)>0$ sufficiently small such that there is a function $m\in C^\infty((0,t_0])$ with \[
        \Vert m\Vert_{p,10N}\leq 1,
        \]
        for some fixed $p\geq \tfrac12$ and such that \begin{align}\label{eq:step3_4}
        \lambda_2^{-1} \Vert e_{N,1}\Vert_{L^2_{rdr}} \lesssim \tau_1^{-N+1+}.
        \end{align}
    \end{lem}
    \begin{proof}
        We will (approximately) solve the equation \begin{align}\label{eq:step3_1}
        - 8 (2m\overline{\lambda}_2 + m^2)   + \sum_{k=1}^N m_k=0,
        \end{align}
        via a Picard iteration argument (in $m$). We say ``approximately" since we will stop the iteration after a finite number of steps. Let us denote by \[
        \mathcal{M}_N(t;m):= \dfrac{1}{16\overline{\lambda}_2}\sum_{k=1}^Nm_k-\dfrac{1}{2\overline{\lambda}_2}m^2,
        \]
        In light of the bounds \eqref{eq:step2_4}, we have that (for $p\geq \tfrac12$) \begin{equation}\label{eq:step3_2}\begin{aligned}
            \big\vert (t\partial_t)^\ell \mathcal{M}_N(t;m)\big\vert & \lesssim_{\ell, N} \tau_1^{-1+} \Big(1+\Vert m\Vert_{p,\ell +2N+2}\Big)^{C_{N,\ell}},
            \\ \big\vert (t\partial_t)^\ell \triangle \mathcal{M}_N(t;m,n) \big\vert & \lesssim_{\ell,N}  \tau_1^{-1-p+}\Big(1+\Vert m\Vert_{p,\ell +2N+2}\Big)^{C_{N,\ell}}\Vert n\Vert_{p,\ell+2N+2} .
        \end{aligned}\end{equation}
        Let us now set \[
        P_N(t;m) := \mathcal{M}_N(t;m) - \mathcal{M}_N(t;0)  , \qquad d(t):= \mathcal{M}_N(t;0).
        \]
        Then, we can reformulate equation \eqref{eq:step3_1} as simply \begin{align*}
            m=P_N(m)+d,
        \end{align*}
        where we have suppressed the dependence on $t$. We note right away that, by \eqref{eq:step3_2}, we have \begin{align*}
            \vert (t\partial_t)^\ell d\vert \lesssim_{\ell,N} \tau_1^{-1+}.
        \end{align*}
        To simplify notation, from now on we write $P$ in place of $P_N$. Define inductively, 
        \begin{align*}
            P^{(1)}(d)=P(d), \qquad P^{(k+1)}(d)=P^{(1)}\Big(d+P^{(k)}(d)\Big), \qquad k\geq1,
        \end{align*}
        and further set
        \begin{align*}
            m^{(k)}:=d+P^{(k)}(d), \quad \hbox{ in particular } \quad P\big(m^{(k)}\big)=P^{(k+1)}(d).
        \end{align*}
        By definition we have $P^{(1)}(d+f)-P^{(1)}(d)=\triangle P(d,f)$, and hence, \begin{align}\label{eq:step3_3}
        P^{(2)}(d)=P^{(1)}(d)+\triangle P\Big(d,P(d)\Big)
        \end{align}
        Similarly, using \eqref{eq:step3_3}, we obtain that \begin{align*}
            P^{(3)}(d)-P^{(2)}(d) &= P^{(1)}\Big(d+P^{(2)}(d)\Big) - P^{(1)}\Big(d+P^{(1)}(d)\Big)
            \\ & = \triangle P\Big(d+P^{(1)}(d),\, \triangle P\big( d,P(d)\big)\Big)
            \\ & = \triangle P\Big( m^{(1)}(d), \, \triangle P\big(d,P(d)\big) \Big),
        \end{align*} 
        and
        \begin{align*}
            P^{(4)}(d)-P^{(3)}(d) &= P^{(1)}\Big(d+P^{(3)}(d)\Big)-P^{(1)}\Big(d+P^{(2)}(d)\Big)
            \\ & = \triangle P\Big(m^{(2)}(d),\triangle P\Big(m^{(1)}(d),\triangle P(d,P(d))\Big)  \Big).
        \end{align*}
        Then, generalizing, we derive the relation \begin{align*}
        & P^{(k+1)}(d)-P^{(k)}(d)
        \\ &= \triangle P\bigg( m^{(k-1)}(d), \, \triangle P\Big( m^{(k-2)}(d), \triangle P\big(m^{(k-3)}(d),..., \triangle P\big(d,P(d)\big)...\Big)\bigg)
        \end{align*}
        Finally, observing that $P_N$ satisfies the same bound \eqref{eq:step3_2} as $\mathcal{M}_N$, we conclude that \begin{align*}
            \big\vert m^{(N+1)}-d-P\big(m^{(N+1)}\big)\big\vert & = \big\vert P^{(N+1)}(d)-P\big(d+P^{(N+1)}(d)\big)\big\vert 
            \\ & = \big\vert P^{(N+1)}(d)-P^{(N+2)}(d)\big\vert,
        \end{align*}
        satisfies the bound \[
        \big\vert (t\partial_t)^\ell \big( m^{(N+1)}-d-P\big(m^{(N+1)}\big)\big)\big\vert \lesssim_{\ell,N} \tau_1^{-N+}.
        \]
        The desired bound \eqref{eq:step3_4} for $e_{N,1}$ is then an immediate consequence. The proof is complete.
\end{proof}
Taking the function $m$ constructed in the preceding lemma, we deduce Proposition \ref{prop:approximateinnerbubble} from the bounds \eqref{eq:step3_5} with $\ell=0$.

%%%%%%%%%%%%%%%%%%%%%%%%%%%%%%%%%%%%%%%%%%%
%%%%%%%%%%%%%%%%%%%%%%%%%%%%%%%%%%%%%%%%%%%
%%%%%%%%%%%%%%%%%%%%%%%%%%%%%%%%%%%%%%%%%%%
%%%%%%%%%%%%%%%%%%%%%%%%%%%%%%%%%%%%%%%%%%%
%%%%%%%%%%%%%%%%%%%%%%%%%%%%%%%%%%%%%%%%%%%
%%%%%%%%%%%%%%%%%%%%%%%%%%%%%%%%%%%%%%%%%%%

%\newpage
\bigskip
\section{Technical preparations for the final perturbation step towards the exact $n$ bubble solution}\label{sec:prep_final}
Throughout this section, we will work with the innermost time scale
\begin{align}\label{linin12}
    \tau_1:= \int_t^{t_0}\lambda_1(s)ds.
\end{align}
In the same spirit as for the previous steps, we now need to solve (and prove some bounds for the solution of) the equation \begin{align}\label{linin1}
-\partial_t^2\varepsilon+\partial_r^2\varepsilon+\dfrac{1}{r}\partial_r\varepsilon-\dfrac{4}{r^2}\cos(2Q_1)\varepsilon=f.
\end{align}
As in Section \ref{sec:linout}, a particularly useful observation that we will use later is that 
\begin{align}\label{linin11}
\tau_1 \sim \dfrac{\lambda_1(t)}{\lambda_2(t)}\bigg(1 + O \Big( \dfrac{  \lambda_2'(t)}{\big(\lambda_2(t)\big)^2} \Big)\bigg)  \quad \hbox{ for } \quad t\sim 0+.
\end{align}
Moreover, by direct calculations we have that \begin{align*}
    \omega_1 := \dfrac{\partial_{\tau_1}\lambda_1}{\lambda_1}=-\dfrac{\lambda_1'}{\lambda_1^2} \sim  \dfrac{1}{\tau_1}, \qquad \big\vert \dot{\omega}_1 \big\vert \sim  \dfrac{1}{\tau_1^2}.
\end{align*}
The first bound above can be written as (here $s> \tau_1$) \[
\dfrac{d}{ds} \log \lambda_1(s)\sim \dfrac{1}{s}, \qquad \hbox{which in turn implies that}, \qquad \lambda_1(s) \gtrsim \lambda_1(\tau_1)\dfrac{s}{\tau_1}.
\]
We also note the following bounds, which follow from the previous inequality and will be particularly useful in what follows (here $\sigma_1>\tau_1$)
\begin{equation}\label{linin5}\begin{aligned}
    \lambda_1(\tau_1)\int_{\tau_{1}}^{\sigma_1} \dfrac{ds}{\lambda_1(s)} & \lesssim \tau_1 \ln\Big(\dfrac{\sigma_1}{\tau_1}\Big),
    \\ \bigg\vert \partial_{\tau_1} \bigg( \lambda_1(\tau_1) \int_{\tau_1}^{\sigma_1} \dfrac{ds}{\lambda_1(s)} \bigg) \bigg\vert & \lesssim \ln\Big(\dfrac{\sigma_1}{\tau_1}\Big) .
\end{aligned}\end{equation}
From now on, we work with the precise inner time $\tau_1$ instead of the approximation $\tau$ used earlier for simplicity (see \eqref{linout20}).

We will follow a strategy similar to that of Section \ref{sec:linout}. We define the (physical) working norms $X_1$ for the solution $\varepsilon$, and $Y_1$ for the forcing term, respectively, \begin{align}\label{linin9}
    \Vert \varepsilon\Vert_{X_1} & = \sup_{0 < t \leq t_0} \tau_1^{N_0} \Vert \varepsilon(t,\cdot)  \Vert_{ \boldsymbol{\mathcal{H}}^1_{rdr} }  + \sup_{0 < t \leq t_0} \tau_1^{N_1} \Vert S \varepsilon(t,\cdot)  \Vert_{\boldsymbol{\mathcal{H}}^1_{rdr}}  + \sup_{0 < t \leq t_0} \tau_1^{N_2} \Vert  S^2 \varepsilon \Vert_{\boldsymbol{\mathcal{H}}^1_{rdr}}\nonumber
    \\ & + \sup_{0 < t \leq t_0} \tau_1^{N_2} \lambda_1^{-1}(t) \Vert \dfrac{1}{r^2}\varepsilon(t,\cdot)  \Vert_{L^2_{rdr}}  +  \sup_{0 < t \leq t_0} \tau_1^{N_2} \lambda_1^{-1/2}(t) \Vert  \dfrac{1}{r} S\varepsilon  \Vert_{ L^4_{rdr}},
\end{align}
and
\begin{equation}\label{linin10}\begin{aligned}
    \Vert f\Vert_{Y_1} & =  \sup_{0 < t \leq t_0} \tau_1^{N_0+2} \lambda_1^{-1}(t) \Vert f(t,\cdot)  \Vert_{L^2_{rdr}}  +  \sup_{0 < t \leq t_0} \tau_1^{N_1+2} \lambda_1^{-1}(t) \Vert Sf(t,\cdot)  \Vert_{L^2_{rdr}}  
    \\ & + \sup_{0 < t \leq t_0} \tau_1^{N_2+2} \lambda_1^{-1}(t) \Vert S^2f(t,\cdot)  \Vert_{L^2_{rdr}}, 
\end{aligned}\end{equation}
where $N_j=N/2$, $j=0,1,2$, and pass to the rescaled variables \[
\wt \varepsilon(\tau_1,R) = R^{1/2} \varepsilon \big( t(\tau_1),\, \lambda_1^{-1} R\big), \qquad R:=\lambda_1(t) r.
\]
Here we have again defined the energy norm (which is just a scale-adapted version of the $H^1$) as \begin{align}\label{linin13}
\Vert \varepsilon(t,\cdot)\Vert_{\boldsymbol{\mathcal{H}}^1_{rdr}} := \Vert \mathcal{L}^{1/2}_t \varepsilon\Vert_{L^2_{rdr}} + \Vert \partial_t \varepsilon\Vert_{L^2_{rdr}}+\lambda_1(t) \,\Vert \varepsilon\Vert_{L^2_{rdr}},
\end{align} 
having denoted by $\mathcal{L}_t:=-\partial_r^2-\tfrac{1}{r}\partial_r+\tfrac{4}{r^2}\cos(2Q_1)$. The following Lemma is the main result of this section.

\begin{lem}\label{lem:linin_1}
    Let $N\gg 1$ be sufficiently large. Assume that $\Vert f\Vert_{Y_1}<\infty$, then, equation \eqref{linin1} admits a solution satisfying the bound \begin{align*}
        \Vert \varepsilon\Vert_{X_1}\lesssim N^{-1} \Vert f\Vert_{Y_1}.
    \end{align*}
\end{lem}

\begin{proof}
    As noted earlier, and following the approach of Section \ref{sec:linout}, we first pass to the rescaled variables that make the Schr\"odinger operator time independent, \[
    \wt \varepsilon(\tau_1,R) = R^{1/2} \varepsilon \big( t(\tau_1),\, \lambda_1^{-1} R\big), \qquad R:=\lambda_1(t) r,
    \]
    which in turn leads us to study the equation \begin{align}\label{linin2}
    & \bigg( - \Big(\partial_{\tau_1} +\omega_1 R\partial_R\Big)^2 + \dfrac{1}{4}\Big( \dfrac{\lambda_{1,\tau_1}}{\lambda_1}\Big)^2+\dfrac{1}{2}\partial_{\tau_1} \Big( \dfrac{\lambda_{1,\tau_1}}{\lambda_1}\Big)  \bigg) \wt \varepsilon - \mathcal{H}\wt \varepsilon = \wt f,
    \end{align}
    having denoted by $\wt f:= \lambda_1^{-2} R^{1/2} f$. It is worth noting that \[
    \Vert \wt f\Vert_{L^2_{RdR}} = \dfrac{1}{\lambda_1} \Vert f\Vert_{L^2_{rdr}}.
    \]
    Recall that \begin{align*}
        \mathcal{F}[\wt\varepsilon](\xi) & =  \int_0^\infty \phi(R,\xi) \wt\varepsilon (R) dR + \langle \phi_{0},\wt\varepsilon\rangle_{L^2_{dR}}.
    \end{align*}
    We adopt the same notation as in Section \ref{sec:linout}, namely, \[
    \widehat{\varepsilon} = \mathcal{F}[\wt \varepsilon], \qquad \hbox{or }, \hbox{ according to the last identity}, \qquad \widehat{\varepsilon}= \begin{pmatrix}
        \widehat{\varepsilon}_p
        \\ \widehat{\varepsilon}_c
    \end{pmatrix}.
    \]
    where $\widehat{\varepsilon}_c\in L^2(\R_+,\rho)$. Then, taking the distorted Fourier transform of equation \eqref{linin2},and using Lemma \ref{dft5} we obtain \begin{equation}\label{linin3}\begin{aligned}
\Big( - D_{\tau_1  }^2 - \omega_1 D_{\tau_1  } - \xi \Big) \mathcal{F} \wt \varepsilon & =  \mathcal{F}\widetilde{f} - 2 \omega_2 \mathcal{K}_{nd} D_{\tau_1  } \mathcal{F}\wt \varepsilon  -  \omega_1 [ \mathcal{K}_{nd}, \mathcal{K}_d]  \mathcal{F}\wt \varepsilon 
\\ &   + \omega_1^2 \big( \mathcal{K}_{nd}^2 - \mathcal{K}_{nd}\big)  \mathcal{F}\wt \varepsilon - \dot\omega_1 \mathcal{K}_{nd}  \mathcal{F}\wt \varepsilon ,
\end{aligned}\end{equation}
having denoted by $D_{\tau_1}:=\partial_{\tau_1}-\omega_1\mathcal{K}_d$. Before going any further, let us define the Fourier norms we will work with in the sequel \begin{align*}
    \Vert \widehat{\varepsilon}\Vert_{\boldsymbol{\mathcal{X}}_1}& := \sup_{0 < t \leq t_0} \tau_1^{N_0} \Big( \vert \widehat{\varepsilon}_p(\tau_1)\vert  + \tau_1 \vert \dot{\widehat{\varepsilon}}_p(\tau_1)\vert\Big) + \sup_{0 < t \leq t_0} \tau_1^{N_0} \Vert \rho^{1/2} \widehat{\varepsilon}_c(\tau_1,\cdot)\Vert_{L^2_{d\xi}}
    \\ & \ \,  + \sup_{0 < t \leq t_0} \tau_1^{N_0+1} \Vert \rho^{1/2} D_{\tau_1}\widehat{\varepsilon}_c(\tau_1,\cdot)\Vert_{L^2_{d\xi}} + \sup_{0 < t \leq t_0} \tau_1^{N_0+1} \Vert \rho^{1/2}\xi^{1/2} \widehat{\varepsilon}_c(\tau_1,\cdot)\Vert_{L^2_{d\xi}}.
\end{align*}
The relevance of this choice of norm becomes clear upon applying Plancherel's theorem, which shows that
\begin{align*}
    \sup_{0 < t \leq t_0} \tau_1^{N_0} \Vert \varepsilon\Vert_{\boldsymbol{\mathcal{H}}^1_{rdr}} \lesssim \Vert \mathcal{F}[\wt \varepsilon]\Vert_{\boldsymbol{\mathcal{X}}_1}.
\end{align*}
For the source term, we use the (Fourier) norm \[
\Vert \mathcal{F}[\wt f]\Vert_{\boldsymbol{\mathcal{Y}}_1}:=\sup_{0<t \leq t_0} \tau_1^{N_0+2}\vert f_p(\tau_1)\vert +  \sup_{0<t \leq t_0} \tau_1^{N_0+2}\Vert \rho^{1/2} f_c(\tau_1,\cdot)\Vert_{L^2_{d\xi}}.
\]
We split the analysis into several steps.

\medskip
\noindent
$\bullet$ \textit{Fixed point argument for \eqref{linin3}}: Decomposing equation \eqref{linin3} into its discrete and continuous components, and denoting by  $\widehat{g}_p$ and $\widehat{g}_c$ the projections of the right-hand side of \eqref{linin3} onto the point and continuous spectra, respectively, we are led to study the following system of equations
\begin{equation}\label{linin4}\begin{aligned}
    -\partial_{\tau_1}\Big(\partial_{\tau_1}-\omega_1\Big)\widehat{\varepsilon}_p & = \widehat{g}_p,
    \\ \Big(-D_{\tau_1}^2 - \omega_1 D_{\tau_1} - \xi\Big) \widehat{\varepsilon}_c & = \widehat{g}_c.
\end{aligned}\end{equation}
The first equation in \eqref{linin4} can be solved directly by means of variation of constants. A fundamental system for the first equation is given by
\[
\widehat{\varepsilon}_{p,\text{hom}}= c_1 \lambda_1(\tau_1) + c_2 \lambda_1(\tau_1) \int_{\tau_1}^\infty \dfrac{1}{\lambda_1(t')} dt'.
\]
Then, solving backwards from infinity, one concludes that the first equation in \eqref{linin4} admits a solution \[
\widehat{\varepsilon}_p = - \lambda_1(\tau_1) \int_{\tau_1}^\infty \Big( \int_{\tau_1}^{\sigma_1} \dfrac{ds}{\lambda_1(s)}  \Big) \, \widehat{g}_p(\sigma_1)d\sigma_1. 
\]
Moreover, recalling \eqref{linin5} we then conclude  that \begin{align}\label{linin7}
    \sup_{0 < t\leq t_0} \tau_1^{N_0} \Big( \vert \widehat{\varepsilon}_p(\tau_1)\vert + \tau_1 \vert \dot{\widehat{\varepsilon}}_p(\tau_1)\vert \Big)  &  \lesssim \sup_{0 < t\leq t_0} \tau_1^{N_0+1} \int_{\tau_1}^\infty \log\Big( \dfrac{\sigma_1}{\tau_1}\Big) \big\vert \widehat{g}_p (\sigma_1)\big\vert d\sigma_1  \nonumber
    \\ & \lesssim \dfrac{1}{(N_0+1)^2} \sup_{0 < t \leq t_0} \tau_1^{N_0+2} \big\vert \widehat{g}_p(\tau_1) \big\vert , 
\end{align}
having used that \begin{align*}
\tau_1^{N_0+1}\int_{\tau_1}^\infty \log\Big(\dfrac{\sigma_1}{\tau_1}\Big) \sigma_1^{-N_0-2} d\sigma_1 = \int_1^\infty \log(s) s^{-N_0-2}ds=\dfrac{1}{(N_0+1)^2}.
\end{align*}
In the case of the second equation in \eqref{linin4} we proceed similarly as for the second equation of \eqref{linout4}. Concretely, we have that the second equation in \eqref{linin4} admits a solution of the form \begin{align}\label{linin6}
\widehat{\varepsilon}_c(\tau_1,\xi)=\int_{\tau_1}^\infty U(\tau_1,\sigma_1,\xi) \widehat{g}_c\Big(\tau_1,  \dfrac{\lambda_1^2(\tau_1)}{\lambda_1^2(\sigma_1)}\xi\Big) d\sigma_1
\end{align}
where the Green function is given by  \begin{align*}
    U(\tau_1,\sigma_1,\xi) := \xi^{-1/2} \dfrac{\rho^{1/2}\Big(\tfrac{\lambda_1^2(\tau_1)}{\lambda_1^2(\sigma_1)}\xi\Big) }{\rho^{1/2}(\xi\big)} \, \dfrac{\lambda_1(\tau_1)}{\lambda_1(\sigma_1)} \sin \bigg( \xi^{1/2} \lambda_1(\tau_1) \int_{\tau_1}^{\sigma_1} \lambda_1^{-1}(s)ds\bigg) .
\end{align*}
Moreover, from \eqref{linin5} we deduce that 
\begin{align*}
    & \big\vert U(\tau_1,\sigma_1,\xi)\big\vert  \lesssim  \dfrac{\rho^{1/2}\Big(\tfrac{\lambda_1^2(\tau_1)}{\lambda_1^2(\sigma_1)}\xi\Big) }{\rho^{1/2}( \xi\big)} \, \dfrac{\lambda_1(\tau_1)}{\lambda_1(\sigma_1)}  \, \min\Big\{ \xi^{-1/2}, \, \tau_1 \ln\Big( \dfrac{\sigma_1}{\tau_1} \Big) \Big\} , 
    \\ & \big\vert \partial_{\tau_1} U(\tau_1,\sigma_1,\xi)\big\vert  \lesssim \dfrac{\rho^{1/2}\Big(\tfrac{\lambda_1^2(\tau_1)}{\lambda_1^2(\sigma_1)}\xi\Big) }{\rho^{1/2}( \xi\big)}   \,  \ln\Big( \dfrac{\sigma_1}{\tau_1} \Big) . 
\end{align*}
It immediately follows from the above bounds and the explicit formula \eqref{linin6} that 
\begin{equation}\label{linin8}\begin{aligned}
    & \sup_{0 < t \leq t_0} \tau_1^{N_0} \Vert \rho^{1/2} \widehat{\varepsilon}_c(\tau_1,\cdot)\Vert_{L^2_{d\xi}}  +  \sup_{0 < t \leq t_0} \tau_1^{N_0+1} \Vert \rho^{1/2} D_{\tau_1} \widehat{\varepsilon}_c(\tau_1,\cdot)\Vert_{L^2_{d\xi}}
    \\ & \sup_{0 < t \leq t_0} \tau_1^{N_0} \Vert \rho^{1/2} \xi^{1/2} \widehat{\varepsilon}_c(\tau_1,\cdot)\Vert_{L^2_{d\xi}} \, \lesssim \, \dfrac{1}{N_0}  \sup_{0 < t \leq t_0} \tau_1^{N_0+2} \Vert \rho^{1/2} \widehat{g}_c(\tau_1,\cdot)\Vert_{L^2_{d\xi}}. 
\end{aligned}\end{equation}
Combining \eqref{linin7} and \eqref{linin8} we conclude that \[
\Vert \widehat{\varepsilon}\Vert_{\boldsymbol{\mathcal{X}}_1} \lesssim \dfrac{1}{N_0} \Vert g\Vert_{\boldsymbol{\mathcal{Y}}_1}.
\]
With these estimates in hand, to solve \eqref{linin3} it suffices to now notice that, according to Proposition \ref{prop:dft1}, \begin{align*}
    &\dfrac{1}{N_0}\bigg(\Vert\omega_2 \mathcal{K}_{nd} D_{\tau_1  } \mathcal{F}\wt \varepsilon \Vert_{\boldsymbol{\mathcal{Y}}_1}  + \Vert  \omega_1 [ \mathcal{K}_{nd}, \mathcal{K}_d]  \mathcal{F}\wt \varepsilon  \Vert_{\boldsymbol{\mathcal{Y}}_1}
\\ &   + \Vert \omega_1^2 \big( \mathcal{K}_{nd}^2 - \mathcal{K}_{nd}\big)  \mathcal{F}\wt \varepsilon  \Vert_{\boldsymbol{\mathcal{Y}}_1}  +  \Vert \dot\omega_1 \mathcal{K}_{nd}  \mathcal{F}\wt \varepsilon  \Vert_{\boldsymbol{\mathcal{Y}}_1} \bigg)  \lesssim  \dfrac{1}{N_0} \Vert \mathcal{F}[\wt \varepsilon]  \Vert_{\boldsymbol{\mathcal{X}}_1}.
\end{align*}
Then, taking $N_0\gg1$ sufficiently large, we obtain the smallness required to close the fixed-point argument associated with \eqref{linin3}, which yields the existence of a solution satisfying the desired bound corresponding to the first norm in \eqref{linin9}. The remainder of the proof is devoted to controlling the remaining four norms in \eqref{linin9}. 

\medskip
\noindent
$\bullet$ \textit{Control of the norms in \eqref{linin9} depending on the scaling vector field $S$}: We apply $S=t\partial_t+r\partial_r$ to equation \eqref{linin1} and use the bounds obtained in the previous step. Indeed, we have that \[
-\partial_t^2S\varepsilon+\partial_r^2S\varepsilon+\dfrac{1}{r}\partial_r S \varepsilon - \dfrac{4}{r^2}\cos(2Q_1)S\varepsilon = Sf+2f+ \dfrac{4}{r^2}\Big(S\cos(2Q_1)\Big)\varepsilon.
\]
Then, noticing that \[
\bigg\vert \dfrac{4}{r^2}S\cos(2Q_1)\bigg\vert \lesssim \dfrac{\lambda_1^3r^2(\lambda_1+t\vert \lambda_1'\vert)}{\langle \lambda_1 r\rangle^8} \lesssim t \lambda_1\vert\lambda_1'\vert  \lesssim \tau_1^{0+}\lambda_1^2,
\]
we infer that \begin{align*}
    \lambda_1^{-1} \Big\Vert \dfrac{4}{r^2}\Big( S\cos(2Q_1)\Big) \varepsilon\Big\Vert_{L^2_{rdr}} \lesssim \tau_1^{0+} \lambda_1 \Vert \varepsilon\Vert_{L^2_{rdr}},
\end{align*}
which in turn implies (recall that $N_1=N_0/2\gg1$, and the definition of the physical norms $X_1$ and $Y_1$ in \eqref{linin9}-\eqref{linin10}) that \begin{align*}
    & \sup_{0 < t \leq t_0 } \tau_1^{N_1+2} \lambda_1^{-1}\Big\Vert \dfrac{4}{r^2}\Big( S\cos(2Q_1)\Big) \varepsilon\Big\Vert_{L^2_{rdr}} 
    \\ & \ll \sup_{0 < t\leq t_0} \tau_1^{(N_1+2)+}\Big(\lambda_1(t) \Vert \varepsilon(t,\cdot)\Vert_{L^2_{rdr}} \Big) \lesssim \Vert \varepsilon\Vert_{X_1},
\end{align*}
provided $N_0>N_1+2$, which holds for any $N_0>4$. Denoting by $\Vert\cdot\Vert_{\widetilde{\boldsymbol{\mathcal{X}}}_1^j}$ and $\Vert\cdot\Vert_{\widetilde{\boldsymbol{\mathcal{Y}}}_1^j}$ the Fourier norms defined in the previous step, with $N_0$ replaced by $N_j$, we then conclude (by applying the bounds obtained in the previous step) that \[
\Vert \mathcal{F}[S\varepsilon]\Vert_{\widetilde{\boldsymbol{\mathcal{X}}}_1^1} \lesssim \Vert \mathcal{F}[\widetilde{Sf}]\Vert_{\widetilde{\boldsymbol{\mathcal{Y}}}_1^1} + \Vert \mathcal{F}[\wt f]\Vert_{\boldsymbol{\mathcal{Y}}_1} .
\]
The bounds for $S^2\varepsilon$ follow identical lines, so we omit them.

\medskip
\noindent
$\bullet$ \textit{Control of the singular norms in \eqref{linin9} depending on $r^{-2}$}: This follows from a direct estimate. No Duhamel formulation or smallness is needed here (in contrast with the previous steps). We prove this only for the case $\ell=0$. The general case $\ell \geq1$ then follows by differentiating the equation as before. We first recall the local elliptic estimate near $r=0$ in \cite{KST2}, namely, \begin{align*}
\Vert r^{-2}\varepsilon\Vert_{L^2_{rdr}} \lesssim \Big\Vert \dfrac{1}{t} \partial_r\varepsilon\Big\Vert_{L^2_{rdr}}+\Big\Vert \dfrac{1}{t r}\varepsilon\Big\Vert_{L^2_{rdr}}+\Big\Vert \Big( \dfrac{t^2-r^2}{t^2}\partial_r^2+\dfrac{1}{r}\partial_r-\dfrac{4}{r^2}\Big)\varepsilon\Big\Vert_{L^2_{rdr}}.
\end{align*}
We rewrite the last operator above using the scaling vector field \begin{align*}
\bigg( \dfrac{t^2-r^2}{t^2} \partial_r^2 + \dfrac{1}{r}\partial_r  - \dfrac{4}{r^2} \bigg) \varepsilon & = \dfrac{1}{t^2}\bigg( -2t \partial_t \varepsilon - S^2 \varepsilon  + S\varepsilon  + 2 t \partial_t S \varepsilon \bigg)
\\ & + \bigg(-\partial_t^2+\partial_r^2+\dfrac{1}{r}\partial_r - \dfrac{4}{r^2}\bigg) \varepsilon.
\end{align*}
From this last identity it follows that \begin{align*}
    \dfrac{1}{\lambda_1}\Big\Vert \Big( \dfrac{t^2-r^2}{t^2} \partial_r^2 + \dfrac{1}{r}\partial_r  - \dfrac{4}{r^2}\Big) \varepsilon\Big\Vert_{L^2_{rdr}} & \lesssim \dfrac{1}{t^2\lambda_1^2} \Big( \lambda_1 \Vert S^2\varepsilon\Vert_{L^2_{rdr}} + \lambda_1 \Vert S\varepsilon\Vert_{L^2_{rdr}} \Big)
    \\ & + \dfrac{1}{t \lambda_1}\Big( \Vert \partial_t S \varepsilon\Vert_{L^2_{rdr}} + \Vert \partial_t \varepsilon\Vert_{L^2_{rdr}}\Big)
    \\ & + \dfrac{1}{\lambda_1}\Big\Vert \Box \varepsilon - \dfrac{4}{r^2}\varepsilon\Big\Vert_{L^2_{rdr}},
\end{align*}
where we recall that $\Box \varepsilon = -\partial_t^2\varepsilon+\partial_r^2\varepsilon+\tfrac1r\partial_r\varepsilon$. By direct calculations,
\begin{align*}
    \Big\vert \dfrac{4}{r^2}\Big(\cos(2Q_1)-1\Big) \Big\vert \lesssim \dfrac{\lambda_1^4r^2}{\langle \lambda_1 r\rangle^8} \lesssim \lambda_1^2.
\end{align*}
from which we infer that \[
\dfrac{1}{\lambda_1}\Big\Vert \Box \varepsilon - \dfrac{4}{r^2}\varepsilon\Big\Vert_{L^2_{rdr}} \lesssim \dfrac{1}{\lambda_1}\Vert f \Vert_{L^2_{rdr}}+ \lambda_1\Vert \varepsilon\Vert_{L^2_{rdr}} ,
\]
which in turn (along with the estimates from the previous steps) implies that \[
\sup_{0 < t \leq t_0} \tau_1^{N_2} \lambda_1^{-1} \Vert \dfrac{1}{r^2}\varepsilon(t,\cdot)\Vert_{L^2_{rdr}} \lesssim \Vert f\Vert_{Y_1}.
\]
The $L^4$ estimate follows from interpolation between the last bound and \[
\Vert S^2 \varepsilon\Vert_{L^\infty} \lesssim \Vert \nabla S^2\varepsilon\Vert_{L^2_{rdr}} + \Big\Vert \dfrac{1}{r}S^2 \varepsilon \Big\Vert_{L^2_{rdr}}.
\]
The proof is complete.
\end{proof}

%\newpage
\bigskip
\section{Construction of the exact solution}

It remains only to construct the final correction $\varepsilon(t,r)$ to complete the approximate solution $u_N$ and turn it into an exact solution
\begin{align*}
    u(t,r):=u_N(t,r)+\varepsilon(t,r)
\end{align*}
of equation \eqref{eq:keq2corotational} in the cone $r\lesssim t$, where $u_N$ is given by Proposition \ref{prop:approximateinnerbubble}. The equation for $\varepsilon(t,r)$ is then (recall equation \eqref{eq:step2_8} and  \eqref{eq:step2_6})
\begin{align}\label{eq:exact1}
    -\partial_t^2\varepsilon + \partial_r^2 \varepsilon + \dfrac{1}{r}\partial_r \varepsilon - \dfrac{4}{r^2}\cos\big(2Q_1-2\bfcq_{n-1}\big)\varepsilon= - e_N + \sum_{k=1}^3 E_k
\end{align}
where the terms $E_k$ are given by
\begin{equation}\label{eq:exact3}\begin{aligned}
    E_1(\varepsilon) &:= \dfrac{2}{r^2}\cos(2Q_1-2\bfcq_{n-1} + 2v_N)  \Big(\sin(2\varepsilon)-2\varepsilon\Big),
    \\ E_2(\varepsilon) &:=  \dfrac{4}{r^2}\Big(\cos(2Q_1-2\bfcq_{n-1}+2v_N)-\cos(2Q_1-2\bfcq_{n-1})\Big) \, \varepsilon
    \\ E_3(\varepsilon) &:= \dfrac{2}{r^2}\sin(2Q_1-2\bfcq_{n-1}+2v_N) \Big(\cos(2\varepsilon)-1\Big).
\end{aligned}\end{equation}
Here we define
\begin{align*}
    e_N :=\chi_{r\lesssim t}\Big(-\partial_t^2 u_N+ \partial_r^2 u_N + \dfrac{1}{r}\partial_r u_N - \dfrac{2}{r^2}\sin(2u_N)\Big).
\end{align*}

\begin{prop}
    There exists $N\gg1$ sufficiently large, and $t_0=t_0(\beta, N)>0$ small enough so that equation \eqref{eq:exact1} admits a solution \[
    \varepsilon\in C^0\big( (0,t_0]; \, \bch^1_{rdr}\big),
    \]
    satisfying the bound \[
    \sum_{k=0}^2 \big\Vert S^k \varepsilon(t,r)\big\Vert_{\bch^1_{rdr}} \ll_{N,t_0} \tau^{-N}, \qquad  S:= t\partial_t + r\partial_r.
    \]
    The same bounds obtain if we replace $\varepsilon$ by $S^l\varepsilon$, $l\geq 1$, with implicit constants also depending on $l$. 
\end{prop}

\begin{proof}
First of all, recall the definitions of the $X_1$ and $Y_1$ norms in \eqref{linin9} and \eqref{linin10}, as well as the $\boldsymbol{\mathcal{H}}^1_{rdr}$ norm (at the scale of $\lambda_1$) in \eqref{linin13}. We plan to use a Banach fixed-point argument. The key idea is to exploit the fast decay of the source term (after simplifying the linear operator in equation \eqref{eq:exact1}). In this way, due to the $N^{-1}$ factor, the bound in Lemma \ref{lem:linin_1} will, in particular, provide the required smallness (for $N\gg1$ sufficiently large).  More specifically, we rewrite equation \eqref{eq:exact1} as 
\begin{equation}\label{eq:varepsilonequation}\begin{split}
    & -\partial_t^2\varepsilon + \partial_r^2 \varepsilon + \dfrac{1}{r}\partial_r \varepsilon - \dfrac{4}{r^2}\cos\big(2Q_1\big)\varepsilon
    \\ & = \dfrac{4}{r^2}\Big( \cos(2Q_1-2\bfcq_{n-1})-\cos(2Q_1)\Big)\varepsilon - e_N + \sum_{k=1}^3 E_k.
\end{split}\end{equation}
We seek to show that \[
\Vert g\Vert_{Y_1} \lesssim \Vert \varepsilon\Vert_{X_1}+o_{t_0}(1),
\]
where $g$ is the right-hand side of the above equation, namely, \begin{align}\label{eq:exact5}
g:=\dfrac{4}{r^2}\Big( \cos(2Q_1-2\bfcq_{n-1})-\cos(2Q_1)\Big)\varepsilon -  e_N + \sum_{k=1}^3 E_k.
\end{align}
Indeed, as usual, we begin by rewriting the difference of the trig functions as
\begin{align*}
    \cos(2Q_1-2\bfcq_{n-1})-\cos(2Q_1)  =  2\sin(\bfcq_{n-1})\sin\big(2Q_1-\bfcq_{n-1}\big).
\end{align*}
Recalling the definition of $\tau_1$ in \eqref{linin12}, the asymptotic \eqref{linin11}, and the inductive hypothesis \eqref{eq:intro_wn-1_decay_indhyp}-\eqref{eq:intro-vn-1boundsrefined}, it is clear that
\begin{align*}
    \bigg\vert \dfrac{4}{R^2}\Big(\cos(2Q_1-2\bfcq_{n-1})-\cos(2Q_1) \Big)\bigg\vert \lesssim \dfrac{\lambda_2^2}{\lambda_1^2} \sim \tau_1^{-2},
\end{align*}
having used that \[
\cos(2Q_1-2\bfcq_{n-1})-\cos(2Q_1) =2\sin(2Q_1-\bfcq_{n-1})\sin(\bfcq_{n-1}).
\]
From the above bound it immediately follows that \begin{align*}
    \lambda_1^{-1}\bigg\Vert \dfrac{4}{r^2}\Big(\cos(2Q_1-2\bfcq_{n-1})-\cos(2Q_1) \Big) \varepsilon \bigg\Vert_{L^2_{rdr}}  \lesssim  \tau_1^{-2} \Big(\lambda_1 \Vert \varepsilon\Vert_{L^2_{rdr}}\Big),
\end{align*}
and, more generally,  \begin{equation}\label{eq:exact8}\begin{aligned}
    & \lambda_1^{-1}\bigg\Vert S^\ell\bigg(\dfrac{4}{r^2}\Big(\cos(2Q_1-2\bfcq_{n-1})-\cos(2Q_1) \Big) \varepsilon \bigg) \bigg\Vert_{L^2_{rdr}} 
    \\ & \lesssim  \tau_1^{-2+} \Big(\lambda_1 \sum_{k=0}^\ell \Vert S^k \varepsilon\Vert_{L^2_{rdr}}\Big).
\end{aligned}\end{equation}
Next, concerning the nonlinear terms $E_k$, $k=1,2,3$, the only delicate case is $E_2$, since it is linear in $\varepsilon$. The other two terms can be bounded directly, using the additional factors of $\varepsilon$ (together with the large power of $\tau$ in the $X$ norm) to recover the two powers of $\tau$. In fact, taking advantage of the singular terms in the $X_1$ norm we obtain that \begin{align}\label{eq:exact2}
    \lambda_1^{-1}\Vert S^\ell E_1\Vert_{L^2_{rdr}} \lesssim \lambda_1^{-1}\Big\Vert \dfrac{\varepsilon}{r^2}\Big\Vert_{L^2_{rdr}} \, \sum_{k=0}^\ell \Vert S^k \varepsilon\Vert_{L^\infty}^2, \quad \ell\leq 2.
\end{align}
In the case of $E_2$, we rewrite said expression as \[
E_2= 4 \bigg( \cos(2Q_1-2\bfcq_{n-1}) \dfrac{\cos(2v_N)-1}{r^2}-\sin(2Q_1-2\bfcq_{n-1})\dfrac{\sin(2v_N)}{r^2}\bigg)\varepsilon.
\]
Then, recalling from Lemma \ref{lem:hzerocorrection}, Lemma \ref{lem:step1_2}, and Lemma \ref{lem:step2_1}, that \[
\Big\vert \dfrac{1}{\langle R\rangle} v_N \Big\vert \lesssim \tau_1^{-2}, \quad \, \vert v_N\vert \lesssim \tau_1^{-2+} \, \quad \hbox{ and } \, \quad 
\big\vert \sin(2Q_1) \big\vert = \dfrac{R^2(R^4-1)}{(R^4+1)^2}\lesssim \dfrac{1}{\langle R\rangle^2},
\]
which, along with the inductive hypothesis \eqref{eq:intro_wn-1_decay_indhyp}-\eqref{eq:intro-vn-1boundsrefined}, then yields \begin{align*}
    \bigg\vert \dfrac{1}{r^2} \sin(2Q_1-2\bfcq_{n-1}) \sin(2v_N) \bigg\vert & \lesssim \dfrac{1}{r^2} \bigg( \big\vert \sin(2Q_1)\big\vert + \big\vert \sin(2\bfcq_{n-1})\big\vert \bigg)  \big\vert \sin(2v_N)\big\vert
    \\ & \lesssim \bigg(\dfrac{1}{r^2}\, \dfrac{1}{\langle R\rangle^2} + \dfrac{\vert \bfcq_{n-1}\vert}{r^2}\bigg) \big\vert \sin(2v_N)\big\vert
    \\ & \lesssim \dfrac{1}{r^2} \, \dfrac{1}{\langle R \rangle} \, \tau_1^{-2} + \lambda_2^2 \tau_1^{-2+}.
\end{align*}
We conclude that \begin{align*}
    & \lambda_1^{-1} \Big\Vert \dfrac{1}{r^2}\Big(\sin(2Q_1-2\bfcq_{n-1})\sin(2v_N)\Big) \varepsilon\Big\Vert_{L^2_{rdr}}
    \\ & \lesssim \tau_1^{-2}\Big(\lambda_1^{-1}  \Big\Vert \dfrac{1}{r^2}\varepsilon\Big\Vert_{L^2_{rdr}}\Big) + \lambda_2^2\lambda_1^{-1} \tau_1^{-2+} \Vert \varepsilon\Vert_{L^2_{rdr}} .
\end{align*}
Similarly, using Lemma \ref{lem:linout_1} for the singular bounds for the outer wave part of $v_N$ and \eqref{eq:step1_16} and \eqref{eq:step2_10} for the inner elliptic part, we infer that \[
\Big\Vert \dfrac{1}{r^2} \Big(\cos(2v_N)-1\Big)\Big\Vert_{L^2_{rdr}} \lesssim \Vert v_N \Vert_{L^\infty} \Big\Vert \dfrac{1}{r^2}v_N\Big\Vert_{L^2_{rdr}} \lesssim \tau_1^{-4+}.
\]
and hence, \[
\lambda_1^{-1} \Vert E_2\Vert_{L^2_{rdr}} \lesssim \tau_1^{-2}\Big(\lambda_1^{-1}  \Big\Vert \dfrac{1}{r^2}\varepsilon\Big\Vert_{L^2_{rdr}}\Big) + \lambda_2^2\lambda_1^{-1} \tau_1^{-2+} \Vert \varepsilon\Vert_{L^2_{rdr}} + \tau^{-5+}\Vert \varepsilon\Vert_{L^\infty}.
\]
On the other hand, due to the estimates in \eqref{eq:step2_4}, we also have the differentiated bounds \[
\lambda_1^{-1} \Vert S^\ell E_2\Vert_{L^2_{rdr}} \lesssim \tau_1^{-2} \, \sum_{k=0}^\ell \Vert S^k \varepsilon\Vert_{L^\infty}, \quad \ \ell\leq 2.
\]
Then, using the elementary bound, \begin{align}\label{eq:exact4}
    \Vert S^k\varepsilon\Vert_{L^\infty} \lesssim \Vert \nabla S^k \varepsilon\Vert_{L^2_{rdr}} + \Big\Vert \dfrac{1}{r}S^k\varepsilon\Big\Vert_{L^2_{rdr}},
\end{align}
as well as \eqref{linout17} and \eqref{linin13}, we obtain that \begin{align}\label{eq:exact6}
    \lambda_1^{-1} \Vert S^\ell E_2(t,\cdot)\Vert_{L^2_{rdr}} \lesssim \tau_1^{-2} \sum_{k=0}^\ell \Vert S^k\varepsilon\Vert_{\boldsymbol{\mathcal{H}}^1_{rdr}}, \quad \ \ell\leq 2.
\end{align}
Similarly, the bound \eqref{eq:exact2} leads to \begin{align}\label{eq:exact7}
    \lambda_1^{-1}\Vert S^\ell E_1\Vert_{L^2_{rdr}} \lesssim \lambda_1^{-1}\Big\Vert \dfrac{\varepsilon}{r^2}\Big\Vert_{L^2_{rdr}} \, \sum_{k=0}^\ell \Vert S^k \varepsilon\Vert_{\bch^1_{rdr}}^{2}, \quad \ell\leq 2.
\end{align}
Finally, it only remains to bound the nonlinear term $E_3$ in \eqref{eq:exact3}. This case is similar to $E_1$. Indeed, we have that (here $\ell\leq2$) \begin{align*}
    \lambda_1^{-1} \Big\Vert  S^\ell\Big(\dfrac{1}{r^2}\big(\cos(2\varepsilon)-1\big)\Big)\Big\Vert_{L^2_{rdr}} \lesssim \lambda_1^{-1} \Big\Vert \dfrac{1}{r^2}\varepsilon\Big\Vert_{L^2_{rdr}}\, \sum_{k=0}^2 \Vert S^k\varepsilon\Vert_{L^\infty} + \lambda_1^{-1} \Big\Vert \dfrac{1}{r}S\varepsilon\Big\Vert_{L^4_{rdr}}^2.
\end{align*}  
Then, using the brutal bound (here $\ell\leq 2$) \begin{align*}
    \big\Vert S^\ell\cos(2Q_1-2\bfcq_{n-1} + 2 v_N)\big\Vert_{L^\infty} \lesssim_\ell 1, 
\end{align*}
along with Leibniz rule and H\"older inequality, we get that 
\begin{align*}
    \lambda_1^{-1}\Vert S^\ell E_3\Vert_{L^2_{rdr}} \lesssim \lambda_1^{-1} \Big\Vert \dfrac{1}{r^2}\varepsilon\Big\Vert_{L^2_{rdr}} \, \sum_{k=0}^2 \Vert S^k\varepsilon\Vert_{L^\infty} + \lambda_1^{-1} \Big\Vert \dfrac{1}{r}S \varepsilon\Big\Vert_{L^4_{rdr}}^2,
\end{align*}
which, along with \eqref{eq:exact4}, in turn implies \begin{align}\label{eq:exact9}
    \lambda_1^{-1}\Vert S^\ell E_3\Vert_{L^2_{rdr}} \lesssim \lambda_1^{-1} \Big\Vert \dfrac{1}{r^2}\varepsilon\Big\Vert_{L^2_{rdr}} \, \sum_{k=0}^2 \Vert S^k\varepsilon\Vert_{\boldsymbol{\mathcal{H}}^1_{rdr}} + \lambda_1^{-1} \Big\Vert \dfrac{1}{r}S \varepsilon\Big\Vert_{L^4_{rdr}}^2, \quad \ \ell\leq2.
\end{align}
Thus, gathering estimates \eqref{eq:exact8}, \eqref{eq:exact6}, \eqref{eq:exact7}, along with \eqref{eq:exact9}, recalling  the definitions of the $X_1$ and $Y_1$ norms in \eqref{linin9} and \eqref{linin10}, the $\boldsymbol{\mathcal{H}}^1_{rdr}$ norm in \eqref{linin13}, as well as the definition of $g$ in \eqref{eq:exact5}, we infer that \[
\Vert g\Vert_{Y_1} \lesssim \Vert \varepsilon\Vert_{X_1} + o_{t_0}(1).
\]
Then, the proof follows by Lemma \ref{lem:linin_1}, and a standard Banach iteration, provided we choose $N$ sufficiently large. 
\\
The bounds for $S^l\varepsilon$ are obtained by applying $S^l$ to the equation and using induction on $l$
The proof is complete.
\end{proof}

In order to prepare the insertion of the next inner bubble, we still need to confirm the Taylor type expansion for $\varepsilon$, as in \eqref{eq:intro-vn-1refined}, \eqref{eq:intro-vn-1boundsrefined}. For this we have 
\begin{lem}\label{lem:varepsfinestructure} The solution $\varepsilon$ described in the preceding lemma allows the representation
\[
\varepsilon(t, r) = c(t)r^2 + g(t, r),
\]
where we have the bounds 
\[
\big|(t\partial_t)^kc(t)\big|\lesssim_k \tau^{-N+1},\,\big|r^{-4}S^kg(t,r)\big|\lesssim_k\tau^{-N+2}. 
\]
\end{lem}
\begin{proof} This essentially mimics the proof of Lemma~\ref{lem:Tayloerexpansion}, the main difference being that we need to also to control the nonlinear source terms. 
\\
{\it{(1): We have the bound 
\begin{equation}\label{eq:Htwobound10}
\big\|\varepsilon_{rr}\big\|_{L^2_{r\,dr}(r\ll t)} + \big\|\frac{\varepsilon_r}{r}\big\|_{L^2_{r\,dr}}\lesssim \tau^{-N+},
\end{equation}
and similarly for $S^k\varepsilon$.}} We take advantage of \eqref{eq:trickidentity1}, with $\epsilon$ replaced by $\varepsilon$. Thanks to the preceding lemma, we see that as in the proof of Lemma~\ref{lem:Tayloerexpansion}, it suffices to bound $\big\|(I), (II), (III)\big\|_{L^2_{r\,dr}}$. The bounds 
\[
\big\|(I)\big\|_{L^2_{r\,dr}} + \big\|(III)\big\|_{L^2_{r\,dr}}\lesssim \tau^{-N+}
\]
is then a direct consequence of the preceding lemma. As for the term $(II)$, we take advantage of 
\begin{align*}
&\Big\|\chi_{r\ll t}\big[ \Box\varepsilon - \frac{4\varepsilon}{r^2}\big]\Big\|_{L^2_{r\,dr}}\\&\leq \Big\|\chi_{r\ll t}\big[ \Box\varepsilon - \frac{4\cos(\bfcq_{n-1})\varepsilon}{r^2}\big]\Big\|_{L^2_{r\,dr}} + \Big\|\chi_{r\ll t}\frac{4(\cos(\bfcq_{n-1})-1)\varepsilon}{r^2}\Big\|_{L^2_{r\,dr}}. 
\end{align*}
The last two terms can be bounded by $\lesssim \tau^{-N+}$ thanks to the preceding lemma and its proof. 
\\

{\it{(2): The estimate $\big|c(t)\big|\lesssim \tau^{-N+1}$.}} Write
\begin{equation}\label{eq:epsilonformula0}
\chi_{r\ll t}\varepsilon = c(t)r^2 +  \frac14 r^2\int_0^r s^{-1}\varphi(s)\,ds - \frac14 r^{-2}\int_0^r s^{3}\varphi(s)\,ds, 
\end{equation}
where $\varphi$ is the sum of the terms on the right hand side in \eqref{eq:trickidentity1}, with $\epsilon$ replaced by $\varepsilon$, as well as the term $\frac{r^2}{t^2}\big(\chi_{r\ll t}\varepsilon\big)_{rr}$. As in the proof of Lemma~\ref{lem:Tayloerexpansion}, it suffices to show the estimate 
\begin{align*}
\big|\frac14 \int_0^r s^{-1}\varphi(s)\,ds - \frac14 r^{-4}\int_0^r s^{3}\varphi(s)\,ds\big|\lesssim \tau^{-N+1}.     
\end{align*}
In light of {it{(1)}} and the preceding lemma, as well as Proposition~\ref{prop:approximateinnerbubble} and the equation for $\varepsilon$, \eqref{eq:varepsilonequation}, the the only new aspect comes from the terms $E_k,\,k = 1, 2, 3$, given explicitly in \eqref{eq:exact3}. We claim that these satisfy the estimates 
\begin{equation}\label{eq:E_kwithsingularweight}
\big\|r^{-2}E_k\big\|_{L^1_{r\,dr}}\lesssim \tau^{-N+1},\,k = 1, 2, 3.  
\end{equation}
Once this is done, we can infer the bounds 
\begin{align*}
&\big|\frac14 \int_0^r s^{-1}E_k(s)\,ds - \frac14 r^{-4}\int_0^r s^{3}E_k(s)\,ds\big|\\
&\lesssim \big\|r^{-2}E_k\big\|_{L^1_{r\,dr}}\lesssim \tau^{-N+1},
\end{align*}
as desired. To derive the bound \eqref{eq:E_kwithsingularweight}, we note that 
\begin{align*}
\big\|r^{-2}E_1\big\|_{L^2_{r\,dr}}\lesssim \big\|r^{-2}\varepsilon\big\|_{L^2_{r\,dr}}^2\cdot \big\|\varepsilon\big\|_{L^\infty}\lesssim \tau^{-3N+}, 
\end{align*}
where we have taken advantage of the preceding lemma and its proof, as well as Holder's inequality. 
\\
Next, we have 
\begin{align*}
\big\|r^{-2}E_2\big\|_{L^2_{r\,dr}}\lesssim \big\|\frac{v_N}{r^2}\big\|_{L^2_{r\,dr}}\big\|r^{-2}\varepsilon\big\|_{L^2_{r\,dr}}\lesssim \tau^{-N+1}, 
\end{align*}
where we have taken advantage of Proposition~\ref{prop:approximateinnerbubble}, Holder's inequality, and the preceding lemma and its proof. 
\\
The term $E_3$ can be handled like the term $E_1$. 
\\

{\it{(3): We have the bounds $\big\|\frac{S^k\varepsilon}{r^2}\big\|_{L^\infty_{r\,dr}}\lesssim_k \tau^{-N+1}$.}} This follows from the preceding step for $k = 0$, and by repeating the precedign analysis for $S^k\varepsilon$, we derive the corresponding estimate for the differentiated function.
\\

{\it{(4): We have the bound 
\begin{equation}\label{eq:sharpgbound}
\big|\frac14 \int_0^r s^{-1}\varphi(s)\,ds - \frac14 r^{-4}\int_0^r s^{3}\varphi(s)\,ds\big|\lesssim \tau^{-N+2} r^2,\,r\lesssim t. 
\end{equation}
}} Once this has been proven, the formula \eqref{eq:epsilonformula0} then implies the conclusion of the lemma for $k = 0$, and the derivative bounds follow by following the same reasoning for $S^k\varepsilon$. The only new feature in the proof of \eqref{eq:sharpgbound} compared to the argument for {\it{(5)}} in the proof of Lemma~\ref{lem:Tayloerexpansion} concerns the required bounds for the contributions of the terms $E_k, k = 1,2, 3$, as in {\it{(2)}}. These bounds follow from {\it{(3)}}. In fact, we claim that 
\begin{align*}
\big|r^{-2}E_k\big|\lesssim \tau^{-N+2},\,k = 1, 2, 3. 
\end{align*}
This the easily implies \eqref{eq:sharpgbound}. To establish the claimed estimate, we first note that (recall \eqref{eq:exact3})
\begin{align*}
\big|r^{-2}E_1\big|\lesssim \big\|r^{-2}\varepsilon\big\|_{L^\infty_{r\,dr}}^2\cdot \big\|\varepsilon\big\|_{L^\infty_{r\,dr}}\lesssim \tau^{-3N+2},
\end{align*}
where we have taken advantage of the preceding lemma, as well as {\it{(3)}}. 
\\
Next, we see that 
\begin{align*}
\big|r^{-2}E_2\big|\lesssim \big\|r^{-2}v_N\big\|_{L^\infty_{r\,dr}}\cdot\big\|r^{-2}\varepsilon\big\|_{L^\infty_{r\,dr}}\lesssim \tau^{-N+2},      
\end{align*}
taking advantage of Proposition~\ref{prop:approximateinnerbubble} and {\it{(3)}}. For the remaining term $E_3$, we have the estimate 
\begin{align*}
\big|r^{-2}E_3\big|\lesssim \big\|r^{-2}\varepsilon\big\|_{L^\infty_{r\,dr}}^2\lesssim \tau^{-2N+2}. 
\end{align*}
\end{proof}

The proof of Theorem \ref{thm:Main} is then a direct consequence of the preceding proposition, together with \eqref{linin11}, as well as Huyghen's principle. Indeed, the function $u_N+\varepsilon$ solves equation \eqref{eq:keq2corotational} in the light cone 
\[
\mathcal{A}:= \{r\leq t\} \cap \{t\in (0,t_0]\}.
\]
Then, letting $u(t,r)$ be the solution of \eqref{eq:keq2corotational} with initial data $(u_N+\varepsilon)(t_0)$ at time $t=t_0$, we get that \[
u\Big\vert_{r\leq t} = (u_N+\varepsilon)\Big\vert_{r\leq t}
\]
for $t\in(0,t_0]$, due to Huyghen's principle. Finally, note that the conservation of energy implies that $u$ cannot develop a singularity outside the light cone $r\leq t$, and so the first singularity of $u$ on $[0,t_0]\times \R_+$ occurs indeed at time $t=0$, as described by $u_N+\varepsilon$, resulting in a tower of $n$ bubbles concentrating simultaneously but at different rates at the origin.

%%%%%%%%%%%%%%%%%%%%%%%%%%%%%%%%%%%%%%%%%%%%
%%%%%%%%%%%%%%%%%%%%%%%%%%%%%%%%%%%%%%%%%%%%
%%%%%%%%%%%%%%%%%%%%%%%%%%%%%%%%%%%%%%%%%%%%
%%%%%%%%%%%%%%%%%%%%%%%%%%%%%%%%%%%%%%%%%%%%
%%%%%%%%%%%%%%%%%%%%%%%%%%%%%%%%%%%%%%%%%%%%

\end{document}